\documentclass[reqno,11pt]{amsart}
\usepackage{amsmath,amssymb,amsthm,graphicx,mathrsfs,bbm,url}
\usepackage{amsthm}
\usepackage{wrapfig}
\usepackage{enumitem}
\usepackage{mathtools}
\usepackage[utf8]{inputenc}
 \usepackage[T1]{fontenc}
\usepackage[usenames,dvipsnames]{color}
\usepackage[colorlinks=true,linkcolor=Red,citecolor=Green]{hyperref}
\usepackage[super]{nth}
\usepackage[open, openlevel=2, depth=3, atend]{bookmark}
\hypersetup{pdfstartview=XYZ}
\usepackage[font=footnotesize]{caption}
\usepackage{a4wide}

\usepackage{epstopdf}
 
\usepackage{hyperref}

\theoremstyle{plain}
\newtheorem{theorem}{Theorem}[section]
\newtheorem*{theorem*}{Theorem}

\newtheorem{lemma}[theorem]{Lemma}
\newtheorem{proposition}[theorem]{Proposition}
\newtheorem{corollary}[theorem]{Corollary}

\newtheorem{conjecture}[theorem]{Conjecture}

\theoremstyle{definition}
\newtheorem{definition}[theorem]{Definition}

\theoremstyle{remark}
\newtheorem{remark}[theorem]{Remark}

\numberwithin{equation}{section}

\newcommand{\mc}{\mathcal}
\newcommand{\rr}{\mathbb{R}}

\newcommand{\cc}{\mathbb{C}}

\newcommand{\la}{\lambda}

\newcommand{\pl}{\partial}
\newcommand{\x}{\times}

\newcommand{\til}{\widetilde}

\newcommand{\cjd}{\rangle}
\newcommand{\cjg}{\langle}

\newcommand{\C}{\mathbb{C}}
\newcommand{\R}{\mathbb{R}}

\newcommand{\N}{\mathbb{N}}

\newcommand{\Ss}{\mathbb{S}}
\newcommand{\eps}{\varepsilon}

\newcommand{\wt}{\widetilde}

\DeclareMathOperator{\vol}{vol}

\DeclareMathOperator{\Tr}{Tr}
\DeclareMathOperator{\Op}{Op}

\DeclareMathOperator{\Var}{Var}
\DeclareMathOperator{\ran}{ran}

\DeclareMathOperator{\Diff}{Diff}

\newcommand{\be}{\begin{equation}}
\newcommand{\ee}{\end{equation}}

% ...

% ...

\title
[Geodesic stretch, pressure metric and marked length spectrum rigidity]
{Geodesic stretch, pressure metric and marked length spectrum rigidity}

\author{Colin Guillarmou}
\address{Laboratoire de Mathématiques d’Orsay, Univ. Paris-Sud, CNRS, Université Paris-Saclay, 91405 Orsay, France}
\email{colin.guillarmou@math.u-psud.fr}

\author{Gerhard Knieper}
\address{Ruhr-Universit\"at Bochum, Fakult\"at f\"ur Mathematik, D-44780 Bochum, Deutschland}
\email{gerhard.knieper@rub.de}

\author{Thibault Lefeuvre}

\address{Laboratoire de Mathématiques d’Orsay, Univ. Paris-Sud, CNRS, Université Paris-Saclay, 91405 Orsay, France}

\email{thibault.lefeuvre@u-psud.fr}

\begin{document}

\begin{abstract}
We refine the recent local rigidity result for the marked length spectrum obtained by the first and third author in \cite{Guillarmou-Lefeuvre-18} and give an alternative proof using the geodesic stretch between two Anosov flows and some uniform estimate on the variance appearing in the central limit theorem for Anosov geodesic flows. In turn, we also introduce a new pressure metric on the space of isometry classes, that reduces to the Weil-Petersson metric in the case of Teichm\"uller space and is related to the works of \cite{McMullen,Bridgeman-Canary-Labourie-Sambarino-15}. 
\end{abstract}

\maketitle

\section{Introduction}

Let $M$ be a smooth closed $n$-dimensional manifold. We denote by $\mc{M}$ the Fr\'echet manifold consisting of smooth metrics on $M$. We denote by $\mc{M}^{k,\alpha}$ the set of metrics with regularity $C^{k,\alpha}$, $k \in \N, \alpha \in (0,1)$. We fix a smooth metric $g_0 \in \mc{M}$ with Anosov geodesic flow $\varphi_t^{g_0}$ and define the unit tangent bundle by $S_{g_0}M := \left\{(x,v) \in TM ~|~ |v|_{g_0}= 1\right\}$. Recall that being Anosov means that there exists a flow-invariant continuous splitting
\[
T(S_{g_0}M) = \R X \oplus E_s \oplus E_u,
\]
such that
\[\begin{gathered}
\|d\varphi_t^{g_0}(w)\| \leq Ce^{-\lambda t}\|w\|, \quad \forall w \in E_s, \forall t \geq 0, \\
 \|d\varphi_t^{g_0}(w)\| \leq Ce^{-\lambda |t|}\|w\|, \quad \forall w \in E_u, \forall t \leq 0,
\end{gathered}\]
where the constants $C, \lambda > 0$ are uniform and the norm here is the one induced by the Sasaki metric of $g_0$. Such a property is satisfied in negative curvature.

\subsection{Geodesic stretch and marked length spectrum rigidity}

The set of primitive free homotopy classes $\mathcal{C}$ of $M$ is in one-to-one correspondance with the primitive conjugacy classes of $\pi_1(M,x_0)$ (where $x_0 \in M$ is arbitrary). When $g_0$ is Anosov, there exists a unique closed geodesic $\gamma_{g_0}(c)$ in each primitive free homotopy class $c \in \mathcal{C}$ (see \cite{Klingenberg-74}). This allows us to define the \emph{marked length spectrum} of the metric $g_0$ by:
\[ L_{g_0} : \mathcal{C} \rightarrow \R_+, \qquad L_{g_0}(c) = \ell_{g_0}(\gamma_{g_0}(c)), \]
where $\ell_{g_0}(\gamma)$ denotes the $g_0$-length of a curve $\gamma \subset M$ computed with respect to $g_0$. The marked length spectrum can alternatively be defined for the whole set of free homotopy classes but it is obviously an equivalent definition. Given $c \in \mc{C}$, we will write $\delta_{g_0}(c)$ to denote the probability Dirac measure carried by the unique $g_0$-geodesic $\gamma_{g_0}(c) \in c$. 

It was conjectured by Burns-Katok \cite{Burns-Katok-85} that the marked length spectrum of negatively curved manifolds determine the metric up to isometry in the sense that two negatively curved metrics $g$ and $g_0$ with same marked length spectrum (namely $L_g = L_{g_0}$) should be isometric. Although the conjecture was proved for surfaces by Croke and Otal \cite{Otal-90, Croke-90}) and in some particular cases in higher dimension (for conformal metrics by Katok \cite{Katok-88} and when $(M,g_0)$ is a locally symmetric space by the work of Hamenst\"adt and Besson-Courtois-Gallot \cite{Besson-Courtois-Gallot-95, Hamenstadt-99}), it is still open in dimension higher or equal to $3$ and open even in dimension $2$ in the more general setting of Riemannian metrics with Anosov geodesic flows. The same type of problems can also be asked for billiards and we mention recent results on this problem by Avila-De Simoi-Kaloshin \cite{Avila} and De Simoi-Kaloshin-Wei \cite{DSKW} for convex domains close to ellipses (although the Anosov case would rather correspond to the case of hyperbolic billiards).
Recently, the first and last author obtained the following result  on Burns-Katok conjecture:

\begin{theorem*}[Guillarmou-Lefeuvre \cite{Guillarmou-Lefeuvre-18}]
Let $(M,g_0)$ be a smooth Riemannian manifold with Anosov geodesic flow and further assume that its curvature is nonpositive if $\dim M \geq 3$. Then there exists $k \in \N$ depending only on $\dim M$ and $\eps > 0$ small enough depending on $g_0$ such that the following holds: if $g \in \mc{M}$ is such that $\|g-g_0\|_{C^k} \leq \eps$ and $L_g = L_{g_0}$, then $g$ is isometric to $g_0$.
\end{theorem*}

One of the aims of this paper is to further investigate this result from different perspectives: new stability estimates and a refined characterization of the condition under which the isometry may hold. More precisely, we can relax the assumption that the two marked length spectra of $g$ and $g_0$ exactly coincide to the weaker assumption that they "coincide at infinity" and still obtain the isometry. 
In the following, we say that $L_g/L_{g_0} \rightarrow 1$ when 
\begin{equation}\label{limLg/Lg0}
\lim_{j \rightarrow +\infty} \dfrac{L_g(c_j)}{L_{g_0}(c_j)} = 1,
\end{equation}
for any sequence $(c_j)_{j \in \N}$ of primitive free homotopy classes such that $\lim_{j\to \infty}L_{g_0}(c_j)=+\infty$, or equivalently $\lim_{j\to \infty}L_{g}(c_j)/L_{g_0}(c_j)=1$, if $\mc{C}=(c_j)_{j\in\N}$ is ordered by the increasing lengths $L_{g_0}(c_j)$. We prove in the Appendix \ref{appendix}, that $L_g/L_{g_0} \rightarrow 1$ is actually equivalent to $L_g = L_{g_0}$. As a consequence, by \cite{Guillarmou-Lefeuvre-18}, if \eqref{limLg/Lg0} holds and if $\|g-g_0\|_{C^k}<\eps$ for some small enough $\eps>0$, then $g$ is isometric to $g_0$. If we restrict ourselves to metrics
with same topological entropy, the knowledge of $L_{g}(c_j)/L_{g_0}(c_j)$ for a subsequence so that the geodesic $\gamma_{g_0}(c_j)$ equidistributes is even sufficient, see Theorem \ref{theorem:entropy}.

We develop a new strategy of proof, different from \cite{Guillarmou-Lefeuvre-18}, which relies on the introduction of the \emph{geodesic stretch} between two metrics. This quantity was first introduced by Croke-Fathi \cite{Croke-Fathi} and further studied by the second author \cite{Knieper-95}. If $g$ is close enough to $g_0$, then by Anosov structural stability, the geodesic flows $\varphi^{g_0}$ and $\varphi^g$ are orbit equivalent via a homeomorphism $\psi_g$, i.e. they are conjugate up to a time reparametrization
\[\varphi^g_{\kappa_{g}(z,t)} (\psi_g(z)) = \psi_g(\varphi^{g_0}_{t}(z))\]
for some time rescaling $\kappa_g(z,t)$.
The \emph{infinitesimal stretch} is the infinitesimal function of time reparametrization $a_g(z)=\pl_{t}\kappa_g(z,t)|_{t=0}$: it satisfies $d\psi_g(z) X_{g_0}(z) = a_g(z) X_{g}(\psi_g(z))$ where $z \in S_{g_0}M$ and $X_{g_0}$ (resp. $X_g$) denotes the geodesic vector field of $g_0$ (resp. $g$). The \emph{geodesic stretch between $g$ and $g_0$ with respect to the Liouville\footnote{Normalized with total mass $1$.} measure $\mu^{\rm L}_{g_0}$ of $g_0$} 
is then defined by 
\[ I_{\mu_{g_0}^{\rm L}}(g_0,g):=\int_{S_{g_0}M} a_g \, d\mu^{\rm L}_{g_0}.\]
The function $a_g$ is uniquely defined up to a coboundary \cite{DeLaLlave-Marco-Moryon-86} so that the geodesic stretch is well-defined\footnote{Although this is only used in \S\ref{ssection:thurston}, we also point out that the existence of the conjugacy $\psi_g$ and of the reparametrization $a_g$ is actually \emph{global} and one does not need to assume that the two metrics are close. This is a very particular feature of the geodesic structure. We refer to Appendix \ref{appendix:conjugacy} for a proof of this fact.}.

Since obviously $\cjg \delta_{g_0}(c_j),a_g\cjd=L_g(c_j)/L_{g_0}(c_j)$, we have 
\[  I_{\mu_{g_0}^{\rm L}}(g_0,g)=\lim_{j\to \infty}\frac{L_{g}(c_j)}{L_{g_0}(c_j)},\]
if $(c_j)_{j\in\N}\subset \mc{C}$ is a sequence so that the uniform probability measures $(\delta_{g_0}(c_j))_{j \in \N}$ supported on the closed geodesics of $g_0$ in the class $c_j$ converge to $\mu^{\rm L}_{g_0}$ in the weak sense of measures.\footnote{The existence of the sequence $c_j$ follows from \cite[Theorem 1]{Sigmund}.} In particular $L_g=L_{g_0}$ implies that $I_{\mu_{g_0}^{\rm L}}(g_0,g)=1$ (alternatively $L_g=L_{g_0}$ implies that $a_g$ is cohomologous to $1$ by Livsic's theorem). 
 While it has an interest on its own, it turns out that this method involving the geodesic stretch provides a new estimate which quantifies locally  the distance between isometry classes in terms of this geodesic stretch functional (below $H^{-1/2}(M)$ denotes the $L^2$-based Sobolev space of order $-1/2$ and $\alpha\in(0,1)$ is any fixed exponent).

\begin{theorem}
\label{theorem:stability}
Let $(M,g_0)$ be a smooth Riemannian $n$-dimensional manifold with Anosov geodesic flow and further assume that its curvature is nonpositive if $n \geq 3$. There exists $k \in \N$ large enough depending only on $n$, some positive constants $C,\eps$ depending on $g_0$ and $C_n>0$ depending on $n$ such that for all $\alpha \in (0,1)$, the following holds: for each $g \in \mc{M}^{k,\alpha}$ with 
$\|g-g_0\|_{C^{k,\alpha}(M)} \leq \eps$, 
there exists a $C^{k+1,\alpha}$-diffeomorphism $\psi : M \rightarrow M$ such that
\[\begin{split}
C\|\psi^*g-g_0\|^2_{H^{-1/2}(M)} \leq & ~{\bf P}\Big(-J_{g_0}^u-a_g+\int_{S_{g_0}M}a_g\, d\mu_{g_0}^L\Big)+C_n\Big(I_{\mu_{g_0}^L}(g_0,g)-1\Big)^2\\
C\|\psi^*g-g_0\|^2_{H^{-1/2}(M)}   \leq &~ |\mc{L}_+(g)|+|\mc{L}_-(g)|
\end{split}\]
where $J^u_{g_0}$ is the unstable Jacobian of $\varphi^{g_0}$, ${\bf P}$ denotes the topological pressure for the $\varphi^{g_0}$ flow defined by \eqref{equation:pression}, $a_g$ is the reparameterization coefficient relating $\varphi^{g_0}$ and $\varphi^g$ defined above, and
\[\mc{L}_+(g):=\limsup_{j\to \infty} \frac{L_g(c_j)}{L_{g_0}(c_j)}-1, \quad \mc{L}_-(g):=\liminf_{j\to \infty} \frac{L_g(c_j)}{L_{g_0}(c_j)}-1.\]
In particular if \eqref{limLg/Lg0} holds, then $g_0$ and $g$ are isometric.
\end{theorem}

Note that $g$ does not need to have nonpositive curvature in the Theorem. We also remark that 
the curvature condition on $g_0$ can be replaced by the injectivity of the X-ray transform $I_2$ on divergence-free symmetric $2$-tensors, and similarly for Theorem \ref{mainth3} below. From the proof one sees that the exponent $k$ can be taken to be $k=3n/2+17$.

Theorem \ref{theorem:stability} is an improvement over the H\"older stability result \cite[Theorem 3]{Guillarmou-Lefeuvre-18} as it only involves the asymptotic behaviour of $L_g/L_{g_0}$ or some natural quantity from thermodynamic formalism. We insist on the fact that the new ingredient here is \emph{the stability estimate in itself} (the rigidity result is not new). 

We also emphasize that one of the key facts to prove this theorem still boils down to some elliptic estimate on some variance operator acting on symmetric $2$-tensors, denoted by $\Pi_2^{g_0}$ in \cite{Guillarmou-Lefeuvre-18,Guillarmou-17-1}:
indeed, we show that the combination of the Hessians of the geodesic stretch at $g_0$ and of the pressure functional can be expressed in terms of this variance operator, which enjoys uniform lower bounds $C_{g_0}\|\psi^*g-g_0\|_{H^{-1/2}}$ for some $C_{g_0}>0$, at least once we have factored out the gauge (the diffeomorphism action by pull-back on metrics).

We also notice that in Theorem \ref{theorem:stability}, although the $H^{-1/2}(M)$ norm is a weak norm, a straightforward interpolation argument using that $\|g\|_{C^{k,\alpha}}\leq \|g_0\|_{C^{k,\alpha}}+\eps$ is uniformly bounded shows that an estimate of the form
\[ \|\psi^*g-g_0\|_{C^{k'}} \leq  C \left(|\mc{L}_+(g)|+|\mc{L}_-(g)|\right)^{\delta}\]
holds for any $k'<k-n/2$ and some explicit $\delta\in (0,1/2)$ depending on $k,k'$ ($C>0$ depending only on $g_0$).

\subsection{Variance and pressure metric}

The variance operator appearing in the proof of Theorem \ref{theorem:stability}
can be defined for $h_1,h_2\in C^\infty(M;S^2T^*M)$ satisfying the condition
\begin{equation}
\label{equation:trace-free}
\int_{M}{\rm Tr}_{g_0}(h_i)\, d{\rm vol}_{g_0}=0,
\end{equation}
for $i=1,2$ (see \S\ref{sec:tensors} for further details on tensor analysis) by 
\[
\cjg \Pi_2^{g_0}h_1,h_2\cjd := \int_{\R}\int_{SM} \pi_2^*h_1(\varphi_t^{g_0}(z))\pi_2^*h_2(z)\, d\mu_{g_0}^{\rm L}(z) dt,
\]
where, $z = (x,v) \in SM$ and given a symmetric $2$-tensor $h \in C^\infty(M;S^2 T^*M)$, we define the pullback operator
\[
\pi_2^*h(x,v):=h_x(v,v).
\]
The quadratic form $\cjg \Pi_2^{g_0}h,h\cjd$ corresponds to the variance ${\rm Var}_{\mu_L}(\pi_2^*h)$ for $\varphi_t^{g_0}$ with respect to the Liouville measure
of the lift $\pi_2^*h$ of the tensor $h$ to $SM$ (see \S\ref{defvariance} and \eqref{defvariance}). Note that the trace-free condition \eqref{equation:trace-free} is equivalent to
\[
 \int_{SM} \pi_2^* h(x,v) d \mu^{\rm L}_{g_0}(x,v) = 0,
\]
see \S \ref{sec:tensors}. The integral defining $\Pi_2^{g_0}$ then converges (in the $L^1$ sense) by the rapid mixing of $\varphi^{g_0}$ (proved in \cite{Liverani}). 
The operator $\Pi^{g_0}_2$ is a pseudodifferential operator of order $-1$ that is elliptic on divergence-free tensors (see \cite{Guillarmou-17-1,Guillarmou-Lefeuvre-18,Gouezel-Lefeuvre-19}). As a consequence, it satisfies elliptic estimates on all Sobolev or H\"older spaces (see Lemma \ref{lemma:positivite-pi2}). More precisely, there is $C_{g_0}>0$ such that for all $h\in H^{-1/2}(M;S^2T^*M)$ which is divergence-free (i.e. ${\rm Tr}_{g_0}(\nabla^{g_0}h)=0$) 
\begin{equation}
\label{coercivePi2} 
\cjg \Pi_2^{g_0}h,h\cjd \geq C_{g_0}\|h\|^2_{H^{-1/2}(M)},
\end{equation}
provided $g_0$ is Anosov with non-positive curvature (or simply Anosov if $\dim M=2$).
We show in Proposition \ref{proposition:pi2-continuite} that $g\mapsto \Pi_2^g$ is continuous with values in $\Psi^{-1}(M)$ and this implies that for $g_0$ a smooth Anosov metric (with non-positive curvature if $\dim M>2$), \eqref{coercivePi2} holds uniformly if we replace $g_0$ by any metric $g$ in a small $C^\infty$-neighborhood of $g_0$. This allows to obtain a more uniform version of Theorem \ref{theorem:stability}.

\begin{theorem}\label{mainth3}
Let $(M,g_0)$ be a smooth Riemannian $n$-dimensional manifold with Anosov geodesic flow and further assume that its curvature is nonpositive if $n \geq 3$. Then there exists $k \in \N$, $\eps > 0$ and $C_{g_0}$ depending on $g_0$ such that for all $g_1,g_2 \in \mc{M}$ 
such that $\|g_1-g_0\|_{C^k} \leq \eps$, $\|g_2-g_0\|_{C^k} \leq \eps$, there is a $C^k$- diffeomorphism $\psi:M\to M$ such that 
\[
\|\psi^*g_2-g_1\|^2_{H^{-1/2}(M)} \leq  C_{g_0}(|\mc{L}_+(g_1,g_2)|+|\mc{L}_+(g_2,g_1)|)
\]
with 
\[\mc{L}_+(g_1,g_2):=\limsup_{j\to \infty} \frac{L_{g_2}(c_j)}{L_{g_1}(c_j)}-1.\]
In particular if $L_{g_1}/L_{g_2}\to 1$, then $g_2$ is isometric to $g_1$.
\end{theorem}

This result suggests to define a distance on isometry classes\footnote{Here we mean isometries homotopic to the Identity.} metrics from the marked length spectrum by setting for $g_1,g_2$ two $C^{k,\alpha}$ metrics  
\[ d_L(g_1,g_2):=\limsup_{j\to \infty}
\Big|\log \frac{L_{g_1}(c_j)}{L_{g_2}(c_j)}\Big|.\]
We have as a corollary of Theorem \ref{mainth3}:

\begin{corollary}
The map $d_L$ descends to the space of isometry classes of Anosov non-positively curved metrics and defines a distance near the diagonal. 
\end{corollary}

We also define the \emph{Thurston asymmetric distance} by 
\[
d_T(g_1,g_2):=\limsup_{j\to \infty}
\log \frac{L_{g_2}(c_j)}{L_{g_1}(c_j)},
\]
and show that this is a distance on isometry classes of metrics with topological entropy equal to $1$, see Proposition \ref{proposition:dt-distance}. This distance was introduced in Teichm\"uller theory by Thurston in \cite{Thurston-98}.\\

The elliptic estimate \eqref{coercivePi2} allows also to define a \emph{pressure metric} on the open set consisting of isometry classes of Anosov non-positively curved metric (contained in 
$\mc{M}/\mc{D}_0$ if $\mc{D}_0$ is the group of smooth diffeomorphisms isotopic to the identity) by
setting for $h_1,h_2\in T_{g_0}(\mc{M}/\mc{D}_0)\subset C^\infty(M;S^2T^*M)$
\[ 
G_{g_0}(h_1,h_2):= \cjg \Pi^{g_0}_2 h_1,h_2 \cjd_{L^2(M,d\vol_{g_0})}.
\]
We show in Section \ref{sec:pressure metric} that this metric is well-defined and restricts to (a multiple of) the Weil-Petersson metric on Teichm\"uller space if $\dim M=2$: it is related to the construction of Bridgeman-Canary-Labourie-Sambarino \cite{Bridgeman-Canary-Labourie-Sambarino-15, Bridgeman-Canary-Sambarino-18} and Mc Mullen \cite{McMullen}, but with the difference that we work here in the setting of variable negative curvature and the space of metrics considered here is infinite dimensional. In a related but different context with infinite dimension, we note that the variance is used to define a metric on the space of H\"older potentials by Giulietti, Kloeckner, Lopes and Marcon \cite{GKLM} and its curvature is studied by Lopes and Ruggiero \cite{LopesRuggiero}.

We finally notice that, in the study of Katok entropy conjecture near locally symmetric spaces, the variance was an important tool in the work of Pollicott and Flaminio \cite{Pollicott,Flaminio}. In that case, one can use representation theory to analyze this operator.\\

\noindent \textbf{Acknowledgement:} We warmly thank the referees for their many helpful comments. In particular, one of the referees suggested a short and elegant argument to show that $L_g/L_{g_0} \rightarrow 1$ implies $L_g=L_{g_0}$ (see Appendix \ref{appendix}). 
This project has received funding from the European Research Council (ERC) under the European Union’s Horizon 2020 research and innovation programme (grant agreement No. 725967). This material is based upon work supported by the National Science Foundation under Grant No. DMS-1440140 while C.G and T.L were in residence at the Mathematical Sciences Research Institute in Berkeley, California, during the Fall 2019 semester. The second author was partially supported by the SFB/TRR 191 "Symplectic structures in geometry, algebra and dynamics".

\section{Preliminaries}

\noindent \textbf{Notation:} If $H=C^k,H^s,C^{-\infty}$ etc is a regularity scale and $E\to M$ a smooth bundle over a smooth compact manifold $M$, we will use the notation $H(M;E)$ for sections of $E$ with regularity $H$ while, if $N$ is a smooth manifold, we use the notation $H(M,N)$ for the space of maps from $M$ to $N$ with regularity $H$.

\subsection{Microlocal calculus}
\label{ssection:microlocal}

On a closed manifold $M$, we will denote by $\Psi^{m}(M;V)$ the space of classical 
pseudo-differential operators of order $m\in\R$ acting on a vector bundle $V$ over $M$ (see \cite{GS94}; the operators could map sections of two distinct vector bundles but this will not be needed here). We recall that for fixed $m \in \R$, this is a Fr\'echet space: indeed, using a fixed smooth cutoff function $\theta$ supported in a small neighborhood of the diagonal, a fixed system of charts, 
each $A\in \Psi^{m}(M;V)$ has Schwartz kernel $\kappa_A$ that can be decomposed as $\theta\kappa_A+(1-\theta)\kappa_A$. For the part $(1-\theta)\kappa_A$ we can use the $C^\infty(M\times M; V\otimes V^*)$ topology while for $\chi\kappa_A$ one can use the semi-norms of the full 
symbols of $\chi\kappa_A$ using the local charts the left quantization in the charts. 
We also denote by $H^{s}(M)$ the $L^2$-based Sobolev space of order $s\in \R$, with norm given by fixing an arbitrary Riemannian metric $g_{0}$ on $M$. More precisely, denoting by $\Delta$ the non-negative Laplacian associated to this metric, we define
\[
\|f\|_{H^s(M)} := \|(\mathbbm{1}+\Delta)^{s/2}f\|_{L^2(M,d{\rm vol})},
\]
and $H^s(M)$ is the completion of $C^\infty(M)$ with respect to this norm. This definition is naturally extended to section of vector bundles. What is important is that the spaces and the norm (up to a scaling factor) do not depend on the choice of metric $g_{0}$. For $k\in \N, \alpha\in(0,1)$, the spaces $C^{k,\alpha}(M)$ are the usual H\"older spaces and $\mc{D}'(M)$ will denote the space of distributions, dual to $C^\infty(M)$. We will denote by $\cjg \cdot,\cdot \cjd_{L^2}$ the continuous extension of the pairing
\[
C^\infty(M) \times C^\infty(M) \ni (f,f') \mapsto \int_{M}f\bar{f}'  d{\rm vol}_{g_0},
\]
to the pairing  $H^{s}(M) \x H^{-s}(M)\to \C$ for each $s\in \R$ (and same thing for sections of bundles).
 
\subsection{Symmetric tensors and X-ray transform}\label{sec:tensors}

In this paragraph, we assume that the metric $g$ is fixed and that its geodesic flow $\varphi^g_t$ is Anosov on the unit tangent bundle $SM$ of $g$. We denote by $\mu^{\rm L}$ the Liouville measure, normalized to be a probability measure on $SM$. For the sake of simplicity, we drop the index $g$ in the notations. Given an integer $m \in \N$, we denote by $\otimes^m T^*M \rightarrow M$, $S^mT^*M \rightarrow M$ the respective vector bundle of $m$-tensors and symmetric $m$-tensors on $M$. Given $f \in C^\infty(M;S^m T^*M)$, we denote by $\pi_m^*f \in C^\infty(SM)$ the canonical morphism $\pi_m^*f : (x,v) \mapsto f_x(v,...,v)$. We also introduce the trace operator $\Tr : C^\infty(M;S^{m+2} T^*M) \rightarrow C^\infty(M;S^{m} T^*M)$ defined pointwise in $x \in M$ by:
\[
\Tr(f) = \sum_{i=1}^n f(\mathbf{e}_i,\mathbf{e}_i, \cdot, ..., \cdot),
\]
where $(\mathbf{e}_1, ..., \mathbf{e}_n)$ denotes an orthonormal basis of $TM$ in a neighborhood of a fixed point $x_0 \in M$. Observe that, for $f = \sum_{i,j=1}^n f_{ij}  \mathbf{e}_i^* \otimes \mathbf{e}_j^* \in C^\infty(M;S^2T^*M)$ defined around $x_0$, we have
\[
\begin{split}
\int_{SM} \pi_2^*f\, d\mu_{g}^L &= \int_{M} \left(\int_{S_x M} \pi_2^*f(x,v) dS_x(v) \right)d{\rm vol}(x)\\
& =  \sum_{i,j=1}^n \int_{M} f_{ij}(x) \left( \int_{S_x M} v_i v_j dS_x(v) \right)d{\rm vol}(x) \\
& = C_n \sum_{i=1}^n \int_{M} f_{ii}(x) d{\rm vol}(x) = C_n \int_{M} \Tr_{g}(f) d{\rm vol},
\end{split}
\]
for some constant $C_n = \int_{\Ss^{n-1}} v_1^2 dv$ depending on $n=\dim M$. This justifies the claim that the trace-free condition \eqref{equation:trace-free} was equivalent to the fact that the pullback of the symmetric tensor to $SM$ was of average $0$.

The natural derivation of symmetric tensors is $D := \sigma \circ \nabla$, where $\nabla$ is the Levi-Civita connection and $\sigma : \otimes^m T^*M \rightarrow S^m T^*M$ is the operation of symmetrization. This operator satisfies the important identity: 
\begin{equation}
\label{equation:identite}
X \pi_m^* = \pi_{m+1}^* D,
\end{equation}
where $X$ denotes the geodesic vector field on $SM$.
The operator $D$ is elliptic \cite[Lemma 2.4]{Gouezel-Lefeuvre-19} with trivial kernel when $m$ is odd and $1$-dimensional kernel when $m$ is even, given by the Killing tensors $c\sigma(g^{\otimes m/2}), c \in \R$ (this is a simple consequence of \eqref{equation:identite} combined with the fact that the geodesic flow is ergodic in the Anosov setting). We denote by $\langle \cdot, \cdot \rangle$ the scalar product on $C^\infty(M;S^mT^*M)$ induced by the metric $g$ (see \cite[Section 2]{Gouezel-Lefeuvre-19} for further details). The formal adjoint of $D$ with respect to this scalar product is $D^*=-\Tr\circ \nabla$. We also denote by the same $\langle \cdot, \cdot \rangle$ the natural $L^2$ scalar product on $C^\infty(SM)$ induced by the Liouville measure $\mu^{{\rm L}}$. The formal adjoint of $\pi_m^*$ with respect to these two scalar products is denoted by 
\[
{\pi_m}_*: \mc{D}'(SM)\to \mc{D}'(M; S^mT^*M)
\]  
where $\mc{D}'$ denotes the space of distributions, dual to $C^\infty$.\\

We recall that $\mathcal{C}$, the set of free homotopy classes in $M$, is in one-to-one correspondance with the set of conjugacy classes of $\pi_1(M,x_0)$ for some arbitrary choice of $x_0 \in M$ (see \cite{Klingenberg-74}) and for each $c \in \mathcal{C}$ there exists a unique closed geodesic $\gamma(c) \in c$. We denote its Riemannian length with respect to $g$ by $L(c)=\ell_g(\gamma(c))$.  The \emph{X-ray transform} on $SM$ is the operator defined by:
\[
I : C^0(SM) \rightarrow \ell^\infty(\mathcal{C}), \quad If(c) = \dfrac{1}{L(c)} \int_0^{L(c)} f(\varphi_t(z)) dt,
\]
where $z \in \gamma(c)$ is any point. This is a continuous linear operator when $\ell^\infty(\mathcal{C})$ is endowed with the sup norm on the sequences. Then, the X-ray transform $I_m$ of symmetric $m$-tensors is simply defined by $I_m := I \circ \pi_m^*$. Using \eqref{equation:identite}, we immediately have
\be
\label{equation:inclusion-noyau} \left\{ Dp ~|~  p \in C^\infty(M;S^{m-1}T^*M) \right\} \subset \ker I_m \cap C^\infty(M;S^{m}T^*M).
\ee
Using the ellipticity of $D$, any tensor $f \in C^\infty(M; S^m T^*M)$ can be decomposed uniquely as a sum
\be
\label{equation:decomposition}
f=Dp+h,
\ee
with $p \in C^\infty(M;S^{m-1}T^*M)$ and $h \in C^\infty(M;S^{m}T^*M)$ is such that $D^*h=0$.  
We call $Dp$ the \emph{potential part} of $f$ and $h$ the \emph{solenoidal part}. The same decomposition holds in Sobolev regularity $H^s(M)$, $s \in \R$, and in the $C^{k,\alpha}(M)$ regularity, $k \in \N, \alpha \in (0,1)$. We will write $h = \pi_{\ker D^*}f$ and the solenoidal projection $\pi_{\ker D^*}:=\mathbbm{1} - D_g \Delta^{-1}_g D^*_g$ is a pseudodifferential operator of order $0$ \cite[Lemma 2.6]{Gouezel-Lefeuvre-19} (here $\Delta_g := D^*_g D_g$ is the Laplacian on $1$-forms). The X-ray transform is said to be \emph{solenoidal injective} (or s-injective in short) if (\ref{equation:inclusion-noyau}) is an equality. It is conjectured that $I_m$ is s-injective as long as the metric is Anosov but it is only known in the following cases:
\begin{itemize}
\item for $m=0,1$ \cite{Dairbekov-Sharafutdinov-10},
\item for any $m \in \N$ in dimension $2$ \cite{Paternain-Salo-Uhlmann-14-2, Guillarmou-17-1},
\item for any $m \in \N$, in any dimension in non-positive curvature \cite{Croke-Sharafutdinov-98}.
\end{itemize}
It is also known that $\ker I_m/\ran D$ is finite dimensional for general Anosov geodesic flow (see \cite[Theorem 1.5]{Dairbekov-Saharfutdinov-Anosov} or \cite[Remark 3.7]{Guillarmou-17-1}).

The direct study of the analytic properties of $I_m$ is difficult as this operator involves integrals over the set of closed orbits, which is not a manifold. 
 Nevertheless, in \cite{Guillarmou-17-1}, the second author introduced an operator $\Pi_m$ that involves a sort of integration of tensors over "all orbits" and this space is essentially the manifold $SM$. The construction of $\Pi_m: C^\infty(M;S^mT^*M)\to \mc{D}'(M;S^mT^*M)$ relies on microlocal tools coming from \cite{Faure-Sjostrand-11,Dyatlov-Zworski-16} but a simpler definition that uses the fast mixing of the flow $\varphi_t$ is given by
\begin{equation}
\label{Pim}
\begin{gathered}
\Pi_m:= {\pi_m}_*(\Pi+\langle \cdot, 1\rangle) \pi_m^* \textrm{ with } \\
\Pi: C^\infty(SM)\to \mc{D}'(SM), \quad \cjg \Pi f,f'\cjd:=\lim_{T\to \infty}\int_{-T}^T
\cjg e^{tX}f,f'\cjd \, dt  
\end{gathered}\end{equation} 
if $\cjg f,1\cjd=\int_{SM} f \, d\mu^{\rm L}=0$ and $\Pi (1):=0$. The convergence of the integral as $T\to \infty$ is ensured by the exponential decay of correlations \cite{Liverani} (but also follows from the existence of the variance \cite{KatsudaSunada}). We can thus write for $\cjg f,1\cjd=0$
\[\cjg \Pi f,f'\cjd =\int_\R 
\cjg f\circ \varphi_t,f'\cjd_{L^2(SM)}\, dt.\]
We note the following useful properties of $\Pi$, proved in \cite[Theorem 1.1]{Guillarmou-17-1}:
\begin{itemize}
\item $\Pi: H^s(SM)\to H^{-s}(SM)$ is bounded for all $s>0$
\item if $f\in H^s(SM)$ with $s>0$, $X\Pi f=0$
\item if $f$ and  $Xf$ belong to  $H^s(SM)$ for $s>0$, then  $\Pi Xf=0$\footnote{In \cite{Guillarmou-17-1}, $f$ is assumed to be in $H^{s+1}(SM)$ but one can reduce to the case $f\in H^s(SM)$ by using a density argument and \cite[Lemma E.45]{Dyatlov-Zworski-book-resonances}}. 
\end{itemize}
 As is well known (see for example \cite[Proof of Proposition 1.2.]{KatsudaSunada}), we can make a link between $\Pi$ and the variance in the central limit theorem for Anosov geodesic flows. Let us quickly explain this fact by using the fast mixing of the flow.
The \emph{variance} of $\varphi_t$ with respect to the Liouville measure $\mu^{\rm L}$ is defined for $u\in C^\alpha(SM), \alpha \in (0,1)$ real-valued by:
\begin{equation}\label{defvariance} 
{\rm Var}_{\mu^{\rm L}}(u):=\lim_{T\to \infty}\frac{1}{T}\int_{SM}\Big(\int_0^Tu(\varphi_t(z)) \, dt\Big)^2d\mu^{\rm L}(z),
\end{equation} 
under the condition that $\int_{SM}u\, d\mu^{\rm L}=0$. We observe, since $\varphi_t$ preserves $\mu^{\rm L}$, that 
\[ \begin{split}
{\rm Var}_{\mu^{\rm L}}(u)=& \lim_{T\to \infty}\frac{1}{T}\int_{SM}\int_0^T\int_{0}^T u(\varphi_{t-s}(z))u(z) \, dtds d\mu^{\rm L}(z)\\
=& \lim_{T\to \infty} \int_0^1\int_{\R}{\bf 1}_{[(t-1)T,tT]}(r)\cjg u\circ \varphi_r,u\cjd_{L^2} \, dr dt. 
\end{split}\]
where the $L^2$ pairing is with respect to $\mu^{\rm L}$.
By exponential decay of correlations \cite{Liverani}, we have for $|r|$ large 
\[|\cjg u\circ \varphi_r,u\cjd_{L^2} |\leq Ce^{-\nu |r|}\|u\|_{C^\alpha}^2\]
for some $\alpha>0,\nu>0$, $C>0$ independent of $u$. Thus, by the Lebesgue theorem, 
\begin{equation}\label{variancevsPi}
{\rm Var}_{\mu^{\rm L}}(u)=\cjg \Pi u,u\cjd,
\end{equation}
if $\cjg u,\mathbf{1}\cjd=0$, where $\mathbf{1}$ denotes the constant function equal to $1$, showing that the quadratic form associated to our operator $\Pi$ is nothing more than the variance. For a symmetric $2$-tensor $h$ satisfying $\cjg h,g\cjd_{L^2}=\int_{M}{\rm Tr}_{g}(h)\, d{\rm vol}_{g}=0$, we have $\int_{SM}\pi_2^*h \, d\mu_{g}^L=0$ and 
\[ \cjg\Pi_2h,h\cjd =\cjg \Pi \pi_2^*h,\pi_2^*h\cjd={\rm Var}_{\mu^{\rm L}}(\pi_2^*h).\]

One has the following properties for $\Pi_m$:

\begin{itemize}
\item $\Pi_m$ is a positive self-adjoint pseudodifferential operator of order $-1$ , elliptic on solenoidal tensors, see \cite[Theorem 3.5]{Guillarmou-17-1} and \cite[Lemma 4.3]{Gouezel-Lefeuvre-19}.
\item $\Pi_mD = 0$ and $D^*\Pi_m=0$ (by \cite[Theorem 3.5]{Guillarmou-17-1} and $X\pi_{m-1}^*=
\pi_m^*D$)
\item If $I_m$ is s-injective, then $\Pi_m$ is invertible on solenoidal tensors in the sense that there exists a pseudodifferential operator $Q$ of order $1$ such that $Q\Pi_m=\pi_{\ker D^*}$, see \cite[Theorem 4.7]{Gouezel-Lefeuvre-19}.
\item Conversely, if $\Pi_m|_{\ker D^*}$ is injective, then $I_m$ is s-injective: indeed, by \cite[Corollary 2.8]{Guillarmou-17-1}, if $I_mh=0$ then $\pi_m^*h=Xu$ for some $u\in C^\infty(SM)$ and thus 
$\Pi_mh={\pi_m}_*\Pi Xu=0$.
\end{itemize} 
In particular, using the spectral theorem, there is a bounded self-adjoint operator 
$\sqrt{\Pi_m}$ on $L^2$ such that $\sqrt{\Pi_m}\sqrt{\Pi_m}=\Pi_m$.
We add the following property which will be crucially used in this article:

\begin{lemma}
\label{lemma:positivite-pi2} If $(M,g)$ has Anosov geodesic flow and $I_2$ is s-injective,
there exists a constant $C > 0$ such that for all tensors $h \in H^{-1/2}(M; S^2T^*M)$, 
\[\langle \Pi_2 h, h \rangle \geq C\|\pi_{\ker D^*}h\|^2_{H^{-1/2}(M)}.\]
\end{lemma}

\begin{proof}
In \cite[Theorem 4.4 and Lemma 2.2.]{Gouezel-Lefeuvre-19}, the principal symbol of $\Pi_2$ was computed and turned out to be
\[ 
\sigma_2 := \sigma(\Pi_2) : (x,\xi) \mapsto  |\xi|^{-1} \pi_{\ker i_\xi}A_2^2\pi_{\ker i_\xi},
\]
for some positive definite diagonal endomorphism $A_2$
which is constant on both subspaces 
$S^2_0T^*M:=\{h\in S^2T^*M |\, {\rm Tr}_g(h)=0\}$ and $\R g=\{\la g\in S^2T^*M\, |\, \la\in \R\}$. Here $i_\xi$ is the interior product with the dual vector 
$\xi^\sharp\in T_xM$ of $\xi$ with respect to the metric.
We introduce the symbol $b\in C^\infty(T^*M)$ of order $-1/2$ defined by $b :(x,\xi) \mapsto \chi(x,\xi)|\xi|^{-1/2}A_2$, where $\chi \in C^\infty(T^*M)$ vanishes near the $0$ section in $T^*M$ and equal to $1$ for $|\xi|>1$ and define $B := \Op(b)\in \Psi^{-1/2}(M;S^2T^*M)$, where $\Op$ is a quantization on $M$. Using that the principal symbol of $\pi_{\ker D^*}$ is $\pi_{\ker i_\xi}$ (see \cite[Lemma 2.6]{Gouezel-Lefeuvre-19}), we observe that $\Pi_2 = \pi_{\ker D^*}B^*B\pi_{\ker D^*} + R$, where $R\in \Psi^{-2}(M;S^2T^*M)$. Thus, given $h \in H^{-1/2}(M,S^2 T^*M)$:
\be
\label{equation:parametrice1}
\langle \Pi_2h, h \rangle_{L^2} = \|B\pi_{\ker D^*}h\|^2_{L^2} + \langle Rh, h \rangle_{L^2}
\ee

By ellipticity of $B$, there exists a pseudodifferential operator $Q$ of order $1/2$ such that $QB\pi_{\ker D^*} = \pi_{\ker D^*} + R'$, where $R'\in \Psi^{-\infty}(M;S^2T^*M)$ is smoothing. Thus there is $C>0$ such that for each $h\in C^\infty(M;S^2T^*M)$
\[
\|\pi_{\ker D^*}h\|^2_{H^{-1/2}} \leq \|QB\pi_{\ker D^*}h\|_{H^{-1/2}}^2 + \|R'h\|_{H^{-1/2}}^2 \leq C \|B\pi_{\ker D^*}h\|^2_{L^2} +  \|R'h\|_{H^{-1/2}}^2.
\]
Since Lemma \ref{lemma:positivite-pi2} is trivial on potential tensors, we can already assume that $h$ is solenoidal, that is $\pi_{\ker D^*}h=h$. Recalling (\ref{equation:parametrice1}), we obtain that
\be
\begin{split}
\label{equation:parametrice2}
\|h\|^2_{H^{-1/2}} & \leq C\langle \Pi_2 h, h \rangle_{L^2} - C\langle Rh, h\rangle_{L^2} +  \|R'h\|_{H^{-1/2}}^2\\
& \leq C\langle \Pi_2 h, h \rangle_{L^2} + C\|Rh\|_{H^{1/2}}\|h\|_{H^{-1/2}}  + \|R'h\|_{H^{-1/2}}^2.
\end{split}
\ee
Now, assume by contradiction that the statement in Lemma \ref{lemma:positivite-pi2} does not hold, that is we can find a sequence of tensors $f_n \in C^\infty(M;S^2 T^*M)$ such that $\|f_n\|_{H^{-1/2}}=1$ with $D^*f_n=0$ and
\[
\|\sqrt{\Pi_2}f_n\|^2_{L^2} =\langle \Pi_2 f_n, f_n \rangle_{L^2} \leq  \frac{1}{n}\| f_n\|^2_{H^{-1/2}}=\frac{1}{n} \rightarrow 0.
\]
Up to a subsequence, and since $R$ is of order $-2$, we can assume that $Rf_n \rightarrow v_1$ in $H^{1/2}$ for some $v_1$, and $R'f_n \rightarrow v_2$ in $H^{-1/2}$. Then, using (\ref{equation:parametrice2}), we obtain that $(f_n)_{n \in \N}$ is a Cauchy sequence in $H^{-1/2}$ which thus converges to an element $v_3 \in H^{-1/2}$ such that $\|v_3\|_{H^{-1/2}}=1$ and $D^*v_3=0$. By continuity, $\Pi_2f_n \rightarrow \Pi_2v_3$ in $H^{1/2}$ and thus $\langle \Pi_2v_3,v_3 \rangle=0$. Since $v_3$ is solenoidal, we get $\sqrt{\Pi_2}v_3=0$, thus $\Pi_2v_3=0$.  Since we assumed $I_2$ s-injective, $\Pi_2$ is also injective by \cite[Lemma 4.6]{Gouezel-Lefeuvre-19}. This implies that $v_3 \equiv 0$, thus contradicting $\|v_3\|_{H^{-1/2}}=1$.
\end{proof}
We note that the same proof also works for tensors of any order $m \in \N$.
In fact we can even get a uniform estimate:
\begin{lemma}
\label{lemma:positivite-pi2unif} Let $(M,g_0)$ be a smooth compact Anosov Riemannian manifold with $I_2^{g_0}$ being $s$-injective.
There exists a $C^\infty$ neighborhood $\mc{U}_{g_0}$ of $g_0$ and a constant $C > 0$ such that for all $g\in \mc{U}_{g_0}$ and all tensors $h \in H^{-1/2}(M; S^2T^*M)$, 
\[\langle \Pi_2^g h, h \rangle_{L^2} \geq C\|\pi_{\ker D_g^*}h\|^2_{H^{-1/2}(M)}.\]
\end{lemma}

\begin{proof}
First, let $g_0$ be fixed Anosov metric with $I_2^{g_0}$ $s$-injective (in particular it is the case if it has non-positive curvature). Proposition \ref{proposition:pi2-continuite} (which will be proved later) shows that the operator $\Pi_2=\Pi_2^g$ is a continuous family as a map
\[ g\in \mc{U}_{g_0} \mapsto \Pi_2^g\in \Psi^{-1}(M;S^2T^*M)\]
where $\mc{U}_{g_0}\subset C^\infty(M;S^2T^*M)$ is a $C^\infty$-neighborhood of $g_0$ and $\Psi^{-1}(M;S^2T^*M)$ is equipped with its Fr\'echet topology as explained before. 
 Let $h \in \ker D^*_g$ be a solenoidal (with respect to $g$) symmetric $2$-tensor, then $h=\pi_{\ker D^*_g}h$. Let $C_{g_0} > 0$ be the constant provided by Lemma \ref{lemma:positivite-pi2} applied to the metric $g_0$. We choose $\mc{U}_{g_0}$ small enough so that $\|\Pi^g_2-\Pi^{g_0}_2\|_{H^{-1/2}\rightarrow H^{1/2}} \leq C_{g_0}/3$ (this is made possible by the continuity of $g \mapsto \Pi^g_2 \in \Psi^{-1}$). Then:
\[
\cjg \Pi^g_2 h,h \cjd  = \cjg (\Pi^g_2-\Pi^{g_0}_2)h,h \cjd + \cjg \Pi^{g_0}_2 h,h\cjd  \geq C_{g_0} \|\pi_{\ker D^*_{g_0}} h\|_{H^{-1/2}}^2 - C_{g_0}/3 \times \|h\|_{H^{-1/2}}^2.
\]
But the map $\mc{U}_{g_0} \ni g \mapsto \pi_{\ker D^*_g} = \mathbbm{1} - D_g \Delta^{-1}_g D^*_g \in \Psi^0$ is continuous: this follows from the fact that one can construct a full parametrix $Q_g\in \Psi^{-2}(M)$ of $\Delta_g$ modulo smoothing in a continuous way with respect to $g$ (by standard elliptic microlocal analysis), the fact that $\Delta_g$ is injective since $\ker D_g=0$ for $g$ Anosov (as $D_gu=0$ implies $X\pi_1^*u=0$, 
thus $\pi_1^*u$ has to be constant, thus $0$ since $\pi_1^*u(x,-v)=-\pi_1^*u(x,v)$) and the continuity of composition of pseudodifferential operators.  
This implies that for $g$ in a possibly smaller neighborhood $\mc{U}_{g_0}$ of $g_0$, using $h=\pi_{\ker D^*_g}h$:
\[
\cjg \Pi^g_2 h,h \cjd \geq C_{g_0} \|\pi_{\ker D^*_{g}} h\|_{H^{-1/2}}^2 - \frac{2C_{g_0}}{3} \times \|h\|_{H^{-1/2}}^2 = C_{g_0}/3 \|\pi_{\ker D^*_{g}} h\|_{H^{-1/2}}^2.
\]
The proof is complete.
\end{proof}

We also observe that the generalization of the previous Lemma to tensors of any order is straightforward. As mentioned earlier, an immediate consequence of the previous lemma is the following

\begin{proposition}
\label{proposition:im-s-injective}
Let $(M,g_0)$ be a smooth Riemannian $n$-dimensional Anosov manifold with $I_m^{g_0}$ $s$-injective. Then, there exists a $C^\infty$-neighborhood $\mc{U}_{g_0}$ of $g_0$ in $\mc{M}$ such that for any $g \in \mc{U}_{g_0}$, for any $m \in \N$, $I^g_m$ is s-injective.
\end{proposition}
\begin{proof}
As mentionned above (before Lemma \ref{lemma:positivite-pi2}), the s-injectivity of $I_m^g$ is equivalent to that of $\Pi_m^g$ on solenoidal tensors and the previous Lemma allows to conclude.
\end{proof}

\subsection{The space of Riemannian metrics}

We fix a \emph{smooth} metric $g_0 \in \mc{M}$ and consider an integer $k \geq 2$ and $\alpha \in (0,1)$. We recall that the space $\mc{M}$ of all smooth metrics is a Fréchet manifold. We denote by $\mc{D}_0 := \Diff_0(M)$ the group of smooth diffeomorphisms on $M$ that are isotopic to the identity, this is a Fr\'echet Lie group in the sense of \cite[Section 4.6]{Hamilton}. The right action 
\[ \mc{M} \times \mc{D}_0\to \mc{M},\quad (g,\psi)\mapsto \psi^*g\]
is smooth and proper  \cite{Ebin-68,Ebin1970}. Moreover, if $g$ is a metric with Anosov geodesic flow, it is directly seen from ergodicity that there are no Killing vector fields and thus the isotropy subgroup $\{\psi\in \mc{D}_0 ~|~ \psi^*g=g\}$ of $g$ is finite. For negatively curved metrics it is shown in \cite{Frankel} that the action is free, i.e. the isotropy group is trivial.
One cannot apply the usual quotient theorem \cite[p.20]{Tromba-92} in the setting of Banach or Hilbert manifolds but rather smooth Fr\'echet manifolds instead (using the Nash-Moser theorem). 
Thus, in the setting of the space of smooth metrics with Anosov geodesic flows\footnote{The important fact, to apply Ebin's slice theorem, is that metrics with Anosov geodesic flows do not have Killing vector fields, i.e. infinitesimal isometries. This is due to the fact that $\ker D|_{C^\infty(M,T^*M)} = \left\{ 0 \right\}$ as mentioned earlier, which itself follows from the ergodicity of the geodesic flow.}, which is an open set of a Frechet vector space, the slice theorem says that there is a neighborhood $\mc{U}$ of $g_0$, a neighborhood $\mc{V}$ of ${\rm Id}$ in $\mc{D}_0$ and a Frechet submanifold $\mc{S}$ containing $g_0$ so that 
\begin{equation}\label{sliceS} 
\mc{S}\times \mc{V}\to \mc{U}, \quad (g,\psi)\mapsto \psi^*g
\end{equation}
is a diffeomorphism of Frechet manifolds and $T_{g_0}\mc{S}=\{h\in T_{g_0}\mc{M} ~|~ D_{g_0}^*h=0\}$, see \cite{Ebin-68,Ebin1970}. Moreover $\mc{S}$ parametrizes the set of orbits $g\cdot \mc{D}_0$ for $g$ near $g_0$ and 
 $T_g\mc{S}\cap T(g \cdot \mc{D}_0)=0$. 

On the other hand, if one considers $\mc{M}^{k,\alpha}$, the space of metrics with $C^{k,\alpha}$ regularity and $\mc{D}^{k+1,\alpha}_0 := \Diff^{k+1,\alpha}_0(M)$, the group of diffeomorphisms isotopic to the identity with $C^{k+1,\alpha}$ regularity,
then both spaces are smooth Banach manifolds. However, the action of $\mc{D}_0^{k+1,\alpha}$ on $\mc{M}^{k,\alpha}$ is no longer smooth but only topological which also prevents us from applying the quotient theorem.

Nevertheless, recalling $g_0$ is smooth, if we consider $\mc{O}^{k,\alpha}(g_0) := g_0\cdot \mc{D}^{k+1,\alpha}_0 \subset \mc{M}^{k,\alpha}$, then this is a smooth submanifold of $\mc{M}^{k,\alpha}$ and
\[
T_{g}\mc{O}^{k,\alpha}(g_0) = \left\{D_g p ~|~ p \in C^{k+1,\alpha}(M;T^*M) \right\}. 
\]
Notice that (\ref{equation:decomposition}) in $C^{k,\alpha}$ regularity exactly says that given $g \in \mc{O}^{k,\alpha}(g_0)$, one has the decomposition:
\be
\label{equation:decomposition-plan-tangent}
T_g \mc{M} = T_{g}\mc{O}^{k,\alpha}(g_0) \oplus \ker D_g^*|_{C^{k,\alpha}(M,S^2T^*M)}.
\ee
Thus, an infinitesimal perturbation of a metric $g \in \mc{O}^{k,\alpha}(g_0)$ by a symmetric $2$-tensor that is solenoidal with respect to $g$ is actually an infinitesimal displacement \emph{transversally to the orbit $\mc{O}^{k,\alpha}(g_0)$}.

We will need a stronger version of the previous decomposition (\ref{equation:decomposition-plan-tangent}) which can be understood as a slice theorem. Its knowledge goes back to \cite{Ebin-68,Ebin1970}, see also \cite[Lemma 4.1]{Guillarmou-Lefeuvre-18} for a short proof in the $C^{k,\alpha}$ category.

\begin{lemma}
\label{lemma:solenoidal-gauge}
Let $k$ be an integer $\geq 2$ and $\alpha \in (0,1)$, let $g_0$ be a $C^{k+3,\alpha}$ metric with Anosov geodesic flow. There exists a neighborhood $\mc{U} \subset \mc{M}^{k,\alpha}$ of $g_0$ in the $C^{k,\alpha}$-topology such that for any $g \in \mc{U}$, there exists a unique $C^{k+1,\alpha}$-diffeomorphism $\psi$ such that $\psi^*g$ is solenoidal with respect to $g_0$. Moreover, the following map is $C^2$ 
\[C^{k,\alpha}(M;S^2T^*M)\times C^{k+3,\alpha}(M;S^2T^*M) \to \mc{D}_0^{k+1,\alpha}(M), \quad 
(g,g_0)\mapsto \psi.\]
\end{lemma}

\begin{remark}
The previous Lemma is not stated exactly this way in \cite[Lemma 4.1]{Guillarmou-Lefeuvre-18}. Indeed, the proof assumes that $g_0$ is smooth and fixed. However, inspecting the proof, it readily applies to $g_0\in C^{k+3,\alpha}$ and the implicit function theorem used in that proof shows the regularity of $\psi$
with respect to $g_0$. We do not include the proof of these details in order not to burden the discussion.

We also see that we need to use to $C^{k,\alpha}$ regularity for $\alpha \neq 0,1$ instead of $C^k$: this is due to the fact that the pseudodifferential operator inverting the linearization $D^*_{g_0}D_{g_0}$ that arises naturally in the proof of this lemma (see \cite[Lemma 4.1]{Guillarmou-Lefeuvre-18}) act on these spaces but on $C^k$, for $k \in \N$. Instead, one would have to resort to Zygmund spaces $C^k_*$. We refer to \cite[Appendix A]{Taylor-91} for further details. \end{remark}

\subsection{Thermodynamic formalism}

\label{ssection:thermodynamic-formalism}

Let $f$ be a Hölder-continuous function on $S_{g_0}M$. 
We recall that its \emph{pressure} \cite[Theorem 9.10]{Walters-82} is defined by:
\be
\label{equation:pression}
\mathbf{P}(f) := \sup_{\mu \in \mathfrak{M}_{\rm inv}} \left({\bf h}_\mu(\varphi_1^{g_0}) + \int_{S_{g_0}M} f\, d\mu\right),
\ee
where $\mathfrak{M}_{\rm inv}$ denotes the set of invariant (by the flow $\varphi^{g_0}$) Borel probability measures and ${\bf h}_\mu(\varphi_1^{g_0})$ is the metric entropy of the flow $\varphi_1^{g_0}$ at time $1$. It is actually sufficient to restrict the $\sup$ to ergodic measures $\mathfrak{M}_{\rm inv,erg}$ \cite[Corollary 9.10.1]{Walters-82}. Since the flow is Anosov, the supremum is always achieved for a unique invariant ergodic measure $\mu_f$ 
(by \cite[Theorem 3.3.]{BR75}, see also \cite[Theorem 9.3.4]{Fisher-Hasselblatt} and the following discussion therein) called the \emph{equilibrium state} of $f$, and 
\begin{equation}\label{equalityequilibrium}
\mu_f=\mu_{f'} \Longrightarrow f-f'=X_{g_0}u +c \textrm{ for some }u\textrm{ H\"older}\textrm{ and }c\textrm{ is constant},
\end{equation} 
see \cite[Theorem 9.3.16]{Fisher-Hasselblatt}. The measure $\mu_f$ is also mixing and positive on open sets which rules out the possibility of a finite combination of Dirac measures supported on a finite number of closed orbits. Moreover $\mu_f$ can be written as an infinite weighted sum of Dirac masses $\delta_{g_0}(c_j)$ supported over the geodesics $\gamma_{g_0}(c_j)$, where $c_j\in \mc{C}$ are the primitive classes (see \cite{Parry} for the case ${\bf P}(f)\geq 0$ or \cite[Theorem 9.17]{PPS} for the general case). For example when ${\bf P}(f)\geq 0$, 
\begin{equation}\label{muf}
\int u \, d\mu_f=\lim_{T\to \infty} \frac{1}{N(T,f)}\sum_{\left\{j | L_{g_0}(c_j)\in [T,T+1]\right\}}e^{\int_{\gamma_{g_0}(c_j)}f} \int_{\gamma_{g_0}(c_j)}u,
\end{equation}
where $N(T,f):=\sum_{j, L_{g_0}(c_j)\in [T,T+1]}L_{g_0}(c_j)e^{\int_{\gamma_{g_0}(c_j)}f}$.
When $f = 0$, this is the measure of maximal entropy, also called the \emph{Bowen-Margulis measure} $\mu_{g_0}^{\rm BM}$; in that case $\mathbf{P}(0) = {\bf h}_{\rm top}(\varphi_1^{g_0})$ is the topological entropy of the flow. When $f =-J_{g_0}^u$, where $J_{g_0}^u : x \mapsto \partial_t (|\det d\varphi^g_t(x)|_{E_u(x)})|_{t=0}$ is the unstable Jacobian, one obtains the Liouville measure $\mu_{g_0}^{\rm L}$ induced by the metric $g_0$; in that case, $\mathbf{P}(-J_{g_0}^u) = 0$. If we fix an exponent of Hölder regularity $\nu > 0$, then the map $C^\nu(S_{g_0}M) \ni f \mapsto \mathbf{P}(f)$ is real analytic (see \cite[Corollary 7.10]{Ruelle-04} for discrete systems and Parry-Pollicott \cite[Proposition 4.7]{Parry-Pollicott-90} for flows).

\subsection{Geodesic stretch}

\label{section:stretch}

We refer to \cite{Croke-Fathi,Knieper-95} for the original definition of this notion.

\subsubsection{Structural stability and time reparametrization}

We fix a smooth metric $g_0 \in \mc{M}$ with Anosov geodesic flow and we view the geodesic 
flow and vector fields of any metric $g$ close to $g_0$ as living on the unit tangent bundle $S_{g_0}M$ of $g_0$ by simply pulling them back by the diffeomorphism 
\[ (x,v)\in S_{g_0}M\to \big(x,\frac{v}{|v|_{g}}\big)\in S_gM.\] 
We fix some constant $k \geq 2$ and $\alpha \in (0,1)$. There exists a regularity parameter $\nu>0$ and a neighborhood $\mc{U} \subset \mc{M}^{k,\alpha}$ of $g_0$ such that, by the structural stability theorem (\cite[Appendix A]{DeLaLlave-Marco-Moryon-86} or \cite[Proposition 2.2]{Katoketal} for the H\"older regularity case), for any $g \in \mc{U}$, there exists a $C^\nu$ H\"older homeomorphism $\psi_g : S_{g_0}M \rightarrow S_{g_0}M$, differentiable in the flow direction, which is an orbit conjugacy i.e. such that 
\be
\label{equation:conjugaison}
d\psi_g(z) X_{g_0}(z) = a_g(z) X_{g}(\psi_g(z)), \qquad  \forall z \in S_{g_0}M,
\ee
where $a_g$ is in $C^{\nu}(S_{g_0}M)$. Moreover, the map 
\[\mc{U} \ni g \mapsto (a_g,\psi_g) \in C^\nu(S_{g_0}M) \times C^\nu(S_{g_0}M,S_{g_0}M)\] 
is $C^{k-2}$ and $\psi_g$ is homotopic to the identity. For the proof of Theorem \ref{mainth3}, 
we will also need the continuity of $a_g=a_{g_0,g}$ and of its $g$-derivatives of order $\ell\leq k-2$ as a function of the base metric $g_0$. This continuity follows essentially 
from the proof of \cite[Proposition 2.2]{Katoketal}, we give a proof of this fact in Proposition \ref{stabilityth} of the Appendix.

Note that neither $a_g$ nor $\psi_g$ are unique but $a_g$ is unique up to a coboundary and in all the following paragraphs, adding a coboundary to $a_g$ will not affect the results. 
From (\ref{equation:conjugaison}), we obtain that for $t \in \R$, $z \in S_{g_0}M$, 
$\varphi^g_{\kappa_{a_g}(z,t)} (\psi_g(z)) = \psi_g(\varphi^{g_0}_{t}(z))$
with:
\be
\label{equation:reparametrisation}
\kappa_{a_g}(z,t) = \int_0^t a_g(\varphi^{g_0}_s(z))\, ds.
\ee
If $c \in \mathcal{C}$ is a free homotopy class, then one has:
\be
\label{equation:longueurs}
L_g(c) = \int_0^{L_{g_0}(c)} a_g(\varphi_s^{g_0}(z))\, ds,
\ee
for any $z \in \gamma_{g_0}(c)$, the unique $g_0$-closed geodesic in $c$.

\subsubsection{Definition of the geodesic stretch}

\label{ssection:geodesic-stretch}

We denote by $\wt{M}$ the universal cover of $M$. Given a metric $g \in \mc{M}$ on $M$, we denote by $\wt{g}$ its lift to the universal cover. Given two metrics $g_1$ and $g_2$ on $M$, there exists a constant $c > 0$ such that $c^{-1}g_1 \leq g_2 \leq cg_1$. This implies that any $\wt{g_1}$-geodesic is a quasi-geodesic for $\wt{g_2}$. We now assume that the two metrics $g_1,g_2$ are Anosov on $M$. The \emph{ideal} (or \emph{visual}) boundary $\partial_\infty \wt{M}$ is independent of the choice of $g$ and is naturally endowed with the structure of a topological manifold (see Appendix \ref{appendix:conjugacy}) whose regularity inherits that of the foliation (i.e. it is at least H\"older continuous and is $C^{2-\eps}$ for any $\eps > 0$ on negatively-curved surfaces by \cite{Hurder-Katok-90}). In negative curvature, we refer to \cite[Chapter H.3]{Bridson-Haefliger-99} and \cite{Knieper-survey} for further details. For the general Anosov case, we refer to \cite{Knieper-12} and the Appendix \ref{appendix:conjugacy} of the present paper.

We denote by $\mc{G}_g := S_{\til{g}}\wt{M}/\sim$ (where $z \sim z'$ if and only if there exists a time $t \in \R$ such that $\varphi_t(z) =z'$) the set of $g$-geodesics on $\wt{M}$: this is smooth $2n$-dimensional manifold. Moreover, there exists a Hölder continuous homeomorphism $\Phi_g : \mc{G}_g \rightarrow \partial_\infty \wt{M} \times \partial_\infty \wt{M} \setminus \Delta$, where $\Delta$ is the diagonal in $\partial_\infty \wt{M} \times \partial_\infty \wt{M}$. Given a point $z \in S_{\til{g}}\wt{M}$, we will denote by $z_+, z_- \in \partial_\infty \wt{M}$ the points (resp. in the future and in the past) on the boundary at infinity of the geodesic generated by $z$.

We now consider a fixed metric $g_0$ on $M$ and a metric $g$ in a neighborhood of $g_0$. If $\psi_g$ denotes an orbit equivalence between the two geodesic flows, then $\psi_g$ induces a homeomorphism $\Psi_g : \mc{G}_{g_0} \rightarrow \mc{G}_g$. The map 
\[
\Phi_g \circ \Psi_g \circ \Phi_{g_0}^{-1} : \partial_\infty \wt{M} \times \partial_\infty \wt{M} \setminus \Delta \rightarrow \partial_\infty \wt{M} \times \partial_\infty \wt{M} \setminus \Delta
\]
is nothing but the identity.


Given $z = (x,v) \in S_{g_0}M$, we denote by $c_{g_0}(z) : t \mapsto c_{g_0}(z,t)\in M$ the unique geodesic\footnote{For the sake of simplicity, we identify the geodesic and its arc-length parametrization.} such that $c_{g_0}(z,0)=x, \dot{c}_{g_0}(z,0)=v$. We consider $\wt{c}_{g_0}(z)$, a lift of $c_{g_0}(z)$ to the universal cover $\wt{M}$ and introduce the function
\[
b : S_{g_0}M \times \R \rightarrow \R, \qquad b(z,t) := d_{\wt{g}}(\wt{c}_{g_0}(z,0),\wt{c}_{g_0}(z,t)),
\]
which computes the $\til{g}$-distance between the endpoints of the $\wt{g_0}$-geodesic joing 
$\wt{c}_{g_0}(z,0)$ to $\wt{c}_{g_0}(z,t)$. It is an immediate consequence of the triangle inequality that $(z,t) \mapsto b(z,t)$ is a subadditive cocycle for the geodesic flow $\varphi^{g_0}$, that is:
\[
b(z,t+s) \leq b(z,t) + b(\varphi_t^{g_0}(z),s), \quad \forall z \in S_{g_0}M, \forall t,s \in \R
\]
As a consequence, by the subadditive ergodic theorem (see \cite[Theorem 10.1]{Walters-82} for instance), we obtain the following

\begin{lemma}
Let $\mu$ be an invariant probability measure for the flow $\varphi_t^{g_0}$. Then, the quantity
\[
I_\mu(g_0,g,z) := \lim_{t \rightarrow +\infty} b(z,t)/t
\]
exists for $\mu$-almost every $z \in S_{g_0}M$, $I_\mu(g_0,g,\cdot) \in L^1(S_{g_0}M,d\mu)$ and this function is invariant by the flow $\varphi_t^{g_0}$. \end{lemma}

We define the \emph{geodesic stretch of the metric $g$, relative to the metric $g_0$, with respect to the measure $\mu$} by:
\[
I_\mu(g_0,g) := \int_{S_{g_0}M} I_\mu(g_0,g,z)\,d\mu(z).
\]
When the measure $\mu$ in the previous definition is ergodic, the function $I_\mu(g_0,g,\cdot)$ is thus ($\mu$-almost everywhere) equal to the constant $I_\mu(g_0,g)$. We recall that $\delta_{g_0}(c)$ is the normalized measure supported on $\gamma_{g_0}(c)$, that is:
\[
\delta_{g_0}(c) : u \mapsto \dfrac{1}{L_{g_0}(c)} \int_0^{L_{g_0}(c)} u(\varphi^{g_0}_t(z))\,dt.
\]
We can actually describe the stretch using the time reparametrization $a_g$. 

\begin{lemma}
\label{lemma:stretch-ergodic}
Let $\mu$ be an ergodic invariant measure with respect to the flow $\varphi_t^{g_0}$. Then:
\[
I_{\mu}(g_0,g) = \int_{SM_{g_0}} a_g \,d\mu = \lim_{j \rightarrow +\infty} \dfrac{L_{g}(c_j)}{L_{g_0}(c_j)},
\]
where $(c_j)_{j \geq 0} \in \mathcal{C}^\N$ is such that\footnote{The existence of $c_j$ follows from \cite[Theorem 1]{Sigmund}.} $\delta_{g_0}(c_j) \rightharpoonup_{j \rightarrow +\infty} \mu$.
\end{lemma}

\begin{proof}
We first prove the left equality.  Let $\widetilde{M}$
be the universal covering of $M$ and $\Gamma$ the group of deck transformations. Denote as above 
$\widetilde{\psi}_{g} : S_{\wt{g}_0}\widetilde{M} \rightarrow S_{\wt{g}} \wt{M}$ the lift of the conjugacy between the geodesic flow of the metrics $ \wt g$ and $\wt g_0$. Then for all $\gamma \in \Gamma$
$$
\widetilde{\varphi}^g_{\kappa_{a_g}(z,t)} (\widetilde{\psi}_g(z)) = \wt{\psi}_g(\wt{\varphi}^{g_0}_{t}(z)) \;  \; \text{and} \;  \;
\widetilde{\psi}_{g}(\gamma_*{z})= \gamma_* \widetilde{\psi}_{g}(z).
$$
If $\pi: T\widetilde{M} \to \widetilde{M}$
is the canonical projection the function $d_{\wt{g}}(\pi(\widetilde{\psi}_g(z)), \pi (z))$ is $\Gamma$-invariant.
This follows since
\begin{align*}
d_{\wt{g}}(\pi(\widetilde{\psi}_g(\gamma_*z)), \pi (\gamma_*z)) &= d_{\wt{g}}(\pi(\gamma_*\widetilde{\psi}_g(z)), \pi (\gamma_*z))\\
 &= d_{\wt{g}}(\gamma\pi(\widetilde{\psi}_g(z)), \gamma \pi (z)) =d_{\wt{g}}(\pi(\widetilde{\psi}_g(z)), \pi (z)).
\end{align*}
Hence, by the compactness of  $M$ and the continuity of $d_{\wt{g}}(\pi(\widetilde{\psi}_g(z)), \pi (z))$ there is a constant $C>0$ such that
$d_{\wt{g}}(\pi (\widetilde{\psi}_g(z)), \pi (z)) \le C$ for all $z\in S\til{M}$. Using the triangle inequality we obtain
\begin{eqnarray*}
| b(z,t) -\kappa_{a_g}(t,z) | &= & | d_{\wt{g}}( \pi (\wt \varphi^{g_0}_{t}(z)) ,\pi(z))  -d_{\wt{g}}( \pi(\widetilde{\varphi}^g_{\kappa_{a_g}(z,t)} (\widetilde{\psi}_g(z))),  \pi (\widetilde{\psi}_g(z))) |\\
&= & |d_{\wt{g}}( \pi (\wt{\varphi}^{g_0}_{t}(z)),\pi(z))  - d_{\wt{g}}( \pi (\wt{\psi}_g( \wt{\varphi}^{g_0}_{t}(z))),  \pi (\widetilde{\psi}_g(z)) )| \\
&\le& d_{\wt{g}}( \pi (\wt{\varphi}^{g_0}_{t}(z)), \pi (\wt{\psi}_g(\wt{ \varphi}^{g_0}_{t}(z)))   ) + d_{\wt{g}}(\pi (\widetilde{\psi}_g(z)), \pi(z)) \le 2C.
\end{eqnarray*}
This implies, using (\ref{equation:reparametrisation}) that:
\[ 
\lim_{t \rightarrow +\infty} b(z,t)/t 
= \lim_{t \rightarrow + \infty} \kappa_{a_g}(z,t)/t = 
\lim_{t \rightarrow +\infty} \dfrac{1}{t}\int_0^{t} a_g(\varphi_s^{g_0}(z)) \,ds = \int_{S_{g_0}M} a_g\,d\mu,
\] 
for $\mu$-almost every $z \in S_{g_0}M$, by the Birkhoff ergodic Theorem \cite[Theorem 1.14]{Walters-82}. By (\ref{equation:longueurs}) we also have  
\[\int_{S_{g_0}M}a_g\,d\mu_f=\lim_{j\to \infty}\cjg \delta_{g_0}(c_j),a_g\cjd=\lim_{j\to \infty}
\frac{L_{g}(c_j)}{L_{g_0}(c_j)}\]
thus the proof is complete.
\end{proof}

As a consequence, we immediately obtain the

\begin{corollary}
\label{corollary:stretch-lengths-ratio}
Let $g$ belong to a fixed neighborhood $\mc{U}$ of $g_0$ in $\mc{M}^{k,\alpha}$, and assume that for any sequence of primitive free homotopy classes $(c_j)_{j \geq 0} \in \mathcal{C}^\N$ such that $L_{g_0}(c_j) \rightarrow \infty$, one has $\lim_{j\to \infty}L_{g}(c_j)/L_{g_0}(c_j)= 1$. Then, for any equilibrium state $\mu_f$ with respect to $\varphi_t^{g_0}$ associated to some H\"older function $f$, we have $I_{\mu_f}(g_0,g)=1$.
\end{corollary}
Combining this with the results of \cite[Theorem 1]{Guillarmou-Lefeuvre-18}, namely the local rigidity of the marked length spectrum, we also easily obtain: 

\begin{theorem}
\label{theorem:entropy}
Let $(M,g_0)$ be a smooth Riemannian $n$-dimensional manifold with Anosov geodesic flow, topological entropy ${\bf h}_{\rm top}(g_0) = 1$ and assume that its curvature is nonpositive if $n \geq 3$. Then there exists $k \in \N$ large enough depending only on $n$, $\eps > 0$ small enough such that the following holds: there is $C>0$ depending on $g_0$ so that for each $g \in C^{k}(M;S^2T^*M)$ with $\|g-g_0\|_{C^k} \leq \eps$, if 
\[ {\bf h}_{\rm top}(g)=1, \quad \quad
\lim_{j \rightarrow +\infty} \dfrac{L_g(c_j)}{L_{g_0}(c_j)} = 1,
\]
for some sequence $(c_j)_{j \in \N}$ of primitive free homotopy classes such that $\delta_{g_0}(c_j) \rightharpoonup_{j \rightarrow +\infty} \mu^{\rm BM}_{g_0}$, then $g$ is isometric to $g_0$.
\end{theorem}
\begin{proof}
Given a metric $g$, one has by \cite[Theorem 1.2]{Knieper-95}\footnote{In \cite{Knieper-95} the metric is assumed to be negatively curved, but the argument applies also for Anosov flows, as is shown in \cite[Proposition 3.8]{Bridgeman-Canary-Labourie-Sambarino-15}: it corresponds to Proposition \ref{BCLS15} below in the case $f:=1$ and $f'=a_g$.} that 
\begin{equation}
\label{equation:inegalite-stretch}
{\bf h}_{\text{top}}(g) \geq \dfrac{{\bf h}_{\text{top}}(g_0)}{I_{\mu^{\rm BM}_{g_0}}(g_0,g)},
\end{equation}
with equality if and only if $\varphi^{g_0}$ and $\varphi^g$ are, up to a scaling, time-preserving conjugate, that is there exists homeomorphism $\psi$  such that $\psi \circ \varphi_{g_0}^{ct} = \varphi_g^t \circ \psi$ with $c:= {\bf h}_{\text{top}}(g)/{\bf h}_{\text{top}}(g_0)$. 
In particular, restricting to metrics with entropy $1$ one obtains that $I_{\mu^{\rm BM}_{g_0}}(g_0,g) \geq 1$ with equality if and only if the geodesic flows are conjugate, that is if and only if $L_g = L_{g_0}$ (by Livsic theorem). As a consequence, given $g_0, g$ with entropy $1$ such that $L_{g}(c_j)/L_{g_0}(c_j) \rightarrow_{j \rightarrow +\infty} 1$ for some sequence $\delta_{g_0}(c_j) \rightharpoonup_{j \rightarrow +\infty} \mu^{\rm BM}_{g_0}$, we obtain that $I_{\mu^{\rm BM}_{g_0}}(g_0,g)=1$, hence $L_g=L_{g_0}$. If $k \in \N$ was chosen large enough at the beginning, we can then conclude by the local rigidity of the marked length spectrum \cite[Theorem 1]{Guillarmou-Lefeuvre-18}.
\end{proof}

In Theorem \ref{theorem:entropy}, we assume that $g_0$ has entropy $1$. This is actually a harmless assumption insofar as the same result holds true on metric of constant topological entropy ${\bf h}_{{\rm top}}(g)=\la>0$. Recall that by considering $\lambda^2 g_0$ for some constant $\lambda > 0$, the entropy scales as ${\bf h}_{\rm top}(\lambda^2 g_0) = {\bf h}_{\rm top}(g_0)/\lambda$ \cite[Lemma 3.23]{Paternain-99} and we can thus always reduce to the previous case ${\bf h}_{{\rm top}}(g_0) = 1$. We also observe that the previous Theorem implies the local rigidity of the marked length spectrum: if $L_g=L_{g_0}$, then ${\bf h}_{\text{top}}(g_0)={\bf h}_{\text{top}}(g)$ 
because the topological entropy ${\bf h}_{\text{top}}(g)$ is the first pole of the Ruelle zeta function (\cite[Theorem 9.1]{Parry-Pollicott-90})
\[
\zeta_{g}(s) := \prod_{c \in \mathcal{C}} (1-e^{-s L_g(c)}).
\]
We can then apply Theorem \ref{theorem:entropy} to deduce that $g$ is isometric to $g_0$.
We will provide an alternate proof of this fact in the next section without using the proof of \cite{Guillarmou-Lefeuvre-18}.

\section{A functional on the space of metrics}

\label{section:functional}

Given a metric $g$ in a $C^{k,\alpha}$-neighborhood $\mc{U}$ of $g_0$, we define the potential 
\be
\label{equation:potentiel}
V_g := J_{g_0}^u + a_g -1 \in C^\nu(S_{g_0}M)
\ee
for some $\nu>0$.
We remark that $\mc{U} \ni g \mapsto V_g \in C^\nu(S_{g_0}M)$ is $C^{k-2}$ and for $g=g_0$, $V_{g_0} = J_{g_0}^{u}$.
Consider the map $\psi : \mc{M}^{k,\alpha}\to \rr$, defined for $g_0$ a fixed smooth metric with Anosov geodesic flow, by:
\begin{equation}
\label{defofPSi}
\Psi(g):={\bf P}\Big(-J_{g_0}^u-a_g+\int_{S_{g_0}M}a_g\, d\mu_{g_0}^L\Big)={\bf P}(-V_g)+I_{\mu_{g_0}^L}(g_0,g)-1.
 \end{equation}
We also define the maps 
\begin{align} 
F: \mc{M}^{k,\alpha}\to \rr, \quad F(g):={\bf P}(-V_g),\label{defF}\\
\Phi: \mc{M}^{k,\alpha}\to \rr , \quad \Phi(g)=I_{\mu_{g_0}^L}(g_0,g)-1\label{defPhi}.
\end{align}
satisfying $\Psi(g)=F(g)+\Phi(g)$. We note that $\Psi,\Phi, F$ are $C^{k-2}$ by \cite{Contreras-92}. 
We also make the following observation: since ${\bf P}(-J^u_{g_0})=0$ and $a_{g_0}$ is cohomologous to $1$, we have 
$\Psi(g_0)=0$ and 
\begin{equation}\label{inequalitypressure} 
\Phi(g)=-\Big({\bf h}_{\mu_{g_0}^L}(\varphi_1^{g_0})+\int_{S_{g_0}M}(1-J^u_{g_0}-a_g)d\mu_{g_0}^L\Big)\geq -{\bf P}(1-J^u_{g_0}-a_g)=-F(g)
\end{equation}
by using the variational definition \eqref{equation:pression} of the pressure. This shows that for all $g\in \mc{M}^{k,\alpha}$
\[ \Psi(g)\geq \Psi(g_0)=0.\]
Moreover, $\Psi(g)=0$ if and only if the inequality \eqref{inequalitypressure} becomes an equality, which means that $\mu_{g_0}^L$ is the equilibrium measure of $-J^u_{g_0}+1-a_g$. Since $\mu_{g_0}^L$ is also the equilibrium measure associated to $-J^u_{g_0}$, we conclude by \eqref{equalityequilibrium} that $1-a_g$ is cohomologous to a constant, or equivalently $a_g$ is cohomologous to a constant.  
We have thus shown
\begin{lemma}
The map $\Psi$ satisfies $\Psi(g)\geq \Psi(g_0)=0$, and $\Psi(g)=\Psi(g_0)=0$ if and only if $a_g$ is cohomologous to a constant, or equivalently $L_g=\la L_{g_0}$ for some $\la>0$. 
\end{lemma}
The proof of Theorem \ref{theorem:stability} will be a consequence of the fact that 
Taylor expansion of $\Psi$ at $g=g_0$ has leading term given by the Hessian, which turns out to be the variance operator $\Pi_2$ studied before.

\subsection{The proof of Theorem \ref{theorem:stability}} In the following paragraphs, we will compute the derivatives of the map $\Psi,\Phi,F$. As mentioned earlier, they are $C^{k-2}$ by \cite[Theorem C]{Contreras-92} and explicit computations of their derivatives can be found in \cite[Proposition 4.10]{Parry-Pollicott-90} (case of subshift) and \cite{Katok-Knieper-Weiss-91}, \cite{Katok-Knieper-Pollicott-Weiss-90} (case of the topological entropy). The first step in the proof is the following

\begin{proposition}\label{estimeePsi}
The non-negative functional $\Psi:\mc{M}^{k,\alpha}\to \rr^+$ defined in \eqref{defofPSi}
satisfies the following property: 
there is a neighborhood $\mc{U}$ of $g_0$ in $C^{5,\alpha}(M,S^2T^*M)$ and a constant $C_{g_0}$ depending on $g_0$ such that  for all $g\in \mc{U}$  
\[
\Psi(g)\geq \frac{1}{8}\Big(\cjg \Pi_2^{g_0}(g-g_0),(g-g_0)\cjd_{L^2} - \cjg (g-g_0),g_0 \cjd^2_{L^2}\Big)-C_{g_0}\|g-g_0\|^3_{C^{5,\alpha}}.
\]
\end{proposition}
\begin{proof}
We shall compute the Taylor expansion of $\Psi$ at $g=g_0$ to second order.
By \cite[Proposition 4.10]{Parry-Pollicott-90}, we have for $h\in T_{g}\mc{M}^{k,\alpha}$:
\[
dF_g.h = -\int_{S_{g_0}M} da_g.h \, dm_g
\]
where $m_g$ is the equilibrium measure of $-V_g$.
In particular, observe that for $g=g_0$, one has:
\begin{equation}\label{formuladF} 
dF_{g_0}.h = -\int_{S_{g_0}M} da_{g_0}.h \,d \mu^{\rm L}_{g_0}, 
\end{equation}
since $m_{g_0} = \mu^{\rm L}_{g_0}$. Next, we get for $h\in T_{g_0}\mc{M}^{k,\alpha}$
\begin{equation}\label{diffPhi}
d\Phi_{g_0}.h = \int_{S_{g_0}M} da_{g_0}.h \, d\mu^{\rm L}_{g_0}=-dF_{g_0}.h
\end{equation}
thus $d\Psi_{g_0}.h=0$ for all $h\in T_{g_0}\mc{M}^{k,\alpha}$.

Let us next compute the second derivative $d^2\Psi_{g_0}(h,h)$. First, we have 
\[
d^2\Phi_{g_0}=\int_{S_{g_0}M}d^2a_{g_0}(h,h)\, d\mu_{g_0}^L.
\]
Then, by \cite[Proposition 4.11]{Parry-Pollicott-90} we know that
\begin{equation}
\begin{gathered}\label{id}
d^2\mathbf{P}_{-V_{g_0}}(dV_{g_0}.h, dV_{g_0}.h) = \Var_{\mu^{\rm L}_{g_0}}(dV_{g_0}.h-\cjg dV_{g_0}.h,1\cjd)=\cjg \Pi^{g_0}dV_{g_0}.h,dV_{g_0}.h\cjd_{L^2}, \\
d\mathbf{P}_{-V_{g_0}}(dV_{g_0}.h) = \int_{S_{g_0}M} dV_{g_0}.h \, d\mu^{\rm L}_{g_0}
\end{gathered}\end{equation}
where $\Var_{\mu^{\rm L}_{g_0}}(h)$ is the variance defined in \eqref{defvariance}, equal to $\cjg \Pi^{g_0} h,h\cjd_{L^2}$ by \eqref{variancevsPi} and $\Pi^{g_0}1=0$. Therefore,
\[\begin{split}
d^2F_{g_0}(h,h)=& -d{\bf P}_{-V_{g_0}}.d^2V_{g_0}(h,h)+d^2{\bf P}_{-V_{g_0}}(dV_{g_0}.h,dV_{g_0}.h)\\
=& -d{\bf P}_{-V_{g_0}}.d^2a_{g_0}(h,h)+\cjg \Pi^{g_0}da_{g_0}.h,da_{g_0}.h\cjd_{L^2}.
\end{split}
\]
All together, we finally get 
\[d^2\Psi_{g_0}(h,h)=\cjg \Pi^{g_0}da_{g_0}.h,da_{g_0}.h\cjd_{L^2}.\]
To conclude, we claim in Lemma \ref{lemma:differentielle-a} below that $da_{g_0}.h-\frac{1}{2}\pi_2^*h$ is a coboundary so that 
\[d^2\Psi_{g_0}(h,h)=\cjg \Pi^{g_0}\pi_2^*h,\pi_2^*h\cjd_{L^2}=\frac{1}{4}\Big(\cjg \Pi_2^{g_0}h,h\cjd_{L^2}-\cjg h,g_0\cjd_{L^2}^2\Big).\]
The statement of the proposition is then simply the Taylor expansion of $\Psi(g)$ at $g=g_0$, with $h=g-g_0$. (We need the map to be $C^3$ for the Taylor expansion, hence the need of the $C^{5,\alpha}$ regularity since we lose two derivatives as mentioned at the beginning of \S\ref{section:functional}.)
\end{proof}
\begin{lemma}
\label{lemma:differentielle-a}
Consider a smooth deformation $(g_\lambda)_{\lambda \in (-1,1)}$ of $g_0$ inside $\mc{M}^{k,\alpha}$. Then, there exists a Hölder-continuous function $f : S_{g_0}M \rightarrow \R$ such that
\[ \pi_2^*\left(\partial_\lambda g_\lambda|_{\lambda=0}\right) - 2 \partial_\lambda a_\lambda |_{\lambda=0} = X_{g_0}f.\]
\end{lemma}
\begin{proof}
Let $c$ be a fixed free homotopy class, $\gamma_0 \in c$ be the unique closed $g_0$-geodesic in the class $c$, which we parametrize by unit-speed $z_0 : [0,\ell_{g_0}(\gamma_0)] \rightarrow S_{g_0}M$. We define $z_\lambda(s)=\psi_\lambda(z_0(s)) = (\alpha_\lambda(s),\dot{\alpha}_\lambda(s))$ (the dot is the derivative with respect to $s$) where $\psi_\lambda$ is the conjugacy between $g_\lambda$ and $g_0$
: this gives a non-unit-speed parametrization of $\gamma_\lambda$, the unique closed $g_\lambda$-geodesic in $c$. We recall that $\pi : TM \rightarrow M$ is the projection. We obtain using \eqref{equation:conjugaison}
\[ \begin{split}  \int_0^{\ell_{g_0}(\gamma_0)} g_\lambda(\dot{\alpha}_\lambda(s),\dot{\alpha}_\lambda(s)) ds & = \int_0^{\ell_{g_0}(\gamma_0)} g_\lambda\left(\partial_s(\pi \circ z_\lambda(s)),\partial_s(\pi \circ z_\lambda(s))\right) ds \\ 
& = \int_0^{\ell_{g_0}(\gamma_0)} g_\lambda\left(\partial_s(\pi \circ \psi_\lambda \circ z_0(s)),\partial_s(\pi \circ \psi_\lambda \circ z_0(s))\right) ds \\
& = \int_0^{\ell_{g_0}(\gamma_0)} a_\lambda^{2}(z_0(s))\underbrace{g_\lambda\left(d\pi(X_{g_\lambda}(z_\lambda(s))),d\pi(X_{g_\lambda}(z_\lambda(s)))\right)}_{=1} ds\\
& = \int_0^{\ell_{g_0}(\gamma_0)} a_\lambda^{2}(z_0(s)) ds. 
\end{split} \]
Since $s \mapsto \alpha_0(s)$ is a unit-speed geodesic for $g_0$, it is a critical point of the energy functional (with respect to $g_0$). Thus, by differentiating the previous identity with respect to $\lambda$ and evaluating at $\lambda=0$, one obtains:
\[
\int_0^{\ell_{g_0}(\gamma_0)} \partial_\lambda g_\lambda|_{\lambda=0}(\dot{\alpha}_0(s),\dot{\alpha}_0(s)) ds = 2 \int_0^{\ell_{g_0}(\gamma_0)} \partial_\lambda a_\lambda|_{\lambda=0}(z_0(s)) ds.
 \]
As a consequence, $\pi_2^*\left(\partial_\lambda g_\lambda|_{\lambda=0}\right) - 2 \partial_\lambda a_\lambda |_{\lambda=0}$ is a Hölder-continuous function in the kernel of the X-ray transform: by the usual Livsic theorem, there exists a function $f$ (with the same Hölder regularity), differentiable in the flow direction, such that $\pi_2^*\left(\partial_\lambda g_\lambda|_{\lambda=0}\right) - 2 \partial_\lambda a_\lambda |_{\lambda=0} = X_{g_0}f$.
\end{proof}

As a corollary, we get:

\begin{corollary}\label{cor1proofTh1.1}
For $k\geq 5$, $\alpha\in(0,1)$, there is a neighborhood $\mc{U}$ of $g_0$ in $C^{k,\alpha}(M;S^2T^*M)$ and constants $C_{g_0},C_{g_0}'>0$ depending on $g_0$ such that for all $g\in \mc{U}$
\[C_{g_0}\|\pi_{\ker D_{g_0}^*}(g-g_0)\|^2_{H^{-1/2}(M)}\leq \Psi(g)+\frac{1}{4}\cjg (g-g_0),g_0\cjd_{L^2}^2
+C_{g_0}'\|g-g_0\|_{C^{5,\alpha}}^3.\]
There is a neighborhood $\mc{U}'$ of $g_0$ in $C^{k,\alpha}(M;S^2T^*M)$ and a constant $C_{g_0}''>0$ depending on $g_0$ such that for all $g\in \mc{U}'$, there is a diffeomorphism $\psi\in C^{k+1,\alpha}(M)$ such that
\[C_{g_0}\|\psi^*g-g_0\|^2_{H^{-1/2}(M)}\leq \Psi(g)+\frac{1}{4}\cjg (\psi^*g-g_0),g_0\cjd_{L^2}^2
+C_{g_0}''\|\psi^*g-g_0\|_{C^{5,\alpha}}^3\]
\end{corollary}

\begin{proof} The first inequality follows from Proposition \ref{estimeePsi} and Lemma \ref{lemma:positivite-pi2unif}. For the second inequality, we apply the first inequality to $\psi^*g$ where $\psi$ is the diffeomorphism obtained from Lemma \ref{lemma:solenoidal-gauge}, and we use that $\Psi(\psi^*g)=\Psi(g)$. 
\end{proof}

The next step is to control the term $\cjg (\psi^*g-g_0),g_0\cjd_{L^2}$ by the geodesic stretch.
We will show

\begin{proposition}\label{estimfinal1}
There is $k \in \N$ large enough, depending only on $n=\dim M$, such that if $g_0$ is smooth with Anosov geodesic flow, and non-positive curvature in the case $n>2$, there is $C_{g_0}>0$ and $C_n>0$, an open neighborhood 
$\mc{U}$ in $C^{k,\alpha}(M;S^2T^*M)$ of $g_0$, such that for each $g\in \mc{U}$, there is a diffeomorphism $\psi$ satisfying
\[ \begin{gathered}
C_{g_0}\|\psi^*g-g_0\|^2_{H^{-1/2}(M)}\leq {\bf P}\Big(-J_{g_0}^u-a_g+\int_{S_{g_0}M}a_g\, d\mu_{g_0}^L\Big)+C_n\Big(I_{\mu_{g_0}^L}(g_0,g)-1\Big)^2,
\\
  C_{g_0}\|\psi^*g-g_0\|^2_{H^{-1/2}(M)}\leq {\bf P}\Big(-J_{g_0}^u-a_g+\int_{S_{g_0}M}a_g\, d\mu_{g_0}^L\Big)+C_n\Big({\bf P}(-J_{g_0}^u-a_g+1)\Big)^2,\\
  C_{g_0}\|\psi^*g-g_0\|^2_{H^{-1/2}(M)}\leq {\bf P}\Big(-J_{g_0}^u-a_g+\int_{S_{g_0}M}a_g\, d\mu_{g_0}^L\Big)+\Big({\rm Vol}_g(M)-{\rm Vol}_{g_0}(M)\Big)^2.
 \end{gathered}\]
Here $C_{g_0}$ depends on $g_0$ and $C_n$ on $n$ only.
\end{proposition}

\begin{proof}
We write the Taylor expansion of $\Phi(\psi^*g)=\Phi(g)=I_{\mu_{g_0}^L}(g_0,g)-1$ at $g=g_0$: by Lemma \ref{formuladF} and Lemma \ref{lemma:differentielle-a}, 
\[d\Phi_{g_0}.h=\int_{S_{g_0}M}da_{g_0}.h\, d\mu_{g_0}^L=\frac{1}{2}\int_{S_{g_0}M}\pi_2^*h\, d\mu_{g_0}^L=C_n\cjg h,g_0\cjd_{L^2}\]
for some $C_n>0$ depending only on $n=\dim M$. Then: 
\[ \Phi(g)=\Phi(\psi^*g)=C_n\cjg \psi^*g-g_0,g_0\cjd_{L^2}+\mc{O}(\|\psi^*g-g_0\|^2_{C^{5,\alpha}})\]
Combining with Corollary \ref{cor1proofTh1.1}, we obtain 
\begin{equation}\label{estimatewithremainder}
C_{g_0}\|\psi^*g-g_0\|^2_{H^{-1/2}(M)}\leq \Psi(g)+2C_n^{-2}\Phi(g)^2+
C_{g_0}''\|\psi^*g-g_0\|_{C^{5,\alpha}}^3
\end{equation}
if $\|\psi^*g-g_0\|_{C^{5,\alpha}}$ is small enough, which is the case if $\|g-g_0\|_{C^{5,\alpha}}$ is small enough by Lemma \ref{lemma:solenoidal-gauge}. To obtain the first inequality of Proposition \ref{estimfinal1}, we apply Sobolev embedding and interpolation estimates\footnote{The interpolation estimate $\|u\|_{H^c}\leq \|u\|_{H^a}^{t}\|u\|_{H^b}^{1-t}$ for $c=ta+(1-t)b$ is obtained by applying Hadamard three lines theorem to the holomorphic function 
$s\mapsto \sum_{j}(1+\lambda_j)^{s}\cjg u,e_j\cjd_{L^2}^2$ on ${\rm Re}(s)\in [a,b]$ where $e_j$ is an orthonormal basis of eigenfunctions of any positive elliptic self-adjoint differential operator of order $2$ on symmetric tensors and $\la_j$ being the corresponding eigenvalues.} (\cite[Chapter 4]{Tay96}) and get, for some constants $c_{g_0}>0,c'_{g_0}>0$ depending on $g_0$ only,
\[
\|\psi^*g-g_0\|^3_{C^{5,\alpha}} \leq c_{g_0} \|\psi^*g-g_0\|^3_{H^{\frac{n}{2}+5+\alpha'}} \leq c'_{g_0} \|\psi^*g-g_0\|^2_{H^{-1/2}} \|\psi^*g-g_0\|_{H^{k}}, 
\]
if $k >\frac{3}{2}n+16+3\alpha$ and $\alpha'>\alpha$. This means that if $\|\psi^*g-g_0\|_{H^k}$ is small enough, depending on the constants $C_{g_0}, C_{g_0}'',c_{g_0},c'_{g_0}$, one can absorb the
$\|\psi^*g-g_0\|^3_{C^{5,\alpha}}$ term of \eqref{estimatewithremainder} into the left-hand side and get the first inequality of Proposition \ref{estimfinal1}. The smallness of $\|\psi^*g-g_0\|_{H^k}$ is implied by the smallness of $\|g-g_0\|_{C^{k,\alpha}}$ by Lemma \ref{lemma:solenoidal-gauge}.
The same exact argument applies by replacing $\Phi(g)$ by $F(g)$ using that $dF_{g_0}=-d\phi_{g_0}$, this proves the second inequality of Proposition \ref{estimfinal1}. The last inequality is similar since 
\[{\rm Vol}_g(M)-{\rm Vol}_{g_0}(M)=\frac{1}{2}\int_{M}{\rm Tr}_{g_0}(h)\, d{\rm vol}_{g_0}+\mc{O}(\|h\|^2_{C^{5,\alpha}})=\frac{1}{2}\cjg h,g_0\cjd_{L^2}+\mc{O}(\|h\|^2_{C^{5,\alpha}})\]
for $h:=g-g_0$.
The proof is complete.
\end{proof}
To conclude the proof of Theorem \ref{theorem:stability}, we need to estimate $\Phi(g)$ and $F(g)$ in terms of $\mc{L}_\pm(g)$. Recall that (see \cite[Corollary 9.17]{PPS}) 
\[\begin{split}
{\bf P}(-V_g)=&\lim_{T\to \infty}\frac{1}{T}\log \sum_{c\in\mc{C}, L_{g_0}(c)\in [T,T+1]}e^{-\int_{\gamma_{g_0}(c)}V_g}\\
=&\lim_{T\to \infty}\frac{1}{T}\log \sum_{c\in\mc{C}, L_{g_0}(c)\in [T,T+1]}e^{-\int_{\gamma_{g_0}(c)}J^u_{g_0}}e^{L_{g_0}(c)-L_{g}(c)}.
\end{split}\]
Thus, if we order $\mc{C}=(c_j)_{j\in\N}$ by the lengths (i.e. $L_{g_0}(c_j)\geq L_{g_0}(c_{j-1})$), and we define 
\[ \mc{L}_+(g):=\limsup_{j\to \infty}\frac{L_g(c_j)}{L_{g_0}(c_j)}-1, \quad \mc{L}_-(g):=\liminf_{j\to \infty}\frac{L_{g}(c_j)}{L_{g_0}(c_j)}-1,\]
we see that for all $\delta>0$ small, there is $T_0>0$ large so that for all $j$ with 
$L_{g_0}(c_j)\in [T,T+1]$ with $T\geq T_0$:
\[  e^{\min((T+1)(-\mc{L}_+(g)-\delta),T(-\mc{L}_+(g)-\delta))}\leq e^{L_{g_0}(c_j)-L_{g}(c_j)} \leq e^{\max((T+1)(-\mc{L}_-(g)+\delta),T(-\mc{L}_-(g)+\delta))}.\]
We deduce, using ${\bf P}(-V_{g_0})=0$, 
\[  -\mc{L}_+(g)-\delta\leq {\bf P}(-V_g)\leq -\mc{L}_-(g)+\delta.\]
Since $\delta>0$ is arbitrarily small, we obtain $|F(g)|\leq \max(|\mc{L}_+(g)|,|\mc{L}_-(g)|)$. Similarly, Lemma \ref{lemma:stretch-ergodic} shows that 
$|\Phi(g)|\leq \max(|\mc{L}_+(g)|,|\mc{L}_-(g)|)$. So the proof of Theorem \ref{theorem:stability} is complete by combining these bounds with Proposition \ref{estimfinal1} (the right hand side in the first and second inequality of Proposition \ref{estimfinal1} being $F(g)+\Phi(g)+C_n \Phi(g)^2$ and $F(g)+\Psi(g)+C_nF(g)^2$).

\subsection{A submanifold of the space of metrics}

It is quite natural to describe the stretch functional $\Phi$ on the space
\begin{equation}\label{Hkalpha}
\mc{N}^{k,\alpha} := \left\{g \in \mc{M}^{k,\alpha} ~|~ \mathbf{P}(-V_g) = 0 \right\}, 
\end{equation}
and on $\mc{N}^{k,\alpha}_{\rm sol} := \mc{N}^{k,\alpha} \cap \ker D^*_{g_0}$. Indeed, as we shall see, this becomes a strictly convex functional near $g_0\in \mc{N}_{\rm sol}^{k,\alpha}$ when restricted to $\mc{N}_{\rm sol}^{k,\alpha}$. It is possible that the map is strictly convex globally on $\mc{N}_{\rm sol}^{k,\alpha}$, in which case that would prove the global 
rigidity of the marked length spectrum.

Given $g \in \mc{N}^{k,\alpha}$, we denote by $m_g$ the unique equilibrium state for the potential $V_g$. We will also denote $\mc{N}$ for the case where $k=\infty$. First we check that these are (infinite dimensional) manifolds.

\begin{lemma}
\label{lemma:h1}
There exists a neighborhood $\mc{U} \subset \mc{M}^{k,\alpha}$ of $g_0$ such that $\mc{N}^{k,\alpha} \cap \mc{U}$ is a codimension one $C^{k-2}$-submanifold of $\mc{U}$ and $\mc{N}^{k,\alpha}_{\rm sol} \cap \mc{U}$ is a $C^{k-2}$-submanifold of $\mc{U}$. Similarly, there is $\mc{U}\subset \mc{M}$ an open neighborhood so that $\mc{N} \cap \mc{U}$ is a Fr\'echet submanifold of $\mc{M}$.
\end{lemma}
\begin{proof}
To prove this lemma, we will use the notion of differential calculus on Banach manifolds as it is stated in \cite[Chapter 73]{Zeidler-88}. Note that $\mc{M}^{k,\alpha}$ is a smooth Banach manifold and $\mc{N}^{k,\alpha} \subset \mc{M}^{k,\alpha}$ is defined by the implicit equation $F(g)=0$ for \begin{equation}\label{pressiondef}
F: g \mapsto \mathbf{P}(-V_g) \in \R.
\end{equation} 
The map $F$ being $C^{k-2}$, we only need to prove that $dF_{g_0}$ does not vanish by \cite[Theorem 73.C]{Zeidler-88}. This will immediately give that $T_{g_0}\mc{N}^{k,\alpha}=\ker dF_{g_0}$. In order to do so, we need a deformation lemma. For the sake of simplicity, we write the objects $\cdot_{\lambda}$ instead of $\cdot_{g_\lambda}$.

We can now complete the proof of Lemma \ref{lemma:h1}.
We first prove the first part concerning $\mc{N}^{k,\alpha}$. Recall the formula \eqref{formuladF} for $dF_{g_0}$. Using Lemma \ref{lemma:differentielle-a}, one obtains
\begin{equation}\label{dFg_0}
d F_{g_0}.h= -\int_{S_{g_0}M} da_{g_0}.h\, d\mu^{\rm L}_{g_0} = - \dfrac{1}{2} \int_{S_{g_0}M} \pi_2^* h\, d\mu^{\rm L}_{g_0} = - C_n \langle h,g_0 \rangle_{L^2},  
\end{equation}
for some constant $C_n > 0$ depending on $n$. This is obviously surjective and we also obtain:
\[ T_{g_0} \mc{N}^{k,\alpha} = \ker dF_{g_0} = \left\{h \in C^{k,\alpha}(M;S^2T^*M) ~|~ \langle h,g_0 \rangle_{L^2} = 0 \right\} = (\R g_0)^\bot, \]
where the orthogonal is understood with respect to the $L^2$-scalar product.

We now deal with $\mc{N}^{k,\alpha}_{\rm sol}$. First observe that $\ker D^*_{g_0}$ is a closed linear subspace of $\mc{M}^{k,\alpha}$ and thus a smooth submanifold of $\mc{M}^{k,\alpha}$. By \cite[Corollary 73.50]{Zeidler-88}, it is sufficient to prove that $\ker D^*_{g_0}$ and $\mc{N}^{k,\alpha}$ are transverse at $g_0$. But observe that $g_0 \in \ker D^*_{g_0} \simeq T_{g_0} \ker D^*_{g_0}$ and thus
\[ T_{g_0} \ker D^*_{g_0} + T_{g_0}\mc{N}^{k,\alpha} = T_{g_0} \mc{M}^{k,\alpha}, \]
showing transversality. 

The case of $\mc{N}$ follows directly from the Nash-Moser theorem: $F$ is obviously a smooth tame map from $C^\infty(M;S^2T^*M)$ to $\R$, moreover $dF_g$ has a right inverse $H_g$ since\footnote{Note that $2da_g.g$ is cohomologous to $a_g$, as can be seen by differentiating $L_g(c)=\int_{\gamma_{g_0}(c)}a_g$ and using Livsic theorem.} 
\[dF_g.g= -\int_{S_{g_0}M}da_g.g\, dm_g=-\frac{1}{2}\int_{S_{g_0}M} a_g\, dm_g=
-\frac{1}{2}I_{m_g}(g_0,g)
\]
one can take $H_g.1:=-2g/I_{m_g}(g_0,g)$. The family of right inverse: $g\mapsto H_g$ is smooth 
since $g\mapsto a_g$ and $g\mapsto m_g$ are smooth by \cite[Theorem C]{Contreras-92}, and it is clearly also tame thus we can apply directly \cite[Theorem 1.1.3, page 172]{Hamilton} to deduce that $F$ has a smooth tame right inverse, which shows that $\mc{N}$ is a Fr\'echet submanifold.  
\end{proof}

We remark that if $L_g=L_{g_0}$, then $a_g$ is cohomologous to $1$, so ${\bf P}(-V_{g})={\bf P}(-V_{g_0})=0$ in that case, which means that $g\in \mc{N}^{k,\alpha}$. From the second inequality in Proposition \ref{estimfinal1}, we obtain: 
\begin{corollary}\label{boundsonN}
Let $g_0$ be a smooth metric with Anosov geodesic flow, with non-positive curvature if $n>2$. There is $C_{g_0}>0$, a neighborhood $\mc{U}\subset \mc{N}^{k,\alpha}$ such that for all $g\in\mc{U}$, there is a diffeomorphism 
$\psi\in \mc{D}_0^{k+1,\alpha}$ so that 
\[ C_{g_0}\|\psi^*g-g_0\|^2_{H^{-1/2}(M)}\leq I_{\mu_{g_0}^L}(g_0,g)-1\]
\end{corollary}

As suggested by this estimate, the functional $\Phi$ turns out to be strictly convex near $g_0$ when restricted on $\mc{N}_{\rm sol}^{k,\alpha}$. First, one has for $h\in T_{g_0}\mc{N}^{k,\alpha}$
\[ d\Phi_{g_0}.h=-dF_{g_0}.h=0\]
so that $\Phi:\mc{N}^{k,\alpha}\to \rr$ has a critical point at $g_0$. For the second derivative at $g_0$, 
the same computation as in the previous section easily gives

\begin{lemma}
\label{lemma:convexite}
The map $\Phi:\mc{N}^{k,\alpha}_{\rm sol} \to \R$ is strictly convex at $g_0$ and there is $C>0$ such that 
\[d^2\Phi_{g_0}(h,h) = \frac{1}{4} \langle \Pi_2^{g_0} h, h \rangle \geq C \|h\|^2_{H^{-1/2}(M)}\] 
for all $h \in T_{g_0}\mc{N}^{k,\alpha}_{\rm sol}$.
\end{lemma}

\begin{proof}
The proof follows exactly that of Proposition \ref{estimeePsi}, using $T_{g_0} \mc{N}^{k,\alpha} = (\R g_0)^\bot$.
\end{proof}

\subsection{The pressure metric on the space of negatively curved metrics}

The results of this paragraph are stated in negative curvature but it is very likely that one could relax the assumption to the Anosov case. Again, the only obstruction for the moment is that it is still not known whether the X-ray transform $I_2$ (hence the operator $\Pi_2$) is injective on solenoidal tensors in the Anosov case when $\mathrm{dim}(M) \geq 3$.

\subsubsection{Definition of the pressure metric using the variance}\label{sec:pressure metric}
On $\mc{M}^-$, the cone of smooth negatively-curved metrics, we introduce the non-negative symmetric bilinear form
\begin{equation}\label{defG} 
G_g(h_1,h_2) := \langle \Pi^{g}_2 h_1, h_2 \rangle_{L^2(M,d\vol_g)}, 
\end{equation}
defined for $g \in \mc{M}$, $h_j\in T_g \mc{M} \simeq C^{\infty}(M; S^2T^*M)$. It is nondegenerate on $T_g \mc{M} \cap \ker D^*_g$, namely $G_g(h,h) \geq C_g \|h\|^2_{H^{-1/2}}$ by Lemma \ref{lemma:positivite-pi2unif} and the constant  $C_g$ turns out to be locally uniform for $g$ near a given metric $g_0$. 
Combining these facts, we obtain 
\begin{proposition}\label{metricG}
Let $g_0\in \mc{M}^-$, then the bilinear form $G$ defined in \eqref{defG} produces a Riemannian metric on the quotient space $\mc{M}^-/\mc{D}_0$ near the class $[g_0]$, where $\mc{M}^-/\mc{D}_0$ is identified with the slice $\mc{S}$ passing through $g_0$ as in \eqref{sliceS}.
\end{proposition}

\begin{proof}
It suffices to show that $G$ is non-degenerate on $T\mc{S}$. Let $h\in T_{g}\mc{S}$ and assume that $G_g(h,h)=0$. We can write $h=\mc{L}_{V}g+h'$ where $D_g^*h'=0$ and $V$ is a smooth vector field and $\mc{L}_V$ the Lie derivative with respect to $V$. 
By Lemma \ref{lemma:positivite-pi2} we obtain $0=G_g(h,h)\geq C\|h'\|_{H^{-1/2}}$. Thus $h=\mc{L}_Vg$, but we also know that $T_{g}\mc{S}\cap \{\mc{L}_Vg ~|~ V\in C^\infty(M;T^*M)\}=\{0\}$ since $\mc{S}$ is a slice. Therefore $h=0$.
\end{proof}

\subsubsection{Definition using the intersection number} 
In this paragraph, we want to relate the pressure metric previously introduced to some renormalized intersection numbers involving some well-chosen potentials. This will be needed to show that the pressure metric coincides with the (a multiple of) Weil-Petersson metric in the case where $M$ is a surface and one restricts to hyperbolic metrics. This also makes a relation with recent work of \cite{Bridgeman-Canary-Labourie-Sambarino-15}.

Let us assume that $g$ is in a fixed $C^2$-neighborhood of $g_0$. Since $J^u_{g_0} > 0$, we obtain that $V_g = J^u_{g_0} + a_g-1 > 0$ if $g$ is close enough to $g_0$. By \cite[Lemma 2.4]{Sambarino-14}, there exists a unique constant ${\bf h}_{V_g} \in \R$ such that $\mathbf{P}(-{\bf h}_{V_g} V_g)= 0$. In particular, $\mc{N}$ coincides in a neighborhood of $g_0$ with the set $\left\{ g \in \mc{M} ~|~ {\bf h}_{V_g} = 1\right\}$. One can express the constant ${\bf h}_{V_g}$ as ${\bf h}_{V_g} = {\bf h}_{\mathrm{top}}(\varphi_t^{g_0,V_g})$, where $\varphi_t^{g_0,V_g}$ is a time-reparametrization of the geodesic flow of $g_0$ (see \cite[Section 3.1.1]{Bridgeman-Canary-Labourie-Sambarino-15}). More precisely, given $f \in C^\nu(S_{g_0}M)$ a Hölder-continuous positive function on $S_{g_0}M$, we introduce ${\bf h}_f$ to be the unique real number such that $\mathbf{P}(-{\bf h}_f f)=0$ and we set:
\[
S_{g_0}M \times \R \ni (z,t) \mapsto \kappa_f(z,t) := \int_0^t f(\varphi^{g_0}_s(z)) \,ds.
\]
For a fixed $z\in S_{g_0}M$, this is a homeomorphism on $\R$ and thus allows to define:
\begin{equation}
\label{equation:reparametrisation2}
\varphi_{\kappa_f(z,t)}^{g_0,f}(z) := \varphi^{g_0}_t(z).
\end{equation}
We now follow the approach of \cite[Section 3.4.1]{Bridgeman-Canary-Labourie-Sambarino-15}. Given two Hölder-continous functions $f,f' \in C^\nu(S_{g_0}M)$ such that $f > 0$, one can define an \emph{intersection number} \cite[Eq. (13)]{Bridgeman-Canary-Labourie-Sambarino-15} 
\[
\mathbf{I}_{g_0}(f,f') := \dfrac{\int_{S_{g_0}M} f' \, d\mu_{- {\bf h}_f f}}{\int_{S_{g_0}M} f \,d\mu_{- {\bf h}_f f}}
\]
where $d\mu_{- {\bf h}_f f}$ is the equilibrium measure for the potential $-{\bf h}_f f$. We have the following result, which follows from \cite[Proposition 3.8]{Bridgeman-Canary-Labourie-Sambarino-15} stated for Anosov flows on compact metric spaces:

\begin{proposition}[Bridgeman-Canary-Labourie-Sambarino \cite{Bridgeman-Canary-Labourie-Sambarino-15}]\label{BCLS15}
Let $f, f' : S_{g_0}M \rightarrow \R_+$ be two Hölder-continuous positive functions. Then:
\[ \mathbf{J}_{g_0}(f,f') := \dfrac{{\bf h}_{f'}}{{\bf h}_f}  \mathbf{I}_{g_0}(f,f') \geq 1 \]
with equality if and only if ${\bf h_f} f$ and ${\bf h_{f'}} f'$ are cohomologous for the geodesic flow $\varphi_t^{g_0}$ of $g_0$. The quantity $\mathbf{J}_{g_0}(f,f')$ is called the \emph{renormalized intersection number}.
\end{proposition}

We apply the previous proposition with $f := J^u_{g_0}$ (then ${\bf h}_{J^u_{g_0}} = 1$) and $f' := V_g$. Without assuming that $g \in \mc{N}$ (that is we do not necessarily assume that ${\bf h}_{V_g}=1$), we have
\[
\begin{split}
\mathbf{J}_{g_0}(J^u_{g_0},V_g) & = {\bf h}_{V_g} \mathbf{I}_{g_0}(J^u_{g_0},V_g)  = {\bf h}_{V_g}  \dfrac{\int_{S_{g_0}M} (J^u_{g_0} + a_g - \mathbf{1}) \, d\mu^{\rm L}_{g_0}}{\int_{S_{g_0}M} J^u_{g_0} d\mu^{\rm L}_{g_0}} \\
& = {\bf h}_{V_g}  \dfrac{{\bf h}_L(g_0) + I_{\mu^{\rm L}_{g_0}}(g_0,g) -1}{{\bf h}_L(g_0)} \geq 1
\end{split}
\]
where ${\bf h}_L(g_0)$ is the entropy of Liouville measure for $g_0$.
In the specific case where $g \in \mc{N}$, ${\bf h}_{V_g}=1$ and we find that $I_{\mu^{\rm L}_{g_0}}(g_0,g) \geq 1$ with equality if and only if $a_g$ is cohomologous to $1$, 
that is if and only if $L_g=L_{g_0}$, or alternatively if and only if $\varphi^g$ and $\varphi^{g_0}$ are time-preserving conjugate. This computation holds as long as $J^u_{g_0} + a_g-1 > 0$ (which is true in a $C^2$-neighborhood of $g_0$).

In particular, on $\mc{N}$, we have the linear relation 
\[
\mathbf{J}_{g_0}(J^u_{g_0},V_g) = 1+ \dfrac{I_{\mu^{\rm L}_{g_0}}(g_0,g) -1}{{\bf h}_L(g_0)}.
\]
In the notations of \cite[Proposition 3.11]{Bridgeman-Canary-Labourie-Sambarino-15}, the second derivative computed for the family $(g_\lambda)_{\lambda \in (-1,1)} \in \mc{N}$ is
\begin{equation}\label{egaliteHessianJ}
\pl_{\la}^2 \mathbf{J}_{g_0}(J^u_{g_0},V_{g_\lambda}) |_{\lambda = 0} = \dfrac{1}{{\bf h}_L(g_0)} \pl_{\la}^2 I_{\mu^{\rm L}_{g_0}}(g_0,g_\lambda) |_{\lambda = 0} = \dfrac{ \langle \Pi_2^{g_0} \dot{g}_0, \dot{g}_0 \rangle}{4 {\bf h}_L(g_0)}
\end{equation}
and is called the \emph{pressure form}. When considering a slice transverse to the $\mc{D}_0$ action on $\mc{N}$, it induces a metric called the \emph{pressure metric} by Lemma \ref{lemma:positivite-pi2}. To summarize:
\begin{lemma}
\label{lemma:lien-metrique}
Given a smooth metric $g_0$, the metric $G_{g_0}$ restricted to $\mc{N}$ can be obtained from the renormalized 
intersection number by
\[
G_{g_0}(h,h) = 4 {\bf h}_L(g_0) \pl_\la^2 \mathbf{J}_{g_0}(J^u_{g_0},V_{g_\lambda})|_{\lambda = 0}
\]
where $(g_\lambda)_{\lambda \in (-1,1)}$ is any family of metrics such that $g_\lambda \in \mc{N}$ and $\dot{g}_0 = h\in T_{g_0}\mc{N}$.
\end{lemma}

\subsubsection{Link with the Weil-Petersson metric}

We now assume that $M=S$ is an orientable surface of genus $\geq 2$. Let $\mc{T}(S)$ be the Teichm\"uller space of $S$. We show that the pressure metric coincides with the (a multiple of) Weil-Petersson metric in restriction to $\mc{T}(S)$. We fix a hyperbolic metric $g_0$. Given $\eta, \rho \in \mc{T}(S)$ and $g_\eta,g_\rho$ the associated hyperbolic metrics, since $\mc{T}(S)$ is connected (indeed a ball in $\cc^{3({\rm genus}(M)-1)}$) there is topological conjugacy between $g_\eta,g_\rho$ and $g_0$ and one can defined the time rescaling $a_{g_\eta}$ and $a_{g_{\rho}}$ by using a path of hyperbolic metrics relating $g_0$ to $g_\eta$ or to $g_\rho$.
The intersection number is defined as
\[ \mathbf{I}(\eta,\rho) := \mathbf{I}_{g_0}(a_{g_\eta},a_{g_\rho})= \frac{\int_{S_{g_0}M}a_{g_\rho}d\mu_\eta}{\int_{S_{g_0}M}a_{g_\eta}d\mu_\eta}\]
where $[g_\eta]=\eta, [g_\rho]=\rho$ and $\mu_\eta$ is the equilibrium state of $-{\bf h}_{a_{g_\eta}}a_{g_{\eta}}$. Note that ${\bf h}_{a_{g_\eta}}={\bf h}_{\rm top}(\varphi^{g_0,a_\eta}_t)=1$ since $\varphi^{g_0,a_\eta}$ is conjugate to the geodesic flow of $g_\eta$, which in turn has constant curvature, and by \cite[Lemma 2.4]{Sambarino-14}, $a_{g_\eta}d\mu_{\eta}/\int_{S_{g_0}M}a_{g_\eta}d\mu_{\eta}$ is the measure of maximal entropy of the flow $\varphi^{g_0,a_\eta}_t$, thus also the normalized Liouville measure of $g_\eta$ (viewed on $S_{g_0}M$). This number $\mathbf{I}(\eta,\rho)$ is in fact \emph{independent of $g_0$} as it  can alternatively be written 
\[\mathbf{I}(\eta,\rho)=\lim_{T\to \infty}\frac{1}{N_T(\eta)}\sum_{c\in \mc{C}, L_{g_\eta}(c)\leq T}\frac{L_{g_\rho}(c)}{L_{g_\eta}(c)}\]
where $N_T=\sharp \{c\in \mc{C} \, | L_{g_\eta}(c)\leq T\}$ (see \cite[Proof of Th. 4.3]{Bridgeman-Canary-Sambarino-18}). 
In particular, taking $g_0 = g_\eta$, one has 
\[
 \mathbf{I}(\eta,\rho) = I_{\mu^{\rm L}_{g_\eta}}(g_\eta,g_\rho).
\]
As explained in \cite[Theorem 4.3]{Bridgeman-Canary-Sambarino-18}, up to a normalization constant 
$c_0$ depending on the genus only, the Weil-Petersson metric on $\mc{T}(S)$ is equal to 
\begin{equation}\label{WP}
\|h\|_{\rm WP}^2= c_0 \pl_{\la}^2  \mathbf{I}(\eta,\eta_\lambda)|_{\lambda=0} = c_0 \pl_{\la}^2 I_{\mu^{\rm L}_{g_\eta}}(g_\eta,g_{\eta_\lambda})|_{\lambda=0},
\end{equation}
where $\dot{\eta}_0=h$ and $(g_{\eta_\lambda})_{\lambda \in (-1,1)}$ is a family of hyperbolic metrics such that $[g_{\eta_\lambda}] = \eta_\lambda$, $\eta=\eta_0 = [g_0]$. This fact follows from combined works of Thurston, Wolpert \cite{Wolpert} and Mc Mullen \cite{McMullen}:  
the length of a random geodesic $\gamma$ on $(S,g_0)$ with respect to 
$g_{\eta_\lambda}$ has a local minimum at $\lambda=0$ and the Hessian is positive definite (Thurston), is equal to the Weil-Petersson norm squared of $\dot{g}$ (Wolpert \cite{Wolpert, Fathi-Flaminio}) 
and is given by a variance (Mc Mullen \cite{McMullen}); here random means equidistributed with respect to the Liouville measure of $g_0$.
 We can check that the metric $G$ also corresponds to this metric:
\begin{proposition}
The metric $G$ on $\mc{T}(S)$ is a multiple of the Weil-Petersson metric.
\end{proposition}
\begin{proof}
This follows directly from \eqref{egaliteHessianJ}, \eqref{WP} and the fact that ${\bf h}_L(g_\eta)=1$ if $g_\eta$ has curvature $-1$.
\end{proof}

\begin{remark}
We notice that the positivity of the metric in the case of Teichm\"uller space follows only from some convexity argument in finite dimension. In the case of general metrics with negative curvature, the elliptic estimate of Lemma \ref{lemma:positivite-pi2} on the variance is much less obvious due to the infinite dimensionality of the space. As it turns out, this is the key for the local rigidity in our results.
\end{remark}

\section{Uniform elliptic estimates on $\Pi_2$}

\label{section:continuite}

In this section, we prove that the operator $\Pi_2^g \in\Psi^{-1}(M;S^2T^*M)$ depends continuously on $g$. Let $\mc{M}^{\rm An}$ be the space of smooth Riemannian metrics with Anosov geodesic flow. 

\begin{proposition}
\label{proposition:pi2-continuite}
The map $\mc{M}^{\rm An} \ni g \mapsto \Pi_2^g \in \Psi^{-1}(M;S^2T^*M)$ is continuous, when $\Psi^{-1}(M;S^2T^*M)$ is equipped with its topology of Fr\'echet spaces. 
\end{proposition}

Recall that the Fr\'echet topology was introduced at the beginning of \S \ref{ssection:microlocal}. We fix a metric $g_0$ and we work in a neighborhood $\mc{U}$ of $g_0$ in the $C^\infty$ topology. In particular, we will always assume that this neighborhood $\mc{U}$ is small enough so that any $g \in \mc{U}$ has an Anosov geodesic flow that is orbit-conjugated to that of $g_0$ by structural stability. We will also see the geodesic flows $(\varphi_t^g)_{t \in \R}$ as acting on the unit bundle $SM:=S_{g_0}M$ for $g_0$ by using the natural identification $S_gM\to S_{g_0}M$ obtained by scaling in the fibers. The operator $\pi_2^*$ associated to $g$ becomes: for $(x,v)\in S_{g_0}M$ 
\[ (\pi_2^*h)(x,v)= h_x(v,v)|v|_g^{-2}.\]
 
\subsection{The resolvents of $X_g$ and anisotropic spaces}
We first recall the construction of resolvents of $X_g$ from Faure-Sj\"ostrand \cite{Faure-Sjostrand-11} (see also \cite{Dyatlov-Zworski-16}) and in particular the version used in Dang-Guillarmou-Rivi\`ere-Shen \cite{DGRS} that deals with the continuity with respect to the flow $X_g$. Let $E_{u/s}^*(g)\subset T^*(SM)$ be the annihilators of $E_{u/s}(g)\oplus E_0(g)$, i.e.
\[
E_u^*(g)(E_u(g) \oplus E_0(g)) = 0, \quad E_s^*(g)(E_s(g) \oplus E_0(g)) = 0.
\]
There are two resolvents  bounded on $L^2$ for $X_{g}$ defined for ${\rm Re}(\la)>0$ by 
\[ R_g^\pm (\la):= \pm \int_{0}^\infty e^{-\la t}e^{\pm tX_g}f\, dt \] 
for $f\in L^2(SM,d\mu^{\rm L}_g)$. They solve $(-X_g\pm \la)R_g^\pm(\la)={\rm Id}$ on $L^2$.
The following results are proved in \cite{Faure-Sjostrand-11}, and we use here the presentation of \cite[Sections 3.2 and 3.3]{DGRS} due to the need of uniformity with respect to $g$: there is $c_0>0$ depending only on $g$, locally uniform with respect to $g$ ($c_0$ depends only on the Anosov exponents of contraction/dilation of $d\varphi_1^g$), such that for each $N_0>0,N_1>16N_0$, $R_g^\pm(\la)$ admits a meromorphic extension in ${\rm Re}(\la)>-c_0N_0$ as a bounded operator 
\begin{equation}\label{Rg+bound} 
R_g^-(\la): \mc{H}^{m_g^{N_0,N_1}}\to \mc{H}^{m_g^{N_0,N_1}}, \quad  R_g^+(\la): \mc{H}^{-m_g^{N_0,N_1}}\to \mc{H}^{-m_g^{N_0,N_1}}
\end{equation}
where $\mc{H}^{\pm m_g^{N_0,N_1}}$ are Hilbert spaces depending on $N_0>0,N_1>0$ satisfying the following properties: 
\[H^{2N_1}(SM)\subset \mc{H}^{m_g^{N_0,N_1}}\subset  H^{-2N_0}(SM), \quad H^{2N_0}(SM)\subset \mc{H}^{-m_g^{N_0,N_1}}\subset  H^{-2N_1}(SM)\]
and defined by 
\[
\mc{H}^{\pm m_g^{N_0,N_1}}=(A_{m_g^{N_0,N_1}})^{\mp 1}L^2(SM), \quad A_{m_g^{N_0,N_1}}:={\rm Op}({e^{m_g^{N_0,N_1}\log f}})
\]
and $A_{m_g^{N_0,N_1}}$ is an invertible pseudo-differential operator with inverse having principal symbol $e^{-m_g^{N_0,N_1}\log f}$. Here ${\rm Op}$ denotes a quantization (with a fixed small semi-classical parameter to ensure that ${\rm Op}({e^{m_g^{N_0,N_1}\log f}})$ is invertible), while $m_g^{N_0,N_1}\in S^{0}(T^*(SM))$, $f\in S^1(T^*(SM), [1,\infty))$ (the usual classes of symbols) are homogeneous of respective degree $0$ and $1$ in $|\xi|>R$, for some $R>1$ independent of $g$, and constructed from the lifted flow $\Phi_t^g=((d\varphi_t^{g})^{-1})^T$ acting on $T^*(SM)$. The function $f$ can be taken depending only on $g_0$ for $g$ in a small enough $C^\infty$ neighborhood $\mc{U}$ of $g_0$.
Moreover there are small conic neighborhoods 
$C_{u}(g_0)$ and $C_s(g_0)$ of $E^*_u(g_0)$ and $E^*_s(g_0)$ such that for any smaller open conic neighborhood 
$C'_u(g_0)\subset C_u(g_0)$ of $E_u^*(g_0)$ and $C'_s(g_0)\subset C_s(g_0)$ of $E_s^*(g_0)$, $m_g^{N_1,N_1}$ satisfies: 
\begin{equation}\label{propofmg} 
\left\{\begin{array}{ll} 
m_g^{N_0,N_1}(z,\xi)\geq N_1 , & (z,\xi)\in C'_s(g_0), \\
m_g^{N_0,N_1}(z,\xi)\geq N_1/8 & (z,\xi)\notin C_u(g_0),\\
m_g^{N_0,N_1}(z,\xi)\leq -N_0 & (z,\xi)\in C'_u(g_0),
\end{array}\right.\end{equation}
and $m_g(x,\xi)\in [-2N_0,2N_1]$ for all $(z,\xi)\in T^*(SM)$.
We note that \cite[Lemma 3.3]{DGRS} shows that $m_g^{N_0,N_1}$ is smooth with respect to the metric $g$ and that $f$ can be taken to be independent of $g$ for $g$ close enough to $g_0$. The spaces 
$\mc{H}^{m_g^{N_0,N_1}}$ are called \emph{anisotropic Sobolev spaces}. The pseudodifferential operators $A_{m_g^{N_0,N_1}}$ belong to the class $\Psi^{2N_1}(SM)$ but also to some anisotropic subclass denoted $\Psi^{m_g^{N_0,N_1}}(SM)$ admitting composition formulas; we refer to \cite{Faure-Roy-Sjostrand,Faure-Sjostrand-11} for details. 

Eventually, \cite[Proposition 6.1]{DGRS} shows that there is a small open neighborhood $W_\delta$ of the circle  $\{\la\in \C ~|~ |\la|=\delta\}$ for some small $\delta>0$ so that
\begin{equation}\label{continuity resolvent}
\mc{U}\times W_\delta \ni (g,\la)  \mapsto A_{m_g^{N_0,N_1}}R_g^-(\la)(A_{m_g^{N_0,N_1}})^{-1} \in\mc{L}(H^1(SM),L^2(SM))
\end{equation}
is continuous.\footnote{In \cite[Proposition 6.1]{DGRS}, there is a small semi-classical parameter $h>0$ appearing: we can just fix this parameter small enough. It does not play any role here except in the quantization procedure ${\rm Op}$. We also add that in \cite[Proposition 6.1]{DGRS}, $N_1$ is chosen to be equal to $20N_0$ for notational convenience, but the proof does not use that fact.}

\subsection{The operator $\Pi_2^g$ in terms of resolvents}

Following \cite{Guillarmou-17-1}, the link between $\Pi^g$ and the resolvent is given by the Laurent expansion
\[ \Pi^g= R_g^{+}(0)-R_g^{-}(0)\]
where $R_g^{+}(\la)$ has a pole of order $1$ and $R_g^\pm(0)$ is defined by 
\[ R_g^{\pm}(\la)=\pm\la^{-1}\cjg\cdot,\mathbf{1}\cjd +R_g^\pm(0)+\mc{O}(\la)\]
and $R_g^-(0)=-(R_g^+(0))^*$ where the adjoint is with respect to the Liouville measure.

\begin{lemma}
\label{lemma:decomposition}
Let $\chi\in C_c^\infty(\R)$ be even and equal to $1$ in $[-T,T]$ and supported in the interval $(-T-1,T+1)$.
Then we have 
\be\label{identitePi}
\begin{split}
\Pi^g = & \int_{\R} \chi(t) e^{tX_g} dt-R_g^+(0) \int_0^{+\infty} \chi'(t) e^{tX_g} dt
+R_g^-(0) \int_{0}^{+\infty} \chi'(t) e^{-tX_g} dt -\cjg \cdot,\mathbf{1}\cjd \int_{\R} \chi.
\end{split}\ee
\end{lemma}

\begin{proof} For ${\rm Re}(\la)>0$, we can write by integration by parts 
\[\begin{split} 
R_g^{\pm}(\la)=& \pm \int_0^\infty \chi(t)e^{-t(\la \mp X_g)}dt\pm \int_0^\infty
(1-\chi(t))e^{-t(\la \mp X_g)}dt\\
=& \pm \int_0^\infty \chi(t)e^{-t(\la \mp X_g)}dt- R_g^\pm(\la)\int_0^\infty \chi'(t)
e^{t(\pm X_g-\la)}dt.
\end{split} 
\]
Then taking the limit as $\la\to 0$, we obtain 
\[R_g^{\pm}(0)=\pm \int_0^\infty \chi(t)e^{\pm tX_g}dt-R_g^{\pm}(0)\int_0^\infty \chi'(t)
e^{\pm tX_g}dt\mp\int_0^\infty \chi(t)dt \cjg \cdot,\mathbf{1}\cjd\]
and summing gives the result.
\end{proof}
Next, we remark that, using that $\varphi_t^g(x,-v)=-\varphi_{-t}^g(x,v)$ (where multiplication by $-1$ is the symmetry in the fibers of $SM$), it is straightforward to check that for all $t\in \R$
\[ {\pi_2}_*e^{tX_g}\pi_2^*={\pi_2}_*e^{-tX_g}\pi_2^*,\]
which also implies that ${\pi_2}_*R_g^+(0)e^{tX_g}\pi_2^*=-{\pi_2}_*R_g^-(0)e^{-tX_g}\pi_2^*$ and thus
\be
\label{equation:pi2-decomposition}
\Pi_2^g = 2{\pi_2}_*\int_{0}^\infty \chi(t) e^{-tX_g}\, dt \pi_2^*+
2{\pi_2}_* R_g^-(0) \int_0^{+\infty} \chi'(t) e^{-tX_g} \, dt\pi_2^* + \left(1-\int_\R \chi\right) \langle \cdot , \mathbf{1} \rangle.
\ee
We are going to prove that these three terms depend continuously on $g$. Note that
\[ \left(1-\int_\R \chi\right) \langle f,\mathbf{1}\rangle = \left(1-\int_\R \chi\right) \int_{SM}f(z) d\mu_g^{\rm L}(z) \]
and thus the $g$-continuity of this term is immediate. Now, we claim the following:

\begin{lemma}\label{choice of T}
There is $T>0$ large enough and a neighborhood $\mc{U}'\subset \mc{U}$ of $g_0$ in $\mc{M}^{\rm An}$ so that for all $x\in M$ and all $g\in \mc{U}'$ the exponential map of $g$ in the universal cover $\til{M}$
\[\exp^{\tilde{g}}_x: \{v\in T_{x}\til{M}; |v|_g\leq T\}\to \til{M}
\]
is a diffeomorphism onto its image and $\Phi_{t}^g(V^*\cap \ker \iota_{X_g})\subset C'_u(g_0)$ for all $t\geq T$, if $\Phi_t^g:=((d\varphi_t^g)^{-1})^T$ is the symplectic lift of $\varphi_t^g$,  $V^*\subset T^*(SM)$ is the annihilator of the vertical bundle $V=\ker d\pi_0\subset T(SM)$ and $\iota_X:T^*(SM)\to \rr$ is the contraction $\iota_{X_g}(\xi)=\xi(X_g)$.
\end{lemma}

We also mention here as it is used in the following proof that, as a consequence of hyperbolicity, 
\[
V^* \cap E_s^* = V^* \cap E_u^* = \left\{0 \right\}.
\]
This can be found in \cite[Theorem 2.50]{Paternain-99} for instance (formulated for the tangent bundle $T(SM)$ but the adaptation to $T^*(SM)$ is straightforward). The $T$ in Lemma \ref{lemma:decomposition} will be chosen accordingly so that Lemma \ref{choice of T} is satisfied.

\begin{proof} By Lemma 3.1 of \cite{DGRS}, the cone $\mc{C}'_u(g_0)$ can be chosen so that there is $T>0$ and $\mc{U}'$ 
such that for all $t\geq T$ and all $g\in \mc{U}'$, $\Phi_{t}^g(C'_u(g_0))\subset C'_u(g_0)$.  
We also know that $\Phi_{T_0}^{g_0}(V^*\cap \ker \iota_{X_g}) \subset C'_u(g_0)$ for some $T_0>T$ by hyperbolicity of $g_0$ (i.e. the stable bundle $E_s^*$ only intersects trivially the vertical bundle $V^* \cap \ker \imath_{X_g}$), but by continuity of $g\mapsto \Phi_{T_0}^g$, the same holds for all $g$ in some possibly smaller neighborhood $\mc{U}''\subset \mc{U}'$, thus for all $t\geq T_0$ and all $g\in \mc{U}''$, 
$\Phi_{t}^g(V^*\cap \ker\iota_{X_g})\subset C'_u(g_0)$. Now, we claim that, up to choosing $\mc{U}''$ even smaller, the exponential map is a diffeomorphism on $\{|v|_g\leq T\}$ in the universal cover: indeed, Anosov geodesic flows have no pair of conjugate points.
\end{proof}

\subsection{Proof or Proposition \ref{proposition:pi2-continuite}}
Let us define 
\[ \Omega_1^g:= {\pi_2}_*\int_{0}^\infty \chi(t) e^{-tX_g}\, dt \pi_2^*,\quad \quad  \Omega_2^g:={\pi_2}_* R_g^-(0) \int_0^{+\infty} \chi'(t) e^{-tX_g} \, dt\pi_2^*\]
Proposition \ref{proposition:pi2-continuite} is a consequence of the following two Lemmas. 
\begin{lemma}\label{Omega1}
For each $g\in \mc{U}'$, $\Omega_1^g\in \Psi^{-1}(M;S^2T^*M)$ with principal symbol 
\[ \sigma(\Omega_1^g)(x,\xi)=c_n|\xi|^{-1}\pi_{\ker i_\xi}A_2^2\pi_{\ker i_\xi}\]
for some $c_n>0$ depending only on $n=\dim M$ and $A_2$ some positive definite endomorphism defined in Lemma \ref{lemma:positivite-pi2}, and the map $g\mapsto \Omega_1^g$ is continuous with respect to the smooth topology on $\mc{U}'$ and the usual Fr\'echet topology on $\Psi^{-1}(M;S^2T^*M)$.
\end{lemma}
\begin{proof}
The fact that, for each $g\in \mc{M}^{\rm An}$, the operator $\Omega_1^g\in \Psi^{-1}(M;S^2T^*M)$ is proved in \cite[Theorem 3.5]{Guillarmou-17-1}, the computation of the principal symbol follows from the computation \cite{Sharafut-Skokan-Uhlmann,Stefanov-Uhlmann-04} and is done in details in our setting in \cite[Theorem 4.4.]{Gouezel-Lefeuvre-19}. We need to check the continuity with respect to $g$ in the $\Psi^{-1}(M;S^2T^*M)$ topology and we can proceed as in \cite[Proposition 1 and 2]{Stefanov-Uhlmann-04}. For $h\in C^\infty(M;S^2T^*M)$, we can write explicitly in $(x_i)_i$ coordinates in the universal cover $\til{M}$ near a point $p\in \til{M}$
\[ (\Omega_1^gh(x))_{ij}=\int_{S_x\til{M}}\int_0^\infty \chi(t) 
\til{h}_{\exp^{\til{g}}_x(tv)}(\pl_t\exp^{\til{g}}_x(tv),\pl_t\exp^{\til{g}}_x(tv)) p_{ij}(x,v) \, dtdS_x(v)\]
where $p_{ij}(x,v)$ are homogeneous polynomials of order $2$ in the $v$ variable, $\til{h}\in C^\infty(\til{M};S^2T^*M)$ is the lift of $h$ to the universal cover $\til{M}$, $dS_x$ is the natural measure on the sphere $S_x\til{M}$. Using Lemma \ref{choice of T}, we can perform the change of coordinates $(t,v)\in (0,T)\x S_x\til{M}\mapsto y:=\exp^{\til{g}}_{x}(tv)\in \til{M}$, we get $t=d_{\til{g}}(x,y)$ the distance in $\til{M}$, and
\[dtdv=\frac{J^g_x(y)}{(d_{\til{g}}(x,y))^{n-1}}d{\rm vol}_g(y), \quad v=-\nabla^{\til{g}}_yd_{\til{g}}(x,y), \quad \pl_t\exp^{\til{g}}_x(tv)=(\nabla^{\til{g}}_xd_{\til{g}}(x,y)), \]
for some $J^g_{x}(y)$ smooth in $x,y,g$. We claim that this implies that 
\[ \Omega_1^gh(x)=\int_{M} K_g(x,y)h(y)\, d{\rm vol}_g(y)\]
for some $K_g(x,y)$
which is smooth in $(g,x,y)$ outside the diagonal $x=y$ and, near the diagonal,
it has the form (for some $L<\infty)$ 
\[ K_g(x,y)= \sum_{\ell=1}^L c_\ell(g,x,y)\omega_{\ell, g,x}(x-y)\]
with $c_\ell$ a matrix valued function, smooth in all its variables and $\omega_{\ell, g,x}(v)$ a vector valued function smooth in $g,x$, homogeneous of degree $-(n-1)$ in $v\in \rr^n$. Indeed, one can work in 
the universal cover $\til{M}$ where $x_i$ are globally defined coordinates, so that writing
$h(x)=\sum_{i,j}h_{ij}(x)dx_idx_j$ and $p=\sum_{ij}p_{ij}(x)dx_idx_j$, and we get that $K_g(x,y)$ is a matrix with coefficients
\[ (K_g(x,y))_{iji'j'}=\chi(d_{\til{g}}(x,y)) p_{ij}(x)F_i^g(x,y)F_{j}^g(x,y)G_{i'}^g(x,y)G_{j'}^g(x,y)\frac{J_{x}^g(y)}{d_{\til{g}}(x,y)^{n-1}}
 \]
where $F^g_i(x,y)=-dx_i(\nabla^{\til{g}}_yd_{\til{g}}(x,y))$ and 
$G_i^g(x,y)=dx_i(\nabla^{\til{g}}_x d_{\til{g}}(x,y))$. Now we can use the standard fact (see for example \cite[Lemma 1]{Stefanov-Uhlmann-04}) that 
\[d_{\tilde{g}}^2(x,y)=\sum_{ij}H^{1}_{ij}(g,x,y)(x-y)_i(x-y)_j, \quad 
dx_i(\nabla^g_x d_{\tilde{g}}(x,y))=\frac{\sum_{ij}H^{2}_{ij}(g,x,y)(x-y)_j}{d_{\tilde{g}}(x,y)}\]
(and the same thing for $dx_i(\nabla^g_y d_{\tilde{g}}(x,y))$ by symmetry) 
where $H_{ij}^{k}(g,x,y)$ are smooth in all variables and positive definite for $x=y$.
The kernel $K_g$ is thus smooth outside the diagonal (as a function of $g,x,y$), and it can be written near the diagonal as a sum of terms of the form $c(g,x,y)\omega_{g,x}(x-y)$ 
where $c$ is smooth in all its variables and $\omega_{g,x}(v)$ is a homogeneous distribution of 
degree $-(n-1)$ in the variable $v$, smooth in $g,x$. The off-diagonal term for the Fr\'echet topology is then clearly smooth in $g$, while the near diagonal term has full local symbols that are Fourier transforms of $c(g,x,x-v)\omega_{g,x}(v)$:
\[ \sigma(g;x,\xi)=\int_{\rr}e^{iv\xi}c(g,x,x-v)\omega_{g,x}(v)dv.\]
 It is then a standard and easy exercise to check that this provides uniform bounds on semi-norms of the symbol\footnote{Alternatively, the semi-norms on the full-symbol are equivalent to 
semi-norms in the space of distributions on $M\times M$ that are conormal to the diagonal, defined through differentiations of $K_g(x,y)$ with respect to smooth fields tangent to ${\rm diag}(M\times M)$, see \cite[Chapter 5, Proposition 6.1.1 and its proof]{Melrose}. Such norms for $K_g$ are clearly uniformly bounded in terms of $g$.}. 
We deduce the continuity (and indeed, smoothness) of 
$\Omega_1^g$ as an element of $\Psi^{-1}(M;S^2T^*M)$ with respect to the metric $g$.  
\end{proof}

\begin{lemma}\label{smoothnessofremainder}
The operator $\Omega_2^g$ has a smooth Schwartz kernel for each $g\in \mc{U}'$ and the map 
\[ g\in \mc{U}'\mapsto \Omega_2^g\in C^\infty(M\times M; S^2T^*M\otimes (S^2T^*M)^*)\]
is continuous if we identify $\Omega_g^2$ with its Schwartz kernel.
\end{lemma}
\begin{proof} First we observe that if 
$B\in \Psi^0(SM)$ is chosen, independently of $g$, so that  $B^*=B$ and $B$ microsupported in a small conic neighborhood of $V^*$ not intersecting $\mc{C}_u(g_0)$ and equal microlocally to the identity in a slightly smaller conic neighborhood of $V^*$, then
\[\pi_2^*= B\pi_2^*+S_g, \quad {\pi_2}_*={\pi_2}_*B+S_g^*\] 
with $S_g$ a continuous family of smoothing operators. This decomposition is a consequence of the fact that $\pi_2^*$ maps $C^{-\infty}(M;S^2T^*M)$ to the space 
$C^{-\infty}_{V^*}(SM)$ of distributions with wavefront set contained in $V^*$ ($\pi_2^*$ being essentially a pullback, this follows for instance from \cite[Theorem 8.2.4]{Hoe03}).
We will show that the operator 
\[ \Omega^g_3:= {\pi_2}_*BR_g^-(0)\int_T^{T+1} \chi'(t)e^{-tX_g}B\pi_2^* dt\]   
is a continuous family (with respect to $g$) of smoothing operators. We need to show that for each $N>0$, $\Omega_3^g: H^{-N}(SM)\to H^{N}(SM)$ is a continuous family with respect to $g$ of bounded operators. To study $R_g^-(0)$, it suffices to write it under the form 
\begin{equation}\label{Rg0residue} 
R_g^-(0)=\frac{1}{2\pi i}\int_{|\la|=\delta}\frac{R_g^-(\la)}{\la}d\la
\end{equation}
with $\delta$ small enough so that the only pole of $R_g^-(\la)$ in $|\la|\leq \delta$ is $\la=0$\footnote{This is possible for $g$ close enough to $g_0$ by continuity of $g\mapsto R_g^-(\la)$ proved in \cite{DGRS}. Note that the spectrum (the Pollicott-Ruelle resonances) depend continuously on the metric as was shown by \cite{Bonthonneau-18}.}, and thus it amounts to analyze $R_g^-(\la)$ on $\left\{|\la|=\delta\right\}$. 
We decompose $B=B^1+B^2$ with $B^i\in \Psi^0(SM)$ where ${\rm WF}(B^1)$ is contained in a conic neighborhood of $\ker \iota_{X_{g_0}}$ not intersecting the annihilator $E_0(g_0)^*$ of $E_u(g_0)\oplus E_s(g_0)$ (the neutral direction) and ${\rm WF}(B^2)\cap \ker \iota_{X_{g_0}}=\emptyset$ ($B^2$ is microsupported in the elliptic region).
For $i=1,2$ we let 
$B_T^{i}\in \Psi^0(SM)$ be microsupported in a conic neighborhood of $\cup_{t\in [T,T+1]}\Phi_{t}^g({\rm WF}(B^i))$, so that by Egorov (or simply the formula of composition of $\Psi^{0}(SM)$ with diffeomorphisms of $SM$)  
\[ \forall t\in [T,T+1], \quad e^{-tX_g}B^i=B_T^ie^{-tX_g}B^i+ S'_{g,i}(t)\]
for some continuous family $(g,t)\mapsto S_{g,i}'(t)$ of smoothing operators (for $g$ close enough to $g_0$).  We note that by taking $\mc{U}'$ small enough and ${\rm WF}(B^1)$ close enough to $V^*\cap \ker\iota_{X_{g_0}}$, 
Lemma \ref{choice of T} insures that we can choose $B^1_T$ depending only on $T$ (thus uniform in $g\in \mc{U}'$) so that 
${\rm WF}(B_T^1)\subset C'_u(g_0)$. Thus 
\[\int_T^{T+1} \chi'(t)e^{-tX_g}B^1 dt =B_T^1\int_T^{T+1} \chi'(t)e^{-tX_g}B^1 dt+S_{g,1}''\]
for some continuous family $g\mapsto S_{g,1}''$ of smoothing operators. 
Next we use \eqref{Rg+bound} 
with the choice $N_0=N+1$ and $N_1/16=N+2$. Since by \eqref{propofmg}
\[
m_g(z,\xi)\leq -N-1 \textrm{ for all }(z,\xi)\in {\rm WF}(B^1_T),
\]
we obtain, using the composition properties in \cite[Theorem 8]{Faure-Roy-Sjostrand} that $A_{m_g^{N_0,N_1}}B^1_T\in \Psi^{-N-1}(SM)$ is uniformly bounded with respect to $g$ and continuous as a map $g\in \mc{U}'\mapsto A_{m_g^{N_0,N_1}}B^1_T\in\mc{L}(H^{-N}(SM), H^1(SM))$. In particular    
\begin{equation}\label{firstguy}
\mc{U}'\ni g\mapsto A_{m_g^{N_0,N_1}} \int_T^{T+1} \chi'(t)e^{-tX_g}B^1 dt \in \mc{L}(H^{-N}(SM), H^1(SM))
\end{equation}
is continuous. Next, we deal with the ``elliptic region'' term, i.e. the term $B^2$. The idea is to show it is smoothing, since it is a Schwartz function of $X_g$ microlocalized in the elliptic region of $X_g$. First, ${\rm WF}(B^2_T)$ does not intersect $\ker \iota_{X_g}$ for $g\in \mc{U}'$ after possibly reducing $\mc{U}'$ since it does not intersect $\ker \iota_{X_{g_0}}$.
Moreover we have 
\[ X_g^{2N} \int_{T}^{T+1}\chi'(t)e^{-tX_g}B^2dt=\int_{T}^{T+1}\chi^{(1+2N)}(t)e^{-tX_g}B^2dt,\]
and since ${\rm WF}(B^2_T)$ does not intersect $\ker \iota_{X_g}$ for $g\in \mc{U}'$, 
there is by microlocal ellipticity (\cite[Proposition E.32]{Dyatlov-Zworski-book-resonances}) a family $Q_g\in \Psi^{-2N}(SM)$ and $Z_g\in \Psi^{-\infty}(SM)$, both continuous with respect to $g$, so that 
\[ Q_gX_g^{2N}=B^2_T+Z_g.\]
We write
\[ \begin{split}
B_T^2\int_T^{T+1} \chi'(t)e^{-tX_g}B^2 dt=& Q_gX_g^{2N}\int_T^{T+1} \chi'(t)e^{-tX_g}B^2 dt+Z_g'\\
=& Q_g\int_{T}^{T+1}\chi^{(1+2N)}(t)e^{-tX_g}B^2dt+Z_g'
\end{split}\] 
where $Z_g'\in \mc{L}(H^{-N}(SM),H^N(SM))$ continuously in $g$. Since 
$\int_T^{T+1} \chi'(t)e^{-tX_g}B^2 dt$ is continuous in $g$ as a bounded map $\mc{L}(H^{-N}(SM))$  
and $Q_g$ is continuous in $g$ as a bounded map $\mc{L}(H^{-N}(SM),H^N(SM))$, we get 
\[B_T^2\int_T^{T+1} \chi'(t)e^{-tX_g}B^2 dt\in \mc{L}(H^{-N}(SM),H^{N}(SM))\]
continuously in $g\in \mc{U}'$.
Combine these facts with \eqref{firstguy}, \eqref{continuity resolvent} and \eqref{Rg0residue}, we deduce that  
\[
\mc{U}'\ni g\mapsto A_{m_g^{N_0,N_1}}R_g^-(0)(A_{m_g^{N_0,N_1}})^{-1}A_{m_g^{N_0,N_1}}\int_T^{T+1} \chi'(t)e^{-tX_g}B dt \in \mc{L}(H^{-N}(SM),L^2(SM))
\] 
is continuous. Finally, using that ${\rm WF}(B)\cap C_u(g_0)=\emptyset$ and $-m_{g}^{N_0,N_1}\leq -2N-4$ outside $C_u(g_0)$ by \eqref{propofmg}, we have that $B(A_{m_g^{N_0,N_1}})^{-1}\in \Psi^{-2N-4}(SM)$ uniformly in $g$ (using again \cite[Theorem 8]{Faure-Roy-Sjostrand} and \cite[Lemma 3.2]{DGRS}) and the following map is continuous 
\[
\mc{U}'\ni g\mapsto B(A_{m_g^{N_0,N_1}})^{-1}\in \mc{L}(L^2(SM),H^N(SM)).
\]
This shows that $\mc{U}'\ni g\mapsto \Omega_3^g \in \mc{L}(H^{-N}(M;S^2T^*M),H^N(M;S^2T^*M))$
is continuous. The terms involving the smoothing remainders $S_g$ appearing in the difference between $\Omega_2^g$ and $\Omega_3^g$ can be dealt using the same argument, and indeed are even simpler to consider. The proof is then complete.
\end{proof}
The proof of Proposition \ref{proposition:pi2-continuite} is simply the combination of Lemma \ref{smoothnessofremainder} and Lemma \ref{Omega1}. \qed\\

As a corollary we prove Theorem \ref{mainth3}.

\subsection{Proof of Theorem \ref{mainth3}}

Let $g_0\in \mc{M}^{\rm An}$ and assume $g_0$ has non-positive curvature if $n\geq 3$. 
Using Lemma \ref{lemma:solenoidal-gauge}, for $g_1,g_2\in \mc{M}$ 
close enough to $g_0$ in $C^{k+3,\alpha}$ norm, we can find 
$\psi\in \mc{D}_0^{k+1,\alpha}$ (with $k\geq 5$ to be chosen later), depending in a $C^2$ fashion on $(g_1,g_2)$ such that $D_{g_1}^*(\psi^*g_2)=0$. 
Moreover $g_2'=\psi^*g_2$ satisfies 
\[\| g_2'-g_1\|_{C^{k,\alpha}}\leq C(\|g_1-g_0\|_{C^{k,\alpha}}+\|g_2-g_0\|_{C^{k,\alpha}})\]
for some $C$ depending only on $g_0$.
We can then rewrite the proof of Theorem \ref{theorem:stability} 
replacing $g_0$ by $g_1$. Let $\Psi_{g_1}(g_2)={\bf P}\Big(-J_{g_1}^u-a_{g_1,g_2}+\int_{S_{g_0}M}a_{g_1,g_2}\, d\mu_{g_1}^L\Big)$ be the map \eqref{defofPSi} with $(g_1,g_2)$ replacing $(g_0,g)$, and $\Phi_{g_1}(g_2)=I_{\mu^{\rm L}_{g_1}}(g_1,g_2)$, where $a_{g_1,g_2}$ is the time reparameterization coefficient \eqref{equation:conjugaison} in the conjugacy between the flows $\varphi^{g_1}$ and $\varphi^{g_2}$ and the pressure and the stretch are 
taken with respect to the flow $\varphi^{g_1}$. Combining \cite[Theorem C]{Contreras-92} and Proposition \ref{stabilityth}, 
the maps $(g_1,g_2)\mapsto \Psi_{g_2}(g_1)$ and $(g_1,g_2)\mapsto \Phi_{g_1}(g_2)$ are $C^3$ in $g_2$ if $k$ is chosen large enough, and each $g_2$-derivative of order $\ell\leq 3$ is continuous with respect to $(g_1,g_2)\in C^{k+3}\times C^{k+3}$ (again $k$ is fixed large enough).
Following the proof of Proposition \ref{estimfinal1}, this gives that for $g_1,g_2$ smooth but close enough to $g_0$ in $\mc{M}^{k+3,\alpha}$
\[C_n(\Phi_{g_1}(g_2)-1)^2+\Psi_{g_1}(g_2)\geq \frac{1}{8}\cjg\Pi_2^{g_1} (g_2'-g_1),(g_2'-g_1)\cjd-C'_{g_1}\|g_2-g_1\|^3_{C^{k_0,\alpha}}\]
where $C_n$ depends only on $n=\dim M$, $C'_{g_1}$ depends on $\|g_1\|_{C^{k_0,\alpha}}$ for some fixed $k_0$. 

Combining Proposition \ref{proposition:pi2-continuite} and Lemma \ref{lemma:positivite-pi2unif}, we deduce that there exist $C_{g_0},C'_{g_0}>0$ depending only on $g_0$ so that for $g_1,g_2\in \mc{M}$ in a small enough neighborhood of $g_0$ in the $C^{k+3,\alpha}$ topology (for $k\geq k_0$), 
\[C_n(\Phi_{g_1}(g_2)-1)^2+\Psi_{g_1}(g_2)\geq C_{g_0}\|g_2'-g_1\|_{H^{-1/2}(M)}-C_{g_0}'\|g_2-g_1\|^3_{C^{5,\alpha}}.\]
This means that there exists $\eps>0$ depending on $g_0$ and $k$ large enough so that for all $g_1,g_2\in \mc{M}$ smooth satisfying $\|g_j-g_0\|_{C^{k+3,\alpha}(M)}\leq \eps$ the estimates of Proposition \ref{estimfinal1} with  $(g_1,g_2)$ replacing $(g_0,g)$ hold uniformly with respect to  $(g_1,g_2)$. This proves the desired result.\qed

\section{Distances from the marked length spectrum}

In this paragraph, we discuss different notions of distances involving the marked length spectrum on the space of isometry classes of negatively-curved metrics. Again, if the X-ray transform $I_2$ were known to be injective, it is likely that one could only assume the Anosov property for the metrics in this paragraph.

\subsection{Length distance}

We define the following map:

\begin{definition}\label{disdL}
Let $k$ be as in Theorem \ref{mainth3}. We define the marked length distance map $d_L: \mc{M}^{k,\alpha}\x \mc{M}^{k,\alpha}\to \R^+$ by
\[ d_L(g_1,g_2):=\limsup_{j\to \infty}
\Big|\log \frac{L_{g_1}(c_j)}{L_{g_2}(c_j)}\Big|.\]
\end{definition}

This is indeed well defined. If $g_1, g_2$ are two such metrics, then there exists a constant $C = C(g_1,g_2) \geq 1$ such that for all $(x,v) \in TM$, $1/C \times |v|_{g_1(x)} \leq |v|_{g_2(x)} \leq C \times |v|_{g_1(x)}$. As a consequence, using that a geodesic is a minimizer of the length among a free homotopy class, we obtain:
\[
\dfrac{L_{g_1}(c_j)}{L_{g_2}(c_j)} = \dfrac{\ell_{g_1}(\gamma_{g_1}(c_j))}{\ell_{g_2}(\gamma_{g_2}(c_j))} \leq  \dfrac{\ell_{g_1}(\gamma_{g_2}(c_j))}{\ell_{g_2}(\gamma_{g_2}(c_j))} \leq C^{1/2} \dfrac{\ell_{g_2}(\gamma_{g_2}(c_j))}{\ell_{g_2}(\gamma_{g_2}(c_j))} = C^{1/2},
\]
and the lower bound follows from a similar computation. We get as a Corollary of Theorem \ref{mainth3}:

\begin{corollary}
The map $d_L$ descends to the set of isometry classes near $g_0$ and defines a distance in a small $C^{k,\alpha}$-neighborhood of the isometry class of $g_0$.
\end{corollary}
\begin{proof}
It is clear that $d_L$ is invariant by action of diffeomorphisms homotopic to Identity since $L_g=L_{\psi^*g}$ for such diffeomorphisms $\psi$. Now let $g_1,g_2,g_3$ three metrics. We have 
\[ \begin{split}
\limsup_{j\to \infty}
\Big|\log \frac{L_{g_1}(c_j)}{L_{g_2}(c_j)}\Big|= & \limsup_{j\to \infty}
\Big|\log \frac{L_{g_1}(c_j)}{L_{g_3}(c_j)}\frac{L_{g_3}(c_j)}{L_{g_2}(c_j)}\Big|\\
\leq & \limsup_{j\to \infty}
\Big|\log \frac{L_{g_1}(c_j)}{L_{g_3}(c_j)}\Big|+\limsup_{j\to \infty}
\Big|\log \frac{L_{g_3}(c_j)}{L_{g_2}(c_j)}\Big|.
\end{split}\]
thus $d_L$ satisfies the triangle inequality. Finally, By Theorem \ref{mainth3}, 
if $d_L(g_1,g_2)=0$ with $g_1,g_2$ in the $C^{k,\alpha}$ neighborhood $\mc{U}_{g_0}$ of Theorem \ref{mainth3}, we have $g_1$ isometric to $g_2$, showing that $d_L$ produces a distance on the quotient of $\mc{U}_{g_0}$ by diffeomorphisms.
\end{proof}
We also note that Theorem \ref{mainth3} states that there is $C_{g_0}>0$  such that for each $g_1,g_2\in C^{k,\alpha}(M;S^2T^*M)$ close to $g_0$ there is a diffeomorphism such that:
\[
d_L(g_1,g_2)^{1/2}\geq C_{g_0}\|\psi^*g_1-g_2\|_{H^{-1/2}}.
\]

\subsection{Thurston's distance}

\label{ssection:thurston}

We also introduce the Thurston distance on metrics with topological entropy $1$, generalizing the distance introduced by Thurston in \cite{Thurston-98} for surfaces on Teichmüller space (all hyperbolic metrics on surface have topological entropy equal to $1$). We denote by $\mc{E}$ (resp. $\mc{E}^{k,\alpha}$) the space of negatively curved metrics in $\mc{M}$ (resp. in $\mc{M}^{k,\alpha}$) with topological entropy $\mathbf{h}_{{\rm top}}(g)=1$. (Let us also recall here for the sake of clarity that $\mathbf{h}_{{\rm top}}(\lambda^2 g)=\mathbf{h}_{{\rm top}}(g)/\lambda$, for $\lambda > 0$.) With the same arguments as in Lemma \ref{lemma:h1}, this is a codimension $1$ submanifold of $\mc{M}$ and if $g_0 \in \mc{E}^{k,\alpha}$, one has:
\begin{equation}
\label{equation:tangent-e}
T_{g_0}\mc{E}^{k,\alpha} := \left\{ h \in C^{k,\alpha}(M;S^2T^*M) ~|~ \int_{S_{g_0}M} \pi_2^*h ~d\mu_{g_0}^{\rm BM} = 0 \right\}.
\end{equation}

\begin{definition}
\label{definition:distance-t}
We define the Thurston non-symmetric distance map $d_T: \mc{E}^{k,\alpha} \x \mc{E}^{k,\alpha} \to \R^+$ by
\[ d_T(g_1,g_2):=\limsup_{j\to \infty}
\log \frac{L_{g_2}(c_j)}{L_{g_1}(c_j)}
.\]
\end{definition}

Note that the finiteness of the previous quantity also follows from the same argument as the one justifying the finiteness of Definition \ref{disdL}. Its non-negativity will be a consequence of Lemma \ref{lemma:lien-stretch-mls} where it is proved that this can be expressed in terms of the geodesic stretch. We will prove the

\begin{proposition}
\label{proposition:dt-distance}
The map $d_T$ descends to the set of isometry classes of metrics in $\mc{E}^{k,\alpha}$ (for $k \in \N$ large enough, $\alpha \in (0,1)$) with topological entropy equal to $1$ and defines a non-symmetric distance in a small $C^{k,\alpha}$-neighborhood of the diagonal.
\end{proposition}

Moreover, this distance is non-symmetric in the pair $(g_1,g_2)$ which is also the case of the original distance introduced by Thurston \cite{Thurston-98} but this is just an artificial limitation\footnote{Thurston, \cite{Thurston-98}.}: \emph{``It would be easy to replace $L$\footnote{In the notations of Thurston, $L(g,h) =  \limsup_{j\to \infty}
\log \frac{L_{g}(c_j)}{L_{h}(c_j)}$.} by its symmetrization $1/2(L(g, h)+L(h, g))$, but it seems that, because of its direct geometric interpretations, $L$ is more useful just as it is.''} In order to justify that this is a distance, we start with the

\begin{lemma}
\label{lemma:lien-stretch-mls}
Let $g_1,g_2 \in \mc{M}$ be negatively curved. Then:
\[
\limsup_{j\to \infty} \dfrac{L_{g_2}(c_j)}{L_{g_1}(c_j)} = \sup_{m \in \mathfrak{M}_{\rm inv,erg}} I_m(g_1,g_2) \geq 0
\]
\end{lemma}

Note that there is no need to assume $g_1$ and $g_2$ to be close in this Lemma: this follows from Appendix \ref{appendix:conjugacy}, where we discuss the fact that the stretch (and the time-reparametrization) is well-defined despite the fact that the metrics may not be close. Here $m$ is seen as an invariant ergodic measure for the flow $\varphi^{g_1}_t$ living on $S_{g_1}M$. However, writing $M = \Gamma \backslash \widetilde{M}$ with $\Gamma \simeq \pi_1(M,x_0)$ for $x_0 \in M$, it can also be identified with a geodesic current on $\partial_\infty \widetilde{M} \times \partial_\infty \widetilde{M} \setminus \Delta$, that is a $\Gamma$-invariant Borel measure, also invariant by the flip $(\xi,\eta) \mapsto (\eta,\xi)$ on $\partial_\infty \widetilde{M} \times \partial_\infty \widetilde{M} \setminus \Delta$. This point of view has the advantage of being independent of $g_1$ (see \cite{Schapira-Tapie-18}).

\begin{proof}
First of all, we claim that\footnote{As pointed out to us by one of the referees, the map $\mathfrak{M}_{\rm inv} \ni m \mapsto I_m(g_1,g_2)$ is continuous and linear on a compact convex set; it thus achieves its maximum on the extremal points of the convex sets (the ergodic measures) so the argument could be shortened.}
\[
\sup_{m \in \mathfrak{M}_{\rm inv,erg}} I_m(g_1,g_2) =  \sup_{m \in \mathfrak{M}_{\rm inv}} I_m(g_1,g_2).
\]
Of course, it is clear that $\sup_{m \in \mathfrak{M}_{\rm inv,erg}} I_m(g_1,g_2) \leq  \sup_{m \in \mathfrak{M}_{\rm inv}} I_m(g_1,g_2)$ and thus we are left to prove the reverse inequality. By compactness, we can consider a measure $m_0 \in \mathfrak{M}_{\rm inv}$ realizing $\sup_{m \in \mathfrak{M}_{\rm inv}} I_m(g_1,g_2)$. By the Choquet representation Theorem (see \cite[pp. 153]{Walters-82}), there exists a (unique) probability measure $\tau$ on $\mathfrak{M}_{\rm inv,erg}$ such that $m_0$ admits the ergodic decomposition $m_0 = \int_{\mathfrak{M}_{\rm inv,erg}}m ~d\tau(m)$. Thus:
\[
\begin{split}
I_{m_0}(g_1,g_2) & = \int_{S_{g_1}M} a_{g_1,g_2} ~dm_0 \\
& = \int_{\mathfrak{M}_{\rm inv,erg}} \int_{S_{g_1}M} a_{g_1,g_2} ~dm ~d\tau(m) \\
& \leq \sup_{m \in \mathfrak{M}_{\rm inv,erg}} \int_{S_{g_1}M} a_{g_1,g_2} ~dm \int_{\mathfrak{M}_{\rm inv,erg}} d\tau(m) = \sup_{m \in \mathfrak{M}_{\rm inv,erg}} I_{m}(g_1,g_2),
\end{split}
\]
which eventually proves the claim.

Let $(c_j)_{j \in \N}$ be a subsequence such that $\lim_{j \rightarrow +\infty} L_{g_2}(c_j)/L_{g_1}(c_j)$ realizes the $\limsup$. Then, by compactness, we can extract a subsequence such that $\delta_{g_1}(c_j) \rightharpoonup m \in \mathfrak{M}_{\rm inv}$. Thus:
\[ 
L_{g_2}(c_j)/L_{g_1}(c_j) = \langle \delta_{g_1}(c_j), a_{g_1,g_2} \rangle \rightarrow_{j \rightarrow +\infty} \langle m, a_{g_1,g_2}\rangle = I_m(g_1,g_2),
\]
which proves, using our preliminary remark, that
\[
\limsup_{j \rightarrow +\infty} L_{g_2}(c_j)/L_{g_1}(c_j) \leq \sup_{m \in \mathfrak{M}_{\rm inv,erg}} I_m(g_1,g_2).
\]
To prove the reverse inequality, we consider a measure $m_0 \in \mathfrak{M}_{\rm inv,erg}$ such that $I_{m_0}(g_1,g_2) = \sup_{m \in \mathfrak{M}_{\rm inv,erg}} I_m(g_1,g_2)$ (which is always possible by compactness). Since $m_0$ is invariant and ergodic, there exists a sequence of free homotopy classes $(c_j)_{j \in \N}$ such that $\delta_{g_1}(c_j) \rightharpoonup m_0$ (by \cite{Sigmund}). Then, as previously, one has
\[
I_{m_0}(g_1,g_2) = \lim_{j \rightarrow +\infty} L_{g_2}(c_j)/L_{g_1}(c_j) \leq \limsup_{j \rightarrow +\infty} L_{g_2}(c_j)/L_{g_1}(c_j),
\]
which provides the reverse inequality.
\end{proof}

We can now prove Proposition \ref{proposition:dt-distance}.

\begin{proof}[Proof of Proposition \ref{proposition:dt-distance}]
By \eqref{equation:inegalite-stretch}, for $g_1,g_2 \in \mc{E}^{k,\alpha}$, we have that $I_{\mu_{g_1}^{\rm BM}}(g_1,g_2) \geq 1$ and thus by Lemma \ref{lemma:lien-stretch-mls}, we obtain that $d_T(g_1,g_2)\geq 0$ (note that $g_1$ and $g_2$ do not need to be close for this property to hold). Moreover, triangle inequality is immediate for this distance. Eventually, if $d_T(g_1,g_2) = 0$, then $0 \leq \log I_{\mu_{g_1}^{\rm BM}}(g_1,g_2) \leq d_T(g_1,g_2)=0$, that is $I_{\mu_{g_1}^{\rm BM}}(g_1,g_2) = 1$ and by Theorem \ref{theorem:entropy}, it implies that $g_1$ is isometric to $g_2$ if $g_2$ is close enough to $g_1$ in the $C^{k,\alpha}$-topology (note that this neighborhood depends on $g_1$).
\end{proof}

We now investigate with more details the structure of the distance $d_T$. A consequence of Lemma \ref{lemma:lien-stretch-mls} is the following expression of the Thurston Finsler norm:

%
%

\begin{lemma}
\label{lemma:finsler-norm}
Let $g_0 \in \mc{E}^{k,\alpha}$ and $(g_t)_{t \in [0,\eps)}$ be a smooth family of metrics and let $f := \pl_tg_t|_{t=0}$. Then:
\begin{equation}
\label{equation:finsler-norm}
\|f\|_T := \left. \dfrac{d}{dt} d_T(g_0,g_t) \right|_{t=0} = \frac{1}{2}\sup_{m \in \mathfrak{M}_{\rm inv,erg}}  \int_{S_{g_0}M} \pi_2^*f ~dm
\end{equation}
The norm $\|\cdot\|_T$ is a \emph{Finsler norm} on $T_{g_0}\mc{E}^{k,\alpha}\cap \ker D^*_{g_0}$ 
\end{lemma}

\begin{proof}
We introduce
\[
u(t) := e^{d_T(g_0,g_t)} =  \sup_{m \in \mathfrak{M}_{\rm inv,erg}} I_m(g_0,g_t)
\]
and write $a_t := a_{g_0,g_t}$ for the time reparametrization (as in \eqref{equation:conjugaison}). Then:
\[
\begin{split}
 \lim_{t \rightarrow 0} \dfrac{u(t)-u(0)}{t}  & = \lim_{t \rightarrow 0} \sup_{m \in \mathfrak{M}_{\rm inv,erg}} \int_{S_{g_0}M} \dfrac{a_t-1}{t} dm  
 =  \sup_{m \in \mathfrak{M}_{\rm inv,erg}} \int_{S_{g_0}M} \dot{a}_0 ~dm \\
 & =  \frac{1}{2}\sup_{m \in \mathfrak{M}_{\rm inv,erg}} \int_{S_{g_0}M} \pi_2^*f ~dm = u'(0)  = \left. \dfrac{d}{dt} d_T(g_0,g_t) \right|_{t=0},
\end{split}
\]
since $\dot{a}_0=\pl_{t}a_{t}|_{t=0}$ and $\pi_2^*f$ are cohomologous by Lemma \ref{lemma:differentielle-a}. This also shows that the derivative exists. The inversion of the limit and the $\sup$ follows from the fact that, writing $F_t(m) := \int_{S_{g_0}M} (a_t-1)/t ~~dm$,
one has $\sup_{m \in \mathfrak{M}_{\rm inv,erg}} |F_t(m)-F_0(m)| \rightarrow_{t \rightarrow 0} 0$. Note that, up to taking a large $k \in \N$ and iterating the same computation for higher order derivatives, shows that $t \mapsto u(t)$ (thus $t \mapsto d_T(g_0,g_t)$) is at least $C^2$.

We now prove that this is a Finsler norm in a neighborhood of the diagonal. We fix $g_0 \in \mc{E}^{k,\alpha}$. By Lemma \ref{lemma:solenoidal-gauge}, isometry classes near $g_0$ can be represented by solenoidal tensors, namely there exists a $C^{k,\alpha}$-neighborhood $\mc{U}$ of $g_0$ such that for any $g \in \mc{U}$, there exists a (unique) $\psi \in \mc{D}_0^{k+1,\alpha}$ such that $D^*_{g_0}\psi^*g = 0$. Moreover, if $g \in \mc{E}^{k,\alpha}$, then $\psi^*g \in \mc{E}^{k,\alpha}$. As a consequence, using \eqref{equation:tangent-e}, the statement now boils down to proving that \eqref{equation:finsler-norm} is a norm for solenoidal tensors $f \in C^{k,\alpha}(M;S^2T^*M)$ such that $\int_{S_{g_0}M} \pi_2^*f \,d\mu_{g_0}^{\rm BM} = 0$. Since triangle inequality, $\R_+$-scaling and non-negativity are immediate, we simply need to show that $\|f\|_T = 0$ implies $f=0$. Now, for such a tensor $f$, we have
\[
\begin{split}
\mathbf{P}(\pi_2^*f) & = \sup_{m \in \mathfrak{M}_{\rm inv,erg}} \mathbf{h}_m(\varphi^{g_0}_1) + \int_{S_{g_0}M} \pi_2^*f dm \\
&\leq  \sup_{m \in \mathfrak{M}_{\rm inv,erg}} \mathbf{h}_m(\varphi^{g_0}_1) +  \sup_{m \in \mathfrak{M}_{\rm inv,erg}} \int_{S_{g_0}M} \pi_2^*f dm  = \underbrace{\mathbf{h}_{\rm top}(\varphi^{g_0}_1)}_{=1} + 0
\end{split}
\]
and this supremum is achieved for $m=\mu^{\rm BM}_{g_0}$ and $\mathbf{P}(\pi_2^*f) = 1$. As a consequence, the equilibrium state associated to the potential $\pi_2^*f$ is the Bowen-Margulis measure $\mu^{\rm BM}_{g_0}$ (the equilibrium state associated to the potential $0$) and thus $\pi_2^*f$ is cohomologous to a constant $c \in \R$ (see \cite[Theorem 9.3.16]{Fisher-Hasselblatt}) which has to be $c=0$ since the average of $\pi_2^*f$ with respect to Bowen-Margulis is equal to $0$, that is there exists a Hölder-continuous function $u$ such that $\pi_2^*f = Xu$. Since $f \in \ker D^*_{g_0}$, the s-injectivity of the X-ray transform $I_2^{g_0}$ implies that $f \equiv 0$.
\end{proof}

The asymmetric Finsler norm $\|\cdot\|_T$ induces a distance $d_F$ between isometry classes namely
\[
d_F(g_1,g_2) = \inf_{\gamma : [0,1] \rightarrow \mc{E}, \gamma(0)=g_1,\gamma(1)=g_2} \int_{0}^{1} \|\dot{\gamma}(t)\|_T ~dt.
\]
It is easy to prove that $d_T(g_1,g_2) \leq d_F(g_1,g_2)$, which shows that $d_F$ is indeed a distance in a neighborhood of the diagonal, just like $d_T$. Indeed, consider a $C^1$-path $\gamma : [0,1] \rightarrow \mc{E}$ such that $\gamma(0)=g_1,\gamma(1)=g_2$. Then, considering $N \in \N, t_i := i/N$, we have by triangle inequality
\[
\begin{split}
d_T(g_1,g_2) \leq \sum_{i=0}^{N-1} d_T(\gamma(t_i),\gamma(t_{i+1})) & = \sum_{i=0}^{N-1} \|\dot{\gamma}(t_i)\|_T(t_{i+1}-t_i) + \mc{O}(|t_{i+1}-t_i|^2) \\
& \rightarrow_{N \rightarrow +\infty} \int_0^{1} \|\dot{\gamma}(t)\|_T ~dt,
\end{split}
\]
which proves the claim (note that we here use the fact that $t \mapsto d_T(g_0,g_t)$ is at least $C^2$). In \cite{Thurston-98}, Thurston proves that, on restriction to Teichmüller space, the asymmetric Finsler norm induces the distance $d_T$, that is $d_T = d_F$. 
We make the following conjecture:
\begin{conjecture}\label{conjThurston}
The distances $d_T$ coincide with $d_F$ for isometry classes of negatively curved metrics with topological entropy equal to $1$. 
\end{conjecture}

This conjecture would imply the marked length spectrum rigidity conjecture. Indeed, as mentioned just after Theorem \ref{theorem:entropy}, two metrics with same marked length spectrum have same topological entropy and it is harmless (up to a scaling of the metrics) to assume that this topological entropy is equal to $1$. Then, if the previous conjecture is true, using that their Thurston distance $d_T$ is zero, we obtain that their Finsler distance $d_F$ is zero. But this implies that the metrics are isometric.

\appendix

\section{Asymptotic marked length spectrum}

\label{appendix}

In this Appendix, we show the following (the proof was communicated to us by one of the referees):

\begin{lemma}
Let $g$ and $g_0$ be two metrics with Anosov geodesic flows on a fixed manifold $M$ and assume that $g$ is close to $g_0$ in $C^{k,\alpha}$ norm. 
Assume that for all sequences $(c_j)_{j \geq 0}$ in $\mc{C}$, $L_g(c_j)/L_{g_0}(c_j) \rightarrow_{j \rightarrow +\infty} 1$. Then $L_g = L_{g_0}$.
\end{lemma}

\begin{proof}
By Sigmund \cite[Theorem 1]{Sigmund}, the set $\mathfrak{D} := \left\{\delta_{g_0}(c) ~|~ c \in \mc{C} \right\}$ is dense in $\mathfrak{M}_{\mathrm{inv}}$ (the set of invariant measures by the $g_0$-geodesic flow on $S_{g_0}M$). If $\mu \in \mathfrak{M}_{\mathrm{inv}} \setminus  \mathfrak{D}$, we can therefore find a sequence such that $\delta_{g_0}(c_j) \rightharpoonup_{j \rightarrow +\infty} \mu$ and $L_{g_0}(c_j) \rightarrow +\infty$. (Indeed, if $L_{g_0}(c_j) \leq C$ for some $C \geq 0$, then the sequence $(\delta_{g_0}(c_j))_{j \geq 0}$ only achieves a finite number of measures which would imply that $\mu$ is a Dirac mass on a closed orbit and this is excluded since $\mu \notin \mathfrak{D}$.) Then, the condition $L_g/L_{g_0} \rightarrow 1$ immediately implies that 
\[
I_{\mu}(g_0,g) = \int_{S_{g_0}M} a_g(z) d\mu(z) = 1.
\]
Now for $c \in \mc{C}$ and $t > 0$ small, the linear combination $t \mu + (1-t)\delta_{g_0}(c) \notin \mathfrak{D}$. Indeed, if not, we would have $t \mu + (1-t)\delta_{g_0}(c) = \delta_{g_0}(c_t)$ but by continuity, $c_t=c_0$ for $t$ small, which contradicts $\mu \notin \mathfrak{D}$. Therefore:
\[
t \int_{S_{g_0}M} a_g(z) d\mu(z) + (1-t) \dfrac{1}{L_{g_0}(c)} \int_0^{L_{g_0}(c)} a_g(\varphi_s^{g_0}(z)) ds = 1,
\]
that is
\[
\dfrac{1}{L_{g_0}(c)} \int_0^{L_{g_0}(c)} a_g(\varphi_s^{g_0}(z)) ds = \dfrac{L_{g}(c)}{L_{g_0}(c)} = 1.\qedhere
\]
\end{proof}

\section{Global conjugacy for Riemannian Anosov flows}

\label{appendix:conjugacy}

Let $(M, g)$ be a closed Riemannian manifold whose geodesic flow is Anosov. As has been shown by Klingenberg  \cite{Klingenberg-74} the geodesic flow has no conjugate points.
Let $(\widetilde M, g)$ be the universal cover of $M$ where for simplicity the lifted metric  is also denoted by $ g$. Let  $\Gamma$  be the group
of deck transformations.
As has been remarked in \cite {Knieper-12} the universal cover $\widetilde M$  is Gromov hyperbolic (see \cite[Section III.H.1]{Bridson-Haefliger-99} for a definition of Gromov-hyperbolicity).
Denote by $\partial_{\infty} \widetilde M$ the Gromov boundary which is equipped with the visibility topology (see e.g \cite{Knieper-survey} for more details).
For $\xi \in \partial_{\infty} \widetilde M$ and $x_0 \in \wt M$ the Busemann function $x \mapsto b_\xi^g(x_0, x)$ is defined by
\begin{equation}
\label{equation:definition-busemann}
 b_\xi^g(x_0, x) := \lim_{z \to \xi} d_g(x_0, z) - d_g(x,z).
\end{equation}
It has the following properties:
\begin{equation}\label{prop1bus}
 b_\xi^g(x_0, x) =  b_\xi^g(x_0, x_1)+  b_\xi^g(x_1, x) \quad \quad (\text{cocycle property})
 \end{equation}
 and
 \begin{equation}\label{prop2bus}
 b_{\gamma(\xi)}^g(\gamma(x_0), \gamma(x) )=  b_\xi^g(x_0, x) \quad \quad (\Gamma \;\text{- equivariance})
 \end{equation}
for all $\gamma \in \Gamma$.
We introduce the gradient of the Busemann function $B^g(x, \xi) : = \nabla_{x} b_\xi^g(x_0,x)$ which is independent of $x_0$ by property \ref{prop1bus}. Also observe that $B^g(x,\xi) \in S_g \wt M$ by the very definition \eqref{equation:definition-busemann}. Here, $S_g \wt M$ is the unit tangent bundle on the universal cover and $\pi : S_g \wt M \rightarrow \wt M$ denotes the projection. Given $z=(x,v) \in S_g \til M $, we introduce $c_g(z,t) := \pi(\varphi_t^g(x,v))$, where $(\varphi_t^g)_{t \in \R}$ is the (lift of) the geodesic flow on $\wt M$. We set $z^g_\pm = c_g(z,\pm \infty)  \in \partial_\infty \wt M$. 

For $\xi = z^g_+$ the submanifolds
$
W^{ss}(z) = \{(x,- B^g(x, \xi)) \in S_g \wt M \mid b_\xi^g(x_0,x) = b_\xi^g(x_0,\pi z) \} $ and 
$W^{uu}(z) = \{ (x,B^g(x, \xi)) \in S_g \wt M \mid b_\xi^g(x_0,x) = b_\xi^g(x_0,\pi z) \}$
are the lifts of the leafs of strong stable and unstable foliations through $z \in  S_g \wt M$. Since the leafs are smooth and the foliations
are H\"older continuous, the Busemann functions $(x,\xi) \mapsto b_\xi^g(x_0, x)$ are smooth with respect to $x$ and  H\"older continuous with respect to $\xi$. The following Lemma was proved in \cite{Schapira-Tapie-18} (see also \cite{Gromov-00}) in negative curvature.

\begin{lemma} 
Let $M =  \til M/ \Gamma$ be a closed manifold, $g_1, g_2$ two Riemannian metrics with Anosov geodesic flow.
Consider the map $\psi_{g_1, g_2}: S_{g_1}\wt M  \to  S_{g_2}\wt M $ defined by $\psi_{g_1, g_2}(z)  =w$ where 
$w \in  S^{g_2}\wt M $ is the unique vector with $w^{g_2}_+ = z^{g_1}_+$ and  $w^{g_2}_-= z^{g_1}_-$ and 
$b^{g_2}_{z^{g_1}_+}( \pi(z), \pi(w)) =0$. Then $\psi_{g_1, g_2}$ is a H\"older continuous homeomorphism with
$$
\wt \varphi^{g_2}_{\tau(z,t)}\psi_{g_1, g_2}(z) = \psi_{g_1, g_2}(\wt \varphi^{g_1}_t(z))
$$
where 
$$
\tau(z,t) = b^{g_2}_{z^{g_1}_+} ( \pi(z), \pi(\wt \varphi^{g_1}_t(z)) ) =\int_{0}^t g_2( B^{g_2}(\pi (\wt \varphi_s^{g_1}(z)), z^{g_1}_+) , \wt \varphi_s^{g_1}(z)) ds
$$
 for all $z \in S^{g_1}\widetilde M $. Furthermore, for all $\gamma \in \Gamma$
we have 
$$
\gamma_{\ast}\psi_{g_1, g_2}(z) =\psi_{g_1, g_2} (\gamma_{\ast}z)
$$ and 
${\tau(\gamma_{\ast}z,t)} = \tau(z,t)$ and therefore  $\psi_{g_1, g_2}$ descends to a conjugacy between the geodesic flows on the quotients.
\end{lemma}

\begin{proof}
We show first that for each $(z,t) \in  S_{g_1}\wt M \times \R$ we have 
\[
\wt \varphi^{g_2}_{\tau(z,t)}\psi_{g_1, g_2}(z) = \psi_{g_1, g_2}(\wt \varphi^{g_1}_t(z))
\]
where $\tau(z,t) = b^{g_2}_{z^{g_1}_+} ( \pi(z),\pi( \wt \varphi^{g_1}_t(z))$. From the cocycle property \ref{prop1bus} of the Busemann function we obtain
\begin{align*}
& b^{g_2}_{z^{g_1}_+}( \pi (\wt \varphi^{g_1}_t(z)), \pi (\wt \varphi^{g_2}_{\tau(z,t)}\psi_{g_1, g_2}(z))) \\
& = b^{g_2}_{z^{g_1}_+}( \pi (\wt \varphi^{g_1}_t(z)), \pi(\psi_{g_1, g_2}(z)) )     +   b^{g_2}_{z^{g_1}_+}  (\pi (\psi_{g_1, g_2}(z)), \pi (\wt \varphi^{g_2}_{\tau(z,t)}\psi_{g_1, g_2}(z))) \\
& =b^{g_2}_{z^{g_1}_+}(  \pi  (\wt\varphi^{g_1}_t(z)),\pi (\psi_{g_1, g_2}(z))) +\tau(z,t) \\
& = b^{g_2}_{z^{g_1}_+}( \pi(\wt  \varphi^{g_1}_t(z)), \pi(z)) +  b^{g_2}_{z^{g_1}_+}( \pi(z),   \pi(\psi_{g_1, g_2}(z))  ) +\tau(z,t)\\
& =- b^{g_2}_{z^{g_1}_+}(  \pi(z), \pi(\wt  \varphi^{g_1}_t(z))  +\tau(z,t) =0.
\end{align*}
By the definition of $\psi_{g_1, g_2}$ this yields $ \wt \varphi^{g_2}_{\tau(z,t)}\psi_{g_1, g_2}(z) = \psi_{g_1, g_2}(\wt \varphi^{g_1}_t(z))$.
The regularity of the Busemann function shows that the conjugacy is H\"older continuous.
The remaining assertions follow from the  $\Gamma$-equivariance \ref{prop2bus} of the Busemann function.
\end{proof}

\section{Anosov Stability}

The proof of the Anosov stability theorem is written down using the implicit function 
theorem in \cite{DeLaLlave-Marco-Moryon-86} in the $C^0$ category, the extension to the H\"older setting (with the same method) is written down in \cite{Katoketal}. We need the continuity with respect to the two metrics here, the proof of \cite{DeLaLlave-Marco-Moryon-86,Katoketal} indeed shows this, as we explain below. Let $\nu\in(0,1)$, then if $X$ is a $C^k$ vector field for $k\geq 4$ with flow $\varphi_t^X$, we will denote $C_{X}^\nu(\mc{M},\mc{M})$ the space of $C^\nu$ maps $\psi$ on a closed manifold $\mc{M}$ so that $d\psi.X:=\pl_{t}(\psi\circ \varphi^X_t)|_{t=0}$ exists and belongs to $C^\nu(\mc{M};T\mc{M})$. This is a Banach manifold \cite[Proposition 2.2.]{Katoketal}.

\begin{proposition}\label{stabilityth}
Let $g_0$ be a smooth metric, and assume that $X_{g_0}$ its geodesic vector field on $\mc{M}:=S_{g_0}M$ is Anosov. We view all geodesic vector field $X_g$ associated to $g$ near $g_0$ as vector fields on $\mc{M}$ (by pulling back from $S_gM$ to $S_{g_0}M$). For $k\geq 4$, there is $\nu>0$ and two open neighborhoods $\mc{U}_0\subset \mc{U}$ of $X_{g_0}$ in $C^{k+1}(\mc{M};T\mc{M})$ such that for each $Y\in\mc{U}$, and each $g\in C^{k+2}(M;S^2T^*M)$ so that $X_g\in\mc{U}_0$, there is a homeomorphism $\psi_{g,Y}\in C_{X_{g}}^{\nu}(\mc{M},\mc{M})$ and $a_{g,Y}\in C^{\nu}(\mc{M},\R^+)$
such that 
\[\forall x\in \mc{M},\,\,  d\psi_{g,Y}(x)X_{g}(x)=a_{g,Y}(x) Y(\psi_{g,Y}(x))\]
where $X_g$ is the geodesic vector field of $g$. Moreover $Y\in \mc{U}\mapsto a_{g,Y}\in C^\nu(\mc{M},\R^+)$ and 
$Y\mapsto \psi_{g,Y}\in C^\nu_{X_g}(\mc{M},\mc{M})$ are $C^{k}$, 
and each derivative of order $\ell\leq k$ with respect to $Y$ is continuous with respect to $(g,Y)$ with values in $C^{\nu}$.
\end{proposition}
\begin{proof} The proof is essentially contained in \cite[Proposition 2.2.]{Katoketal}, except for the statement about the continuity with respect to $X_g$.
Consider for $\nu\in(0,1)$ the map 
\[
F^{X_g}:  C^{k+1}(\mc{M};T\mc{M}) \times C^{\nu}_{X_g}(\mc{M},\mc{M})\times C^\nu(\mc{M})\to C^{k+1}(\mc{M};T\mc{M}) \times C^\nu(\mc{M},T\mc{M}) =: E
\]
defined by 
\[
F^{X_g}(Y,u,\gamma) := (Y, \gamma du.X_g-Y\circ u).
\]
This is a $C^k$ map between Banach manifolds. The differential of $F^{X_g}$ at $(X_g,{\rm Id},1)$ is given (as in \cite{Katoketal}) by 
\begin{equation}\label{calculdFX} 
dF^{X_g}_{(X_g,{\rm Id},1)}(Y,V,\gamma)=(Y, -Y +\mc{L}_{X_g}V+\gamma X_g).
\end{equation}
where $V\in C_{X_g}^{\nu}(\mc{M};T\mc{M}):=\{ V\in C^\nu(\mc{M};T\mc{M})\,|\, \mc{L}_{X_g}V\in C^{\nu}\}$.
Let $\alpha_g$ be the contact form of $g$, so that $\ker \alpha_g=E_u(g)\oplus E_s(g)$ is the smooth bundle of stable or unstable vectors for $g$. 
By \cite[Proposition 2.2. and Lemma 2.3]{Katoketal}, the operator $\mc{L}_{X_g}: V\mapsto \mc{L}_{X_g}V$ is invertible from $C^{\nu}_{X_g}(\mc{M};\ker \alpha_g)\to C^\nu(\mc{M};\ker \alpha_g)$ for some $\nu$ depending on the maximal/minimal expansion rates of the flow $\varphi_t^{X_g}$. The inverse is given by:
\[
\mc{L}_{X_g}^{-1}: V=V_u+V_s\mapsto \mc{L}_{X_g}^{-1}V=-\int_0^{+\infty} d\varphi_{-t}^{X_g}V_u\circ \varphi_t^{X_g} \, dt+\int_0^{+\infty} d\varphi_{t}^{X_g}V_s\circ \varphi_{-t}^{X_g} \, dt,
\]
where the integrals converge due to the contraction of the differential: for all $t\geq 0$
\begin{equation}\label{contractionexp}
C^{-1}e^{-\la_+ t}\leq \|d\varphi_t^{X_g}|_{E_s(g)}\|\leq Ce^{-\la_- t}, \quad 
C^{-1}e^{-\la_+ t}\leq \|d\varphi_{-t}^{X_g}|_{E_u(g)}\|\leq Ce^{-\la_- t}.
\end{equation} 
This operator maps continuously $C^\nu(\mc{M};\ker \alpha_g)$ to $C^{\nu}_{X_g}(\mc{M};\ker \alpha_g)$ if $\nu>0$ is small enough depending on $\la_+$ and $\|\varphi_T^{X_g}\|_{C^2}$ for $T>0$ large (see below). Moreover, by continuity of the bundles $E_s(g),E_u(g)$ with respect to $g$ (\cite[Theorem 3.2]{HP70}), for $g$ close enough to $g_0$ in $C^{k+5}$, $E_u(g)$ and $E_s(g)$ are contained in a small conic neighborhood of $E_u(g_0)$ and $E_s(g_0)$ respectively, and the contraction exponents $\la_\pm(g)$ are also close to $\la_\pm(g_0)$ (see for ex \cite[Lemma 3]{Bonthonneau-18}), so this will give the boundedness of $\mc{L}_{X_g}^{-1}$ in $C^\nu$  for some fixed $\nu>0$ for $g$ close enough to $g_0$ in $C^{k+5}$.
From the expression of $\mc{L}_{X_g}^{-1}$, and the fact that \eqref{contractionexp}
holds uniformly for $g$ close to $g_0$ for some $0<\la_-<\la_+$ (and similarly on $E_u(g)$), we claim that, if $\pi_{g}:T\mc{M}\to \ker \alpha_g$ is the projection given by $\pi_{g}(V)=V-\alpha_{g}(V)X_g$, then 
\[
\mc{L}_{X_g}^{-1}\pi_{g}: C^{\nu}(\mc{M};T\mc{M})\to C^\nu(\mc{M};T\mc{M})
\] 
is continuous with respect to $g$ (in $C^{k+5}$) for $\nu>0$ small enough. To prove this, we rewrite $\mc{L}_{X_g}^{-1}\pi_g$ as 
\begin{equation}\label{formulaLXinv}
\mc{L}_X^{-1}\pi_g=\int_0^{\infty}e^{-t\mc{L}_{X_g}}\pi^s_g dt-\int_0^\infty e^{t\mc{L}_{X_g}}\pi^u_g dt\end{equation}
where $\pi_g^u: C^\nu(\mc{M};T\mc{M})\to C^\nu(\mc{M};T\mc{M})$ is the projection on $E_u$ parallel to $E_s$ and $\pi_g^s: C^\nu(\mc{M};T\mc{M})\to C^\nu(\mc{M};T\mc{M})$ is the projection on $E_s$ parallel to $E_u$, and $e^{t\mc{L}_{X_g}} Y := d\varphi_{-t}^{X_g} Y \circ \varphi_t^{X_g}$ is the propagator. Here $\nu$ is chosen small so that $E_u$ and $E_s$ are $C^\nu$ bundles (see \cite{HP70}), and by \cite{Contreras-92} the maps $g\mapsto \pi_g^u$ and $g\mapsto \pi_g^s$ are continuous (actually $C^r$ for some $r$ depending on the smoothness of $g$). Next, there is $C>0$ and $\Lambda>0$ such that for all $t$, $\|\varphi^{X_g}_t\|_{C^2}\leq Ce^{\Lambda |t|}$ for all $g$ near $g_0$ in $C^{k+5}$,  which implies for all $V\in C^1(\mc{M};T\mc{M})$,
\[  \forall t,\quad \|e^{t\mc{L}_{X_g}}V\|_{C^1}\leq Ce^{2\Lambda |t|}\|V\|_{C^1}, \quad \|e^{t\mc{L}_{X_g}}V\|_{C^0}\leq C e^{\Lambda |t|} \|V\|_{C^0}\]
thus if $\nu_0\in(0,1)$ is such that $E_u\in C^{\nu_0}$, we have by interpolation that 
$\|e^{t\mc{L}_{X_g}}\|_{\mc{L}(C^{\nu})}\leq Ce^{(1+\nu_0) \Lambda |t|}$ for each $\nu\leq \nu_0$. 
Since $\|e^{t\mc{L}_{X_g}}\pi_g^u\|_{\mc{L}(C^0)}+\|e^{-t\mc{L}_{X_g}}\pi_g^s\|_{\mc{L}(C^0)}\leq Ce^{-\la_- t}$ for all $t\geq 0$, we obtain by interpolating $C^\nu$ between the spaces $C^0$ and $C^{\nu_0}$ with $\nu = \theta \times 0 + (1-\theta) \times \nu_0$ (for $\theta \in (0,1)$) that for all $t\geq 0$
\[\|e^{t\mc{L}_{X_g}}\pi_g^u\|_{\mc{L}(C^\nu)}+\|e^{-t\mc{L}_{X_g}}\pi_g^s\|_{\mc{L}(C^\nu)}\leq 
Ce^{\left(-\theta \lambda_- +(1-\theta)(1+\nu_0)\Lambda\right)t}\]
We can now fix $\nu$ small enough (i.e. $\theta$ close enough to $1$) to guarantee
\[
-\theta \lambda_- +(1-\theta)(1+\nu_0)\Lambda<0,
\]
which implies that \eqref{formulaLXinv} is uniformly converging with respect to $g$ near $g_0$ in $C^{k+5}$. Since $g\mapsto e^{t\mc{L}_{X_g}}\pi_g^u$ and $g\mapsto e^{-t\mc{L}_{X_g}}\pi_g^s$ are continuous for each $t\geq 0$, we can apply the Lebesgue Theorem to deduce the continuity of $g\mapsto 
\mc{L}_{X_g}^{-1}\pi_g\in \mc{L}(C^\nu)$ for $\nu>0$ small enough. 

Next, we consider the map $\til{F}^{X_g}: E\rightarrow E$ defined by
\[  
\begin{split}
 \til{F}^{X_g}(Y,V) &:= F^{X_g}(Y,\exp_{g_0}(\mc{L}_{X_g}^{-1}\pi_g(Y+V)),\alpha_g(Y+V)) \\
 & = (Y,\alpha_g(Y+V)d(\exp_{g_0}(\mc{L}_{X_g}^{-1}\pi_g(Y+V))).X_g -Y\circ \exp_{g_0}(\mc{L}_{X_g}^{-1}\pi_g(Y+V)))
 \end{split}
 \]
where we recall that $E = C^{k+1}(\mc{M};T\mc{M}) \times C^\nu(\mc{M},T\mc{M}) $ and $\exp_{g_0}$ is the exponential map of $g_0$. This map satisfies $\til{F}^{X_g}(X_g,0)=(X_g,0)$. We want to apply the inverse function theorem to find a pre-image to each $(Y,0)$ close to $(X_g,0)$.
As in \cite[Proposition 2.2]{Katoketal} (see also \cite[Appendix A]{DeLaLlave-Marco-Moryon-86}), the map $\til{F}^{X_g}$ is $C^{k}$, and moreover it depends continuously on $g\in C^{k+5}(M;S^2 T^*M)$, with all its derivatives of order $\ell\leq k$ being also continuous with respect to $g$, due to the continuity of $g\mapsto \mc{L}_{X_g}^{-1}\pi_g$ as a map $C^{k+5}(M;S^2 T^*M)\to \mc{L}(C^\nu(\mc{M};T\mc{M}))$. Now, we have 
\[
d\til{F}^{X_g}(X_g,0)={\rm Id},
\]
by using \eqref{calculdFX} and $\pi_g(X_g)=0$. In particular there is $\eps>0$ such that if $\|g-g_0\|_{C^{k+5}}<\eps$, $\|Y-X_g\|_{C^{k+1}}<\eps$ and 
$\|V\|_{C^\nu}<\eps$, then 
\[\|d\til{F}^{X_g}_{(Y,V)}-{\rm Id}\|_{\mc{L}(C^\nu(\mc{M};T\mc{M}))}<1/4.\]
For each $Y$ close to $X_g$, we can then use the fixed point theorem (like in the proof of the inverse function theorem) to 
the map $(Z,V)\in E\mapsto (Z+Y,V)-\til{F}^{X_g}(Z,V)$ and obtain that there is a unique $(Y,V(Y))$ such that $(Y,0)=\til{F}^{X_g}(Y,V(Y))$, and $V(Y)\in C^{\nu}(\mc{M};T\mc{M})$ depends in a $C^{k}$ fashion on $Y$ and is continuous with respect to $g$. Moreover, the usual argument in the inverse function theorem used to prove the $C^k$ property of $Y\mapsto V(Y)$ also shows that the derivatives of order $\ell\leq k$ are continuous with respect to $(X_g,Y)$, by using the continuity of $\til{F}^{X_g}$ and its derivatives with respect to $g$.  This shows that for each $Y$ close to $X_g$ in $C^{k+1}$ norm and $g$ close to $g_0$ in $C^{k+5}$ norm, there is a   
\[ u=\exp_{g_0}(\mc{L}_{X_g}^{-1}\pi_g(Y+V))\in C^{\nu}_{X_g}(\mc{M},\mc{M}) ,\quad \gamma=\alpha_g(Y+V)\in C^{\nu}(\mc{M})\]
so that $\gamma du.X_g=Y\circ u$, with
\[
C^{k+1}(\mc{M};T\mc{M}) \ni Y \mapsto (u,\gamma)\in C^\nu(\mc{M}\times \mc{M})\times C^\nu(\mc{M}),
\]
being $C^{k}$ and all the derivatives of order $\ell\leq k$ are continuous in $(g,Y)$ (with values in $C^\nu(\mc{M},\mc{M})\times C^\nu(\mc{M})$).
\end{proof}

\bibliographystyle{alpha}
\bibliography{biblio}

\begin{thebibliography}{dlLMM86}

\bibitem[ADSK16]{Avila}
A.~Avila, J.~De~Simoi, and V.~Kaloshin.
\newblock An integrable deformation of an ellipse of small eccentricity is an
  ellipse.
\newblock {\em Annals of Mathematics}, 184:527--558, 2016.

\bibitem[BCG95]{Besson-Courtois-Gallot-95}
G.~Besson, G.~Courtois, and S.~Gallot.
\newblock Entropies et rigidit\'{e}s des espaces localement sym\'{e}triques de
  courbure strictement n\'{e}gative.
\newblock {\em Geom. Funct. Anal.}, 5(5):731--799, 1995.

\bibitem[BCLS15]{Bridgeman-Canary-Labourie-Sambarino-15}
M.~Bridgeman, R.~Canary, F.~Labourie, and A.~Sambarino.
\newblock The pressure metric for {A}nosov representations.
\newblock {\em Geom. Funct. Anal.}, 25(4):1089--1179, 2015.

\bibitem[BCS18]{Bridgeman-Canary-Sambarino-18}
M.~Bridgeman, R.~Canary, and A.~Sambarino.
\newblock An introduction to pressure metrics for higher {T}eichm\"{u}ller
  spaces.
\newblock {\em Ergodic Theory Dynam. Systems}, 38(6):2001--2035, 2018.

\bibitem[BH99]{Bridson-Haefliger-99}
M.~R. Bridson and A.~Haefliger.
\newblock {\em Metric spaces of non-positive curvature}, volume 319 of {\em
  Grundlehren der Mathematischen Wissenschaften [Fundamental Principles of
  Mathematical Sciences]}.
\newblock Springer-Verlag, Berlin, 1999.

\bibitem[BK85]{Burns-Katok-85}
K.~Burns and A.~Katok.
\newblock Manifolds with nonpositive curvature.
\newblock {\em Ergodic Theory Dynam. Systems}, 5(2):307--317, 1985.

\bibitem[BR75]{BR75}
R.~Bowen and D.~Ruelle.
\newblock The ergodic theory of axioma flows.
\newblock {\em Inventiones Mathematicae}, 29(3):181--202, 1975.

\bibitem[CF90]{Croke-Fathi}
C.B. Croke and A.~Fathi.
\newblock An inequality between energy and intersection.
\newblock {\em Bull. London Math. Soc.}, 22:489--494, 1990.

\bibitem[Con92]{Contreras-92}
G.~Contreras.
\newblock Regularity of topological and metric entropy of hyperbolic flows.
\newblock {\em Math. Z.}, 210(1):97--111, 1992.

\bibitem[Cro90]{Croke-90}
C.~B. Croke.
\newblock Rigidity for surfaces of nonpositive curvature.
\newblock {\em Comment. Math. Helv.}, 65(1):150--169, 1990.

\bibitem[CS98]{Croke-Sharafutdinov-98}
C.~B. Croke and V.~A. Sharafutdinov.
\newblock Spectral rigidity of a compact negatively curved manifold.
\newblock {\em Topology}, 37(6):1265--1273, 1998.

\bibitem[DGRS20]{DGRS}
N.V. Dang, C.~Guillarmou, G.~Rivi\`ere, and S.~Shen.
\newblock Fried conjecture in small dimensions.
\newblock {\em Inventiones Math.}, 220:525--579, 2020.

\bibitem[dlLMM86]{DeLaLlave-Marco-Moryon-86}
R.~de~la Llave, J.~M. Marco, and R.~Moriy\'{o}n.
\newblock Canonical perturbation theory of {A}nosov systems and regularity
  results for the {L}iv\v{s}ic cohomology equation.
\newblock {\em Ann. of Math. (2)}, 123(3):537--611, 1986.

\bibitem[DS03]{Dairbekov-Saharfutdinov-Anosov}
N.~S. Dairbekov and V.~A. Sharafutdinov.
\newblock Some problems of integral geometry on anosov manifolds.
\newblock {\em {Ergodic Theory Dyn. Syst.}}, 23:59--74, 2003.

\bibitem[DS10]{Dairbekov-Sharafutdinov-10}
N.~S. Dairbekov and V.~A. Sharafutdinov.
\newblock Conformal {K}illing symmetric tensor fields on {R}iemannian
  manifolds.
\newblock {\em Mat. Tr.}, 13(1):85--145, 2010.

\bibitem[DSKW17]{DSKW}
J.~De~Simoi, V.~Kaloshin, and Q.~Wei.
\newblock Dynamical spectral rigidity among $2$-symmetric strictly convex
  domains close to a circle.
\newblock {\em Annals of {M}athematics}, 186:277--314, 2017.

\bibitem[DZ16]{Dyatlov-Zworski-16}
S.~Dyatlov and M.~Zworski.
\newblock Dynamical zeta functions for {A}nosov flows via microlocal analysis.
\newblock {\em Ann. Sci. \'{E}c. Norm. Sup\'{e}r. (4)}, 49(3):543--577, 2016.

\bibitem[DZ19]{Dyatlov-Zworski-book-resonances}
Semyon Dyatlov and Maciej Zworski.
\newblock {\em Mathematical Theory of Scattering Resonances}, volume 200.
\newblock American Mathematical Society, 2019.

\bibitem[Ebi68]{Ebin-68}
D.~G. Ebin.
\newblock On the space of {R}iemannian metrics.
\newblock {\em Bull. Amer. Math. Soc.}, 74:1001--1003, 1968.

\bibitem[Ebi70]{Ebin1970}
D.~G. Ebin.
\newblock The space of {R}iemannian metrics.
\newblock In {\em Proc. Symp. Pure Math.}, volume~15, pages 11--40. Amer. Math.
  Soc, 1970.

\bibitem[FF93]{Fathi-Flaminio}
A.~Fathi and L.~Flaminio.
\newblock Infinitesimal conjugacies and {W}eil-{P}etersson metric.
\newblock {\em Annales de l'Institut Fourier}, 43(1):279--299, 1993.

\bibitem[Fla95]{Flaminio}
L.~Flaminio.
\newblock Local entropy rigidity for hyperbolic manifolds.
\newblock {\em Communications in Analysis and Geometry}, 3(4):555--596, 1995.

\bibitem[Fra66]{Frankel}
T.~Frankel.
\newblock On theorems of {H}urwitz and {B}ochner.
\newblock {\em J. Math. Mech.}, 15:373--377, 1966.

\bibitem[FRS08]{Faure-Roy-Sjostrand}
F.~Faure, N.~Roy, and J.~Sj\"ostrand.
\newblock Semi-classical approach for {A}nosov diffeomorphisms and {R}uelle
  resonances.
\newblock {\em Open J. Math}, 1:35--81, 2008.

\bibitem[FS11]{Faure-Sjostrand-11}
F.~Faure and J.~Sj\"{o}strand.
\newblock Upper bound on the density of {R}uelle resonances for {A}nosov flows.
\newblock {\em Comm. Math. Phys.}, 308(2):325--364, 2011.

\bibitem[GKLM18]{GKLM}
P.~Giulietti, B.~Kloeckner, A.O Lopes, and D.~Marcon.
\newblock {The calculus of thermodynamical formalism}.
\newblock {\em Journ of the EMS}, 20(10):2357--2412, 2018.

\bibitem[GL]{Gouezel-Lefeuvre-19}
S.~{Go{\"u}ezel} and T.~{Lefeuvre}.
\newblock {Classical and microlocal analysis of the X-ray transform on Anosov
  manifolds}.
\newblock {\em Analysis and PDE, to appear.}

\bibitem[GL19]{Guillarmou-Lefeuvre-18}
C.~{Guillarmou} and T.~{Lefeuvre}.
\newblock {The marked length spectrum of Anosov manifolds}.
\newblock {\em Annals of Mathematics}, 190(1):321--344, 2019.

\bibitem[Gro00]{Gromov-00}
M.~Gromov.
\newblock Three remarks on geodesic dynamics and fundamental group.
\newblock {\em Enseign. Math. (2)}, 46(3-4):391--402, 2000.

\bibitem[GS94]{GS94}
A.~Grigis and J.~Sj{\"o}strand.
\newblock {\em {Microlocal analysis for differential operators: an
  introduction}}.
\newblock Cambridge Univ Pr, 1994.

\bibitem[{Gue}20]{Bonthonneau-18}
Y.~{Guedes Bonthonneau}.
\newblock {Flow-independent Anisotropic space, and perturbation of resonances}.
\newblock {\em {R}evista de la {U}ni{\'o}n {M}atem{\'a}tica {A}rgentina},
  61(1):63--72., 2020.

\bibitem[Gui17]{Guillarmou-17-1}
C.~Guillarmou.
\newblock Invariant distributions and {X}-ray transform for {A}nosov flows.
\newblock {\em J. Differential Geom.}, 105(2):177--208, 2017.

\bibitem[Ham82]{Hamilton}
R.S. Hamilton.
\newblock The inverse function theorem of {N}ash and {M}oser.
\newblock {\em Bulletin of the A.M.S.}, 7(1):65--222, June 1982.

\bibitem[Ham99]{Hamenstadt-99}
U.~Hamenst\"{a}dt.
\newblock Cocycles, symplectic structures and intersection.
\newblock {\em Geom. Funct. Anal.}, 9(1):90--140, 1999.

\bibitem[HF19]{Fisher-Hasselblatt}
B.~Hasselblatt and T.~Fisher.
\newblock {\em Hyperbolic flows}.
\newblock Zurich Lectures in Advanced Mathematics. European Math Society, 2019.

\bibitem[HK90]{Hurder-Katok-90}
S.~Hurder and A.~Katok.
\newblock Differentiability, rigidity and {G}odbillon-{V}ey classes for
  {A}nosov flows.
\newblock {\em Inst. Hautes \'{E}tudes Sci. Publ. Math.}, (72):5--61 (1991),
  1990.

\bibitem[H{\"o}r03]{Hoe03}
L.~H{\"o}rmander.
\newblock {\em {The analysis of linear partial differential operators. I.
  Distribution theory and Fourier analysis. Reprint of the second (1990)
  edition}}.
\newblock Springer, Berlin, 2003.

\bibitem[HP68]{HP70}
M.W. Hirsch and C.C. Pugh.
\newblock Stable manifolds and hyperbolic sets.
\newblock In R.I.~1970. Amer. Math.~Soc., Providence, editor, {\em In Global
  Analysis (Proc. Sympos. Pure Math., Vol. XIV, Berkeley, Calif.)}, volume XIV,
  pages 133--163, 1968.

\bibitem[Kat88]{Katok-88}
A.~Katok.
\newblock Four applications of conformal equivalence to geometry and dynamics.
\newblock {\em Ergodic Theory Dynam. Systems}, 8$^*$(Charles Conley Memorial
  Issue):139--152, 1988.

\bibitem[KKPW89]{Katoketal}
A.~Katok, G.~Knieper, M.~Pollicott, and H.~Weiss.
\newblock {D}ifferentiability and analyticity of topological entropy for
  {A}nosov and geodesic flows.
\newblock {\em {Invent. Math.}}, 98:581--597, 1989.

\bibitem[KKPW90]{Katok-Knieper-Pollicott-Weiss-90}
A.~Katok, G.~Knieper, M.~Pollicott, and H.~Weiss.
\newblock Differentiability of entropy for {A}nosov and geodesic flows.
\newblock {\em Bull. Amer. Math. Soc. (N.S.)}, 22(2):285--293, 1990.

\bibitem[KKW91]{Katok-Knieper-Weiss-91}
Anatole Katok, Gerhard Knieper, and Howard Weiss.
\newblock Formulas for the derivative and critical points of topological
  entropy for {A}nosov and geodesic flows.
\newblock {\em Comm. Math. Phys.}, 138(1):19--31, 1991.

\bibitem[Kli74]{Klingenberg-74}
W.~Klingenberg.
\newblock Riemannian manifolds with geodesic flow of {A}nosov type.
\newblock {\em Ann. of Math. (2)}, 99:1--13, 1974.

\bibitem[Kni95]{Knieper-95}
G.~Knieper.
\newblock Volume growth, entropy and the geodesic stretch.
\newblock {\em Math. Res. Lett.}, 2(1):39--58, 1995.

\bibitem[Kni02]{Knieper-survey}
G.~Knieper.
\newblock {\em Hyperbolic dynamical systems, in {H}andbook of {D}ynamical
  {S}ystems Vol 1A}, pages 239--319.
\newblock Elsevier, Amsterdam, 2002.

\bibitem[Kni12]{Knieper-12}
Gerhard Knieper.
\newblock New results on noncompact harmonic manifolds.
\newblock {\em Comment. Math. Helv.}, 87(3):669--703, 2012.

\bibitem[KS90]{KatsudaSunada}
A.~Katsuda and T.~Sunada.
\newblock Closed orbits in homology classes.
\newblock {\em Publ. Math. IHES}, 71:5--32, 1990.

\bibitem[Liv04]{Liverani}
C.~Liverani.
\newblock On contact anosov flows.
\newblock {\em Annals of Mathematics}, 159(3):1275--1312, 2004.

\bibitem[LR18]{LopesRuggiero}
A.O Lopes and R.O. Ruggiero.
\newblock The infinite dimensional manifold of h\"older equilibirum
  probabilities has non-negative curvature.
\newblock arXiv 1811.07748, 2018.

\bibitem[Mel]{Melrose}
R.B. Melrose.
\newblock Manifolds with corners.
\newblock In preparation.

\bibitem[MM08]{McMullen}
C.T. Mc~Mullen.
\newblock Thermodynamics, dimension and the {W}eil-{P}etersson metric.
\newblock {\em Invent. Math.}, 1973:365--425, 2008.

\bibitem[Ota90]{Otal-90}
J-P. Otal.
\newblock Le spectre marqu\'{e} des longueurs des surfaces \`a courbure
  n\'{e}gative.
\newblock {\em Ann. of Math. (2)}, 131(1):151--162, 1990.

\bibitem[Par88]{Parry}
W.~Parry.
\newblock Equilibrium states and weighted uniform distribution of closed
  orbits.
\newblock {\em Dynamical Systems, Lecture Notes in Math., Spinger Verlag},
  1342:617--625, 1988.

\bibitem[Pat99]{Paternain-99}
G.~P. Paternain.
\newblock {\em Geodesic flows}, volume 180 of {\em Progress in Mathematics}.
\newblock Birkh\"{a}user Boston, Inc., Boston, MA, 1999.

\bibitem[Pol94]{Pollicott}
M.~Pollicott.
\newblock Derivatives of topological entropy for {A}nosov and geodesic flows.
\newblock {\em J. Differential Geometry}, 39:457--489., 1994.

\bibitem[PP90]{Parry-Pollicott-90}
W.~Parry and M.~Pollicott.
\newblock Zeta functions and the periodic orbit structure of hyperbolic
  dynamics.
\newblock {\em Ast\'{e}risque}, (187-188):268, 1990.

\bibitem[PPS15]{PPS}
F.~Paulin, M.~Pollicott, and B.~Schapira.
\newblock Equilibrium states in negative curvature,.
\newblock {\em Ast\'erisque 373, {S}oc. {M}ath. {F}rance}, 373, 2015.

\bibitem[PSU14]{Paternain-Salo-Uhlmann-14-2}
G.~P. Paternain, M.~Salo, and G.~Uhlmann.
\newblock Spectral rigidity and invariant distributions on {A}nosov surfaces.
\newblock {\em J. Differential Geom.}, 98(1):147--181, 2014.

\bibitem[Rue04]{Ruelle-04}
D.~Ruelle.
\newblock {\em Thermodynamic formalism}.
\newblock Cambridge Mathematical Library. Cambridge University Press,
  Cambridge, second edition, 2004.
\newblock The mathematical structures of equilibrium statistical mechanics.

\bibitem[Sam14]{Sambarino-14}
A.~Sambarino.
\newblock Quantitative properties of convex representations.
\newblock {\em Comment. Math. Helv.}, 89(2):443--488, 2014.

\bibitem[Sig72]{Sigmund}
K.~Sigmund.
\newblock On the space of invariant measures for hyperbolic flows.
\newblock {\em American Journal of Mathematics}, 94(1):31--37, 1972.

\bibitem[SSU05]{Sharafut-Skokan-Uhlmann}
V.~Sharafutdinov, S.~Skokan, and G.~Uhlmann.
\newblock Regularity of ghosts in tensor tomography.
\newblock {\em J. Geom. Anal.}, 15(3):499--542, 2005.

\bibitem[STar]{Schapira-Tapie-18}
B.~Schapira and S.~Tapie.
\newblock {Regularity of entropy, geodesic currents and entropy at infinity}.
\newblock {\em Annales de l'Ecole Normale Sup{\'e}rieure}, to appear.

\bibitem[SU04]{Stefanov-Uhlmann-04}
P.~Stefanov and G.~Uhlmann.
\newblock Stability estimates for the {X}-ray transform of tensor fields and
  boundary rigidity.
\newblock {\em Duke Math. J.}, 123(3):445--467, 2004.

\bibitem[Tay91]{Taylor-91}
M.~E. Taylor.
\newblock {\em Pseudodifferential operators and nonlinear {PDE}}, volume 100 of
  {\em Progress in Mathematics}.
\newblock Birkh\"{a}user Boston, Inc., Boston, MA, 1991.

\bibitem[Tay96]{Tay96}
M.E. Taylor.
\newblock {\em {Partial differential equations: basic theory}}.
\newblock Springer Verlag, 1996.

\bibitem[{Thu}98]{Thurston-98}
W.~P. {Thurston}.
\newblock {Minimal stretch maps between hyperbolic surfaces}.
\newblock {\em arXiv Mathematics e-prints}, page math/9801039, Jan 1998.

\bibitem[Tro92]{Tromba-92}
A.~J. Tromba.
\newblock {\em Teichm\"{u}ller theory in {R}iemannian geometry}.
\newblock Lectures in Mathematics ETH Z\"{u}rich. Birkh\"{a}user Verlag, Basel,
  1992.
\newblock Lecture notes prepared by Jochen Denzler.

\bibitem[Wal82]{Walters-82}
P.~Walters.
\newblock {\em An introduction to ergodic theory}, volume~79 of {\em Graduate
  Texts in Mathematics}.
\newblock Springer-Verlag, New York-Berlin, 1982.

\bibitem[Wol86]{Wolpert}
S~Wolpert.
\newblock Thurston's {R}iemannian metric for {T}eichm\"uller space.
\newblock {\em J. Diff. Geom.}, 23:143--174, 1986.

\bibitem[Zei88]{Zeidler-88}
E.~Zeidler.
\newblock {\em Nonlinear functional analysis and its applications. {IV}}.
\newblock Springer-Verlag, New York, 1988.
\newblock Applications to mathematical physics, Translated from the German and
  with a preface by Juergen Quandt.

\end{thebibliography}

\end{document}